\documentclass[reqno,a4paper,twoside]{amsart}

\usepackage[a4paper,margin=0.65in,footskip=0.25in]{geometry}

\usepackage{graphicx}
\usepackage{enumitem}

\setenumerate{label=\textnormal{(\arabic*)}}

\allowdisplaybreaks[4]

\usepackage{amsmath,amssymb,dsfont,verbatim,bm,mathrsfs,stmaryrd,mathabx}
\usepackage{mathtools}
\usepackage[raggedright]{titlesec}

\usepackage[utf8]{inputenc}
\usepackage[T1]{fontenc}

\usepackage[x11names,dvipsnames]{xcolor}
	\definecolor{egyptianblue}{rgb}{0.06, 0.2, 0.65}
	\definecolor{green(ncs)}{rgb}{0.0, 0.62, 0.42}

\usepackage{hyperref}
\hypersetup{
	colorlinks=true,
	citecolor=green(ncs),
	linkcolor= egyptianblue,
	filecolor=magenta,
	urlcolor=cyan,
	unicode=true,
	pdfpagemode=FullScreen,
}

\titleformat{\chapter}[display]
{\normalfont\huge\bfseries}{\chaptertitlename\\thechapter}{20pt}{\Huge}
\titleformat{\section}
{\normalfont\Large\bfseries\center}{\thesection}{1em}{}
\titleformat{\subsection}
{\normalfont\large\bfseries}{\thesubsection}{1em}{}
\titleformat{\subsubsection}
{\normalfont\normalsize\bfseries}{\thesubsubsection}{1em}{}
\titleformat{\paragraph}[runin]
{\normalfont\normalsize\bfseries}{\theparagraph}{1em}{}
\titleformat{\subparagraph}[runin]
{\normalfont\normalsize\bfseries}{\thesubparagraph}{1em}{}

\titlespacing*{\chapter} {0pt}{50pt}{40pt}
\titlespacing*{\section} {0pt}{3.5ex plus 1ex minus .2ex}{2.3ex plus .2ex}
\titlespacing*{\subsection} {0pt}{3.25ex plus 1ex minus .2ex}{1.5ex plus .2ex}
\titlespacing*{\subsubsection}{0pt}{3.25ex plus 1ex minus .2ex}{1.5ex plus .2ex}
\titlespacing*{\paragraph} {0pt}{3.25ex plus 1ex minus .2ex}{1em}
\titlespacing*{\subparagraph} {\parindent}{3.25ex plus 1ex minus .2ex}{1em}

\usepackage{float, subfig}
\usepackage{amsrefs}

\usepackage{titletoc}

\contentsmargin{2.55em}
\dottedcontents{section}[1.5em]{}{2em}{1pc}
\dottedcontents{subsection}[4.35em]{}{2.8em}{1pc}
\dottedcontents{subsubsection}[7.6em]{}{3.2em}{1pc}
\dottedcontents{paragraph}[10.3em]{}{3.2em}{1pc}

\usepackage[columns=3,rule=0pt]{idxlayout}
\setindexprenote{For convenience of the reader we list below almost all symbols used in this work, indicating the page in which each
of them is introduced. We only except those that are extremely standard.}

\makeindex

\newtheorem{theorem}{Theorem}[section]
\newtheorem{lemma}[theorem]{Lemma}
\newtheorem{proposition}[theorem]{Proposition}
\newtheorem{corollary}[theorem]{Corollary}

\theoremstyle{definition}
\newtheorem{definition}[theorem]{Definition}

\newtheorem{notation}[theorem]{Notation}
\newtheorem{example}[theorem]{Example}

\theoremstyle{remark}
\newtheorem{remark}[theorem]{Remark}
\newtheorem{note}[theorem]{Note}

\usepackage{tikz}
\usetikzlibrary{matrix,shapes}
\usetikzlibrary{cd}
\tikzcdset{diagrams={nodes={inner sep=2pt}}}
\usetikzlibrary{arrows}
\tikzcdset{arrow style=tikz, diagrams={>={Latex[round,length=3pt,width=2pt,inset=1.75pt]}}}

\DeclareMathOperator{\ide}{id}
\DeclareMathOperator{\Ide}{Id}
\DeclareMathOperator{\Ext}{Ext}

\DeclareMathOperator{\Harr}{Harr}
\DeclareMathOperator{\Soc}{Soc}

\DeclareMathOperator{\Z}{Z}

\DeclareMathOperator{\ima}{Im}

\DeclareMathOperator{\Ho}{H}

\DeclareMathOperator{\Hom}{Hom}

\DeclareMathOperator{\ho}{H}
\DeclareMathOperator{\M}{M}
\DeclareMathOperator{\Se}{S}
\DeclareMathOperator{\AS}{AS}

\newcommand{\xcirc}{\hspace{0.9pt}}

\newcommand{\ov}{\overline}
\newcommand{\ot}{\otimes}
\newcommand{\wh}{\widehat}
\newcommand{\wt}{\widetilde}

\numberwithin{equation}{section}

\DeclareMathAlphabet{\mathpzc}{OT1}{pzc}{m}{it}

\usepackage{mathtools}
\newtagform{brackets}{[}{]}
\usetagform{brackets}

\let\origmaketitle\maketitle
\def\maketitle{
	\begingroup
	\def\uppercasenonmath##1{} 
	\let\MakeUppercase\relax 
	\origmaketitle
	\endgroup
}

\usepackage{fancyhdr}
\usepackage[foot]{amsaddr}

\usepackage{adjustbox}

\begin{document}
	
\title{\large Cohomology of Linear Cycle Sets when the adjoint group is finite abelian}

\author[Jorge A. Guccione]{Jorge A. Guccione$^{1}$}
\address{$^{1}$ Pontificia Universidad Cat\'olica del Per\'u, Secci\'on Matem\'aticas, PUCP, Av. Universitaria 1801, San Miguel, Lima 32, Per\'u.}
\email{vander@dm.uba.ar}
	
\author[Jorge A. Guccione]{Juan J. Guccione$^{1}$}
\email{jjguccione@gmail.com}
	
\author[Christian Valqui]{Christian Valqui$^{1}$}
\email{cvalqui@pucp.edu.pe}

\thanks{Jorge A. Guccione, Juan J. Guccione and Christian Valqui were supported by PUCP-CAP 2023-PI0991}


\begin{abstract}
This paper analyzes the second cohomology group $\ho^2_{\blackdiamond,\Yleft}(H,I)$, for linear cycle sets with commutative adjoint operation, focusing on the finite abelian case. It aims to classify extensions of such structures through cohomological methods. Techniques are developed to systematically construct explicitly $2$-cocycles. Finally, some illustrative examples are explored to validate the theoretical framework.	
\end{abstract}

\maketitle
	
\tableofcontents
	
\section*{Introduction}	

Motivated by the relevance of the Yang-Baxter equation (or equivalently the braid equation), in~\cite{Dr}, Drinfeld proposed the study of the set-theoretic solutions, which are related with many important mathematical structures such as affine torsors, solvable groups, Biberbach groups and groups of I-type, Artin-Schelter regular rings, Garside structures, biracks, Hopf algebras, left symmetric algebras, etcetera; see for example~\cites{CR,CJO,CJR,Ch,De1,DDM, GI2,GI4,GIVB,JO,S}.

An important class of set-theoretical solutions of the braid equation are the non-degenerate involutive solutions. The study of these solutions led to the introduction of the linear cycle set structures by Rump in~\cite{R2}. This notion is equivalent to the notion of braces and to the notion of bijective $1$-cocycles (see~\cites{CJO1,ESS,LV,R2,R3}). The study of extensions of linear cycle sets was done in the following papers~\cites{B,LV,GGV1}.

In the article~\cite{GGV1} it was proved that, in order to classify the classes of extensions of a lineal cycle set $H$ by~a~trivial cycle set $I$, we must determine all the maps $\blackdiamond\colon H\times I\to I$ and $\Yleft \colon I\times H\to I$ satisfying suitable conditions, and then, for each one of this pair of maps, to compute the second cohomology group $\ho_{\blackdiamond,\Yleft}^2(H,I)$, of a filtrated cochain complex ${\mathcal{C}}_{\blackdiamond,\Yleft}(H,I)$. In the same paper it was proved that the classification of the extensions $(\iota,B,\pi)$, in which $\iota(H)$ is included in the socle of $B$, is obtained considering the cases in which $\Yleft = 0$.

The aim of this paper is to obtain a description of $\ho_{\blackdiamond,\Yleft}^2(H,I)$, when in the associated brace of $H$, the operation~$\times$ (see~\eqref{equivalencia linear cycle set braza}) is commutative, i.e., the adjoint group is abelian.

The paper is organized as follows: in Section~\ref{section:Preliminaries} we make a brief review of the notions of the braces, linear cycle sets, extensions of linear cycle sets and cohomology of linear cycle sets. In Section~\ref{section 2}, we obtain the main results, namely, Theorems~\ref{teorema central} and \ref{cobordes Bne0}, and Corollary~\ref{calculo de la homo B'}, where we prove that the cohomology $\ho_{\blackdiamond,\Yleft}^2(H,I)$ is isomorphic to a subquotient of $I^{(s+1)n}$, which is explicitly computable in many cases. Here $n$ is the rank of the finite abelian ad\-ditive group of $H$ and $s$ is the rank of the adjoint group, which we assume to be also abelian. Moreover, in Subsection~\ref{subsection 2.3}, we provide a method to construct the $2$-cocycles in the complex ${\mathcal{C}}_{\blackdiamond,\Yleft}(H,I)$. By~\cite{GGV1}*{Corollary~7.10}, this allows us to construct all the corresponding extensions. In Section~\ref{section 3} we consider the case in which the linear cycle set~$H$ is trivial and we give several examples. \textcolor{red}{Finally, as an application, in Section~\ref{section 4} we obtain the extensions of a trivial linear cycle set, whose underlying group is a cyclic $p$-group, by a finite cycle $p$-group $I$, such that $I$ is included in socle of the extension. This constitutes the main result of~\cite{GGV2}, now with an improved presentation and streamlined proofs.}

\paragraph{Remark} \textcolor{red}{There have been several efforts to classify linear cycle sets (or, equivalently, braces) of fixed sizes; see, for instance, \cite{B2} for size~$p^3$, \cite{Di} for size~$p^2q$, \cite{Do} for size $p^4$, and~\cite{Mu} for size~$p^5$. When a non-trivial linear cycle set has a non-zero trivial ideal~$I$, it can be described as an extension of a linear cycle set~$H$, by~$I$. Given a linear cycle set~$H$ of cardinal~$n$, and an abelian group~$I$ of order~$m$, one obtains linear cycle sets of cardinal~$nm$, by determining the extension classes of~$H$, by~$I$. This requires identifying all maps $\blackdiamond\colon H \times I \to I$ and $\Yleft\colon I \times H \to I$ satisfying conditions~\eqref{eqq 1}--\eqref{eqq 3}, and then computing the associated second cohomology group~$\ho_{\blackdiamond,\Yleft}^2(H,I)$. The most challenging aspect of this process is the cohomology computation, which we  carry out in this paper (see Corollary~\ref{calculo de la homo B'}). For given~$H$ and~$I$, determining the admissible actions~$\blackdiamond$ and~$\Yleft$, can be particularly intricate. We will employ the method developed in Section~\ref{section 2}, to classify extensions of linear cycle sets by trivial ideals, in a series of forthcoming papers, focusing on specific families of linear cycle sets.}

\paragraph{Precedence of operations} The operations precedence in this paper is the following: the operators with the highest precedence are the unary operations $a\mapsto a^{\times n}$ (where $n\in \mathds{Z}$) and the binary operation $(a,b)\mapsto {}^ab$; then come the operations $\cdot$ and $\blackdiamond$, that have equal precedence; then the multiplication $(a,b)\mapsto a\times b$; then the operator~$\Yleft$, and finally, the sum. Of course, as usual, this order of precedence can be modified by the use of parenthesis.

\section[Preliminaries]{Preliminaries}\label{section:Preliminaries}

In this paper we work in the category of abelian groups, all the maps are $\mathds{Z}$-linear maps, $\ot$ means $\ot_{\mathds{Z}}$ and $\Hom$ means $\Hom_{\mathds{Z}}$.

\subsection[Braces and linear cycle sets]{Braces and linear cycle sets}\label{Braces and linear cycle sets}
A {\em left brace} is an additive abelian group $H$ endowed with an additional group operation $(h,l)\mapsto h\times l$, such that
\begin{equation}\label{compatibilidad braza}
h\times (l+m) = h\times l+h\times m-h\quad\text{for all $h,l,m\in H$.}
\end{equation}
The two group structures necessarily share the same neutral element, denoted by $0$. Of course, the additive inverse of $h$ is denoted by $-h$, while its multiplicative inverse is denoted by $\cramped{h^{\times-1}}$. In particular $\cramped{0^{\times-1}} = 0$.

Let $H$ and $H'$ be left braces. A map $f\colon H\to H'$ is an homomorphism of left braces if $f(h+l) = f(h)+f(l)$ and$~f(h\times l) = f(h)\times f(l)$, for all $h,l\in H$.

\smallskip

A {\em linear cycle set} is an additive abelian group $H$ endowed with a binary operation $\cdot$, having bijective left translations $h\mapsto l\cdot h$, and satisfying the conditions
\begin{equation}\label{compatibilidades cdot suma y suma cdot}
h\cdot (l+m)=h\cdot l + h\cdot m\quad\text{and}\quad (h+l)\cdot m = (h\cdot l)\cdot (h\cdot m),
\end{equation}
for all $h,l,m\in H$. It is well known that each linear cycle set is a cycle set. That is:
\begin{equation}
(h\cdot l)\cdot (h\cdot m)=(l\cdot h)\cdot (l\cdot m)\quad\text{for all $h,l,m\in H$}\label{condicion de cycle set}.
\end{equation}
Let $H$ and $H'$ be linear cycle sets. A map $f\colon H\to H'$ is an homomorphism of linear cycle sets if $f(h+l) = f(h)+f(l)$ and $f(h\cdot l)=f(h)\cdot f(l)$, for all $h,l\in H$.

\smallskip

The notions of left braces and linear cycle sets were introduced by Rump, who proved that they are equivalent via the relations
\begin{equation}
h\cdot l = \cramped{h^{\times-1}}\times (h+l)\quad\text{and}\quad h\times l = {}^hl + h,\label{equivalencia linear cycle set braza}
\end{equation}
where the map $l\mapsto {}^hl$ is the inverse of the map $l\mapsto h\cdot l$. Finally, by~\eqref{compatibilidades cdot suma y suma cdot} and~\eqref{equivalencia linear cycle set braza}, we have
\begin{equation}
(h\times l)\cdot m = (h + {}^hl)\cdot m = (h\cdot {}^hl)\cdot (h\cdot m) = l\cdot (h\cdot m),\label{accion tines con punto}
\end{equation}
which implies
\begin{equation}
{}^{h\times l} m = {}^h ({}^l m).\label{accion tines con inversa de punto}
\end{equation}

\smallskip

Each additive abelian group $H$ is a linear cycle set via $h\cdot l\coloneqq l$. These ones are the so called {\em trivial linear cycle sets}. Note that by~\eqref{equivalencia linear cycle set braza}, a linear cycle set $H$ is trivial if and only if $h\times l = h+l$, for all $h,l\in H$.

\smallskip

Let $H$ be a linear cycle set. An {\em ideal} of $H$ is a subgroup $I$ of $H$ such that $h\cdot y\in I$ and $y\cdot h - h\in I$, for all~$h\in H$ and $y\in I$. An ideal $I$ of $H$ is called a {\em central ideal} if $h\cdot y=y$ and $y\cdot h=h$, for all $h\in H$ and $y\in I$. The {\em socle} of $H$ is the ideal $\Soc(H)$, of all $y\in H$ such that $y\cdot h = h$, for all $h\in H$. The {\em center} of $H$ is the set $\Z(H)$, of all~$y\in \Soc(H)$ such that $h\cdot y = y$, for all $h\in H$. Thus, the center of $H$ is a central ideal of $H$, and each central ideal of $H$ is included in $\Z(H)$. Finally, an ideal $I$ of $H$ is {\em trivial} if $h\cdot l=l$, for all $h,l\in I$.

\begin{notation}\label{Yleft} Given an ideal $I$ of a linear cycle set $H$, for each $y\in I$ and $h\in H$, we set $y\Yleft h\coloneqq y\cdot h-h$.
\end{notation}

\subsection[Extensions of linear cycle sets]{Extensions of linear cycle sets}\label{Extensions of linear cycle sets}

Let $I$ and $H$ be linear cycle sets. An {\em extension} $(\iota,B,\pi)$, of $H$ by $I$, is a short exact sequence
\begin{equation}
    \begin{tikzpicture}
    \begin{scope}[yshift=0cm,xshift=0cm, baseline]
			\matrix(BPcomplex) [matrix of math nodes, row sep=0em, text height=1.5ex, text
			depth=0.25ex, column sep=2.5em, inner sep=0pt, minimum height=5mm, minimum width =6mm]
			{0 & I & B & H & 0,\\};
			\draw[-{latex}] (BPcomplex-1-1) -- node[above=1pt,font=\scriptsize] {} (BPcomplex-1-2);
			\draw[-{latex}] (BPcomplex-1-2) -- node[above=1pt,font=\scriptsize] {$\iota$} (BPcomplex-1-3);
			\draw[-{latex}] (BPcomplex-1-3) -- node[above=1pt,font=\scriptsize] {$\pi$} (BPcomplex-1-4);
			\draw[-{latex}] (BPcomplex-1-4) -- node[above=1pt,font=\scriptsize] {} (BPcomplex-1-5);
    \end{scope}\label{sucesion exacta corta de grupos abelianos}
    \end{tikzpicture}
\end{equation}
of additive abelian groups, in which $B$ is a linear cycle set and both $\iota$ are $\pi$ are linear cycle sets morphisms. Two extensions $(\iota,B,\pi)$ and $(\iota',B',\pi')$, of $H$ by $I$, are {\em equivalent} if there exists a linear cycle set morphism $\phi\colon B\to B'$ such that $\pi'\xcirc \phi=\pi$ and $\phi\xcirc \iota=\iota'$. As in the case of extension of groups, necessarily $\phi$ is an isomorphism.


Assume that~\eqref{sucesion exacta corta de grupos abelianos} is an extension of linear cycle sets and let $s$ be a section of $\pi$ with $s(0)=0$. For the sake of simplicity, in the rest of this section we identify $I$ with $\iota(I)$ and we write $y$ instead of $\iota(y)$. For each~$b\in B$, there exist unique $y\in I$ and $h\in H$ such that $b=y+s(h)$. Let $\blackdiamond\colon H\times I\to I$ and $\Yleft \colon I\times H\to I$ be the maps defined by
$$
h\blackdiamond y \coloneqq s(h)\cdot y\quad\text{and}\quad y\Yleft h\coloneqq y\Yleft s(h) = y\cdot s(h)-s(h),
$$
respectively. From now on we assume that $I$ is trivial. By~\cite{GGV1}*{Proposition~3.3}, this implies that $\blackdiamond$ and $\Yleft$ do not depend on~$s$. Let $\beta\colon H\times H\to I$ and $f\colon H\times H\to I$ be the unique maps such that
\begin{align}
&y+s(h) + z+s(l)= y + z+ s(h) + s(l)= y+z + \beta(h,l) + s(h+l) \label{construccion de beta}\\
\shortintertext{and}
&(y+s(h))\cdot (z+s(l))= h\blackdiamond z + f(h,l) + h\blackdiamond y \Yleft h\cdot l + s(h\cdot l), \label{formula para cdot en I times H}
\end{align}
for $y,z\in I$ and $h,l\in H$.

\begin{notation}\label{notacion ext} Fix maps $\blackdiamond\colon H\times I\to I$ and $\Yleft \colon I\times H\to I$. We let $\Ext_{\blackdiamond,\Yleft}(H;I)$ denote the set of equivalence classes of  extensions $(\iota,B,\pi)$, of $H$ by $I$, such that $s(h)\cdot y = h\blackdiamond y$ and $y\cdot s(h)-s(h) = y\Yleft h$, for every section $s$ of $\pi$ satisfying $s(0) = 0$.
\end{notation}

\subsection[Building extensions of linear cycle sets]{Building extensions of linear cycle sets}\label{Building extensions of linear cycle sets}
Let $H$ be a linear cycle set and let $I$ be an additive abelian group. Given maps $\blackdiamond\colon H\times I\to I$, $\Yleft \colon I\times H\to I$, $\beta\colon H\times H\to I $ and $f\colon H\times H\to I$, we let $I\times_{\beta,f}^{\blackdiamond,\Yleft} H$ denote $I\times H$, endowed with the binary operations $+$ and $\cdot$ defined by
\begin{align}
&(y+w_h) + (z+w_l) = y+z+\beta(h,l)+ w_{h+l}\label{formula para suma en I times H}\\
\shortintertext{and}
&(y+w_h)\cdot (z+w_l)= h\blackdiamond z + f(h,l) + h\blackdiamond y \Yleft h\cdot l + w_{h\cdot l},\label{formula para cdot en I times H'}
\end{align}
where we are writing $y$ instead of $(y,0)$ and $w_h$ instead of $(0,h)$.

\smallskip

Let $\triangleleft\colon I \times H\to I$ be the map defined by $y\triangleleft h\coloneqq  h \blackdiamond (y - y\Yleft h)$. For each $y\in I$ and $h\in H$, set $y^h\coloneqq y - y\Yleft h$. Consider $I$ endowed with the trivial linear cycle set structure and let $\iota\colon I\to I\times H$ and $\pi\colon I\times H\to H$ be the canonical maps. By the discussion at the beginning of~\cite{GGV1}*{Section~4} and~\cite{GGV1}*{Theorem~5.6 and Remark~5.12} we know that $(\iota,I\times_{\beta,f}^{\blackdiamond,\Yleft} H, \pi)$ is an extension of linear cycle sets if and only if the following conditions are fulfilled:
\begin{align}
& y \triangleleft (h\times l) = (y \triangleleft h) \triangleleft l,\quad (y + z)\triangleleft h = y\triangleleft h + z\triangleleft h,\quad y \triangleleft 0 = y,\label{eqq 1}\\
& (l\times h)\blackdiamond y  = h\blackdiamond (l \blackdiamond y),\quad h\blackdiamond (y + z) = h\blackdiamond y + h\blackdiamond z, \quad 0\blackdiamond y = y,\label{eqq 2}\\
& y^{h+l} + y = y^h + y^l,\quad 0^h = 0,\label{eqq 3}\\
& \beta(h,l)+\beta(h+l,m) = \beta(l,m)+\beta(h,l+m),\label{condicion de cociclo}\\
&\beta(h,0) = \beta(0,h) = 0,\label{normalidad}\\
&\beta(h,l) = \beta(l,h),\label{normalidad y cociclo abeliano}\\
& h\blackdiamond \beta(l,m) + f(h,l+m) = f(h,l)+f(h,m)+\beta(h\cdot l,h\cdot m)\label{eqq 4}
\shortintertext{and}
&\begin{aligned}\label{eqq5}
f(h+l,m) + (h+l)\blackdiamond \beta(h,l)\Yleft (h+l)\cdot m & = (h\cdot l)\blackdiamond f(h,m) + f(h\cdot l,h\cdot m)\\
&+ (h\cdot l)\blackdiamond f(h,l) \Yleft (h\cdot l)\cdot(h\cdot m).
\end{aligned}
\end{align}
for all $h,l,m\in H$ and $y,z\in I$. 

\smallskip

By~\cite{GGV1}*{Remark~5.14} two extensions $(\iota,I\times_{\beta,f}^{\blackdiamond,\Yleft} H,\pi)$ and $(\iota,I\times_{\beta',f'}^{\blackdiamond,\Yleft} H,\pi)$, of $H$ by $I$, are equivalent if and only if there exists a map $\varphi\colon H\to I$ satisfying
\begin{align}
&\varphi(0) = 0,\qquad \varphi(h)-\varphi(h+l)+\varphi(l)=\beta(h,l) - \beta'(h,l)\label{equivalencia a nivel de qrupos}
\shortintertext{and}
&\varphi(h\cdot l) + f(h,l) = h\blackdiamond \varphi(l) + f'(h,l) + h\blackdiamond \varphi(h)\Yleft h\cdot l,\label{tercera condicion caso I trivial}
\end{align}
for all $h,l\in H$.

\smallskip

By~\cite{GGV1}*{Remark~4.9}, every extension $(\iota,B,\pi)$, of $H$ by $I$ is equivalent to an extension of the form $(\iota,I\times_{\beta,f}^{\blackdiamond,\Yleft} H,\pi)$.

\subsection[Cohomology of linear cycle sets]{Cohomology of linear cycle sets}

Let $H$ be a linear cycle set and let $I$ be an additive abelian group, considered as a trivial linear cycle set. Fix two maps $\blackdiamond\colon H\times I\to I$ and $\Yleft \colon I\times H\to I$. Assume that conditions~\eqref{eqq 1}--\eqref{eqq 3} are fulfilled. We claim that this is equivalent to requiring that $\Yleft$ be bilinear, condition~\eqref{eqq 2} is satisfied and
\begin{equation}\label{triangulo compatible con times}
h\blackdiamond y \Yleft h\cdot l = h\blackdiamond (y\Yleft l) + h\blackdiamond (y\Yleft h) \Yleft h\cdot l\qquad\text{for all $h,l\in H$ and $y\in I$.}
\end{equation}
In fact, by~\cite{GGV1}*{Remark~5.12}, conditions~\eqref{eqq 1}--\eqref{eqq 3} hold, if and only if items~(1)–(4) of~\cite{GGV1}*{Proposition~3.2} and conditions~[5.1], [5.2], [5.5] and~[5.6] of~\cite{GGV1} hold. However, conditions~[5.1] and~[5.5] say that $\Yleft$ is bilinear, which implies items~(3) and~(4) of~\cite{GGV1}*{Proposition~3.2}. The first identity in~\cite{GGV1}*{condition~[5.2]} is automatically fulfilled, while the second one is equivalent to~\cite{GGV1}*{condition~[5.3]}, which is the first identity in~\eqref{eqq 2}. Items~(1) and~(2) of~\cite{GGV1}*{Proposition~3.2} are the second and third identities in~\eqref{eqq 2}; and finally,~\cite{GGV1}*{condition~[5.6]} is~\eqref{triangulo compatible con times}.

In~\cite{GGV1}*{Section~7} we constructed a diagram $(\wh{C}_N^{**}(H,I),\partial_{\mathrm{h}},\partial_{\mathrm{v}},D)$, illustrated in~\cite{GGV1}*{Figure~1}. By applying to this diagram the same procedure used to form the total complex of a double complex, we obtain a cochain complex $(\wh{C}_N^*(H,I),\partial+D)$, whose second cohomology group $\Ho^2_{\blackdiamond,\Yleft}(H,I)$ is canonically isomorphic to $\Ext_{\blackdiamond,\Yleft}(H;I)$ (see~\cite{GGV1}*{Proposition~7.9 and Corollary~7.10}). Our purpose is to compute the group $\Ho^2_{\blackdiamond,\Yleft}(H,I)$. For this, we will need only the following part of $(\wh{C}_N^{**}(H,I),\partial_{\mathrm{h}},\partial_{\mathrm{v}},D)$, where we write~$\wh{C}_N^{**}$ instead of $\wh{C}_N^{**}(H,I)$:
\begin{equation}\label{diagrama1}
\begin{tikzcd}[row sep=3.2em, column sep=5.2em, scale=1.4]
\wh{C}_N^{03}\\
\wh{C}_N^{02} \arrow[r, "\partial_{\mathrm{h}}^{12}"] \arrow[u, "\partial_{\mathrm{v}}^{03}"] \arrow[rrd, green(ncs), shift left=0ex, end anchor={[yshift=0.5ex]}, outer sep=-1.6pt, "D_{02}^{21}" pos=0.25] & \wh{C}_N^{12} \\
\wh{C}_N^{01} \arrow[r, "\partial_{\mathrm{h}}^{11}"{below}] \arrow[u, "\partial_{\mathrm{v}}^{02}"] \arrow[r, green(ncs), shift left=0.5ex, outer sep=-1.4pt, "D_{01}^{11}"]  & \wh{C}_N^{11} \arrow[r, "\partial_{\mathrm{h}}^{21}"{below}] \arrow[u, crossing over, "\partial_{\mathrm{v}}^{12}" {yshift=6pt, swap}] \arrow[r, green(ncs), shift left=0.5ex, outer sep=-1.4pt, "D_{11}^{21}" pos=0.4] & \wh{C}_N^{21}
\end{tikzcd}
\end{equation}
Let $\ov{H}\coloneqq H\setminus\{0\}$. Recall from~\cite{GGV1}*{Sections~6 and~7}, that
\begin{align*}
&\wh{C}_N^{01} \simeq \{f\colon \ov{H}\to I\},\quad \wh{C}_N^{11} \simeq \{f\colon \ov{H}^2\to I\},\quad \wh{C}_N^{21} \simeq \{f\colon \ov{H}^3\to I\},\\
&\wh{C}_N^{02} \simeq \{f\colon \ov{H}^2\to I : f(h,l) = f(l,h)\},\quad \wh{C}_N^{12} \simeq \{f\colon \ov{H}^3\to I : f(h,l,m) = f(h,m,l)\},\\
&\wh{C}_N^{03} \simeq \{f\colon \ov{H}^3\to I : f(h,l,m) = f(l,h,m)-f(l,m,h)= f(h,m,l)-f(m,h,l)\},
\end{align*}
and that, for $h,l,m\in H$, $t\in \wh{C}_N^{01}$, $\beta\in \wh{C}_N^{02}$ and $f\in \wh{C}_N^{11}$, we have
\begin{align*}
&\partial_{\mathrm{v}}^{03}(\beta)(h,l,m)=\beta(l,m) - \beta(h+l,m) + \beta(h,l+m) - \beta(h,l), \\
&\partial_{\mathrm{v}}^{02}(t)(h,l)=t(l)-t(h+l)+t(h), \\
&\partial_{\mathrm{v}}^{12}(f)(h,l,m) =-f(h,m)+ f(h,l+m)-f(h,l),\\
&\partial_{\mathrm{h}}^{12}(\beta)(h,l,m)=\beta(h\cdot l,h\cdot m)-h\blackdiamond\beta(l,m),\\
&(\partial_{\mathrm{h}}^{11}+ D_{01}^{11})(t)(h,l) = t(h\cdot l)-h\blackdiamond t(l)-l\blackdiamond t(h)\Yleft h\cdot l,\\
&(\partial_{\mathrm{h}}^{21}+D_{11}^{21})(f)(h,l,m) = f(h\cdot l,h\cdot m)-f(h+l,m)+(h\cdot l)\blackdiamond f(h,m)+(h\cdot l)\blackdiamond f(h,l)\Yleft (h+l)\cdot m\\
\shortintertext{and}
& D_{02}^{21}(\beta)(h,l,m) = (h+l)\blackdiamond \beta(h,l)\Yleft (h+l)\cdot m,
\end{align*}
with the understanding that $t(0) = \beta(0,h) = \beta(h,0) = f(0,h) = f(h,0) = 0$, for all $h\in \ov{H}$.

\smallskip

Note that $\{f\colon \ov{H}^r\to I\}$ can be canonically identified with the set of all $f\colon H^r\to I$, such that $f(h_1,\dots,h_r) = 0$, if some $h_i=0$. We will use this freely.

\begin{remark} When $\Yleft = 0$, we will write $\Ext_{\blackdiamond}(H;I)$ instead of $\Ext_{\blackdiamond,0}(H;I)$ and $\Ho^2_{\blackdiamond}(H,I)$ instead of $\Ho^2_{\blackdiamond,0}(H,I)$. Recall that in an extension $(\iota,I\times_{\beta,f}^{\blackdiamond,\Yleft} H,\pi)$, the map $\Yleft = 0$ if and only if $I$ is included in the socle of $I\times_{\beta,f}^{\blackdiamond,\Yleft} H$.
\end{remark}

\section[Cohomology of linear cycle sets]{Computing the second group of cohomology}\label{section 2}

Let $H$ be a linear cycle set and let $I$ be an abelian group. By the comments at the beginning of~\cite{GGV1}*{Section~7} and by~\cite{GGV1}*{Corollary~7.10}, in order to classify the classes of the extensions of a lineal cycle set $H$ by a trivial cycle set $I$, we must determine all the maps $\blackdiamond\colon H\times I\to I$ and $\Yleft \colon I\times H\to I$ such that the conditions~\eqref{eqq 1}--\eqref{eqq 3} are satisfied, and then, for each one of this pair of maps, we must compute the second cohomology group $\ho_{\blackdiamond,\Yleft}^2(H,I)$. In this section we fix maps $\blackdiamond$ and $\Yleft$ as above, and we obtain a description of $\Ho^2_{\blackdiamond,\Yleft}(H,I)$, under the hypothesis that~$H$ is a finite abelian group via $\times$.

\subsection[The cohomology of the first column]{The cohomology of the first column}
In this subsection we compute the cohomology groups $\Harr^1(H,I)\coloneqq \ker(\partial_{\mathrm{v}}^{02})$ and $\Harr^2(H,I)\coloneqq \frac{\ker(\partial_{\mathrm{v}}^{03})} {\ima(\partial_{\mathrm{v}}^{02})}$ of the first col\-umn of the diagram~\eqref{diagrama1}, when $H$ is finite. To
begin with, we note that this cohomology groups only depend on the additive structure of $H$.

\smallskip

In the following propositions, we assume that $H$ is the direct sum of two additive abelian groups, $H_1$ and $H_2$, which are not necessarily sublinear cycle sets of $H$. The canonical projections and inclusions $p_1\colon H\to H_1$, $p_2\colon H\to H_2$, $\iota_1\colon H_1\to H$ and $\iota_2\colon H_2\to H$, induce morphisms $\iota_1^*$, $\iota_2^*$, $p_1^*$ and $p_2^*$ between the cochain complexes $\bigl(\wh{C}_N^{0*}(H_1,I),\partial_{\mathrm{v}}^{0*}\bigr)$, $\bigl(\wh{C}_N^{0*}(H_2,I),\partial_{\mathrm{v}}^{0*}\bigr)$ and $\bigl(\wh{C}_N^{0*}(H,I),\partial_{\mathrm{v}}^{0*}\bigr)$. Clearly
\begin{equation}\label{dos comp}
\iota_1^*\xcirc p_1^*=\Ide_{(\wh{C}_N^{0*}(H_1,I),\partial_{\mathrm{v}}^{0*})}\quad\text{and}\quad\iota_2^*\xcirc p_2^* = \Ide_{(\wh{C}_N^{0*}(H_1,I),\partial_{\mathrm{v}}^{0*})}.
\end{equation}
We let $\bar{\iota}_1$, $\bar{\iota}_2$, $\bar{p}_1$ and $\bar{p}_2$ denote the maps induced by $\iota_1^*$, $\iota_2^*$, $p_1^*$ and $p_2^*$, respectively, between $\Harr^2(H_1,I)$, $\Harr^2(H_2,I)$ and $\Harr^2(H,I)$.

\begin{proposition} If $H = H_1\oplus H_2$, then $\Harr^1(H,I) = \Harr^1(H_1,I)\oplus \Harr^1(H_2,I)$.
\end{proposition}

\begin{proof} In fact,
$$
\Harr^1(H,I)\simeq \Hom(H,I) \simeq \Hom(H_1,I)\oplus \Hom(H_2,I) = \Harr^1(H_1,I)\oplus \Harr^1(H_2,I),
$$
as desired.
\end{proof}

\begin{proposition}\label{Harr^2 es aditivo} If $H=H_1\oplus H_2$, then $\Harr^2(H,I)=\Harr^2(H_1,I)\oplus\Harr^2(H_2,I)$. Under this correspondence, the class $[\beta]\in\Harr^2(H,I)$, of a~nor\-malized symmetric cocycle $\beta$, is identified with the class of $(\beta_1,\beta_2)$, where
$$
\beta_1(h_1,l_1)\coloneqq \beta((h_1,0),(l_1,0))\quad\text{and}\quad \beta_2(h_2,l_2)\coloneqq \beta((0,h_2),(0,l_2)).
$$
\end{proposition}

\begin{proof} We will prove that $\Harr^2(-,I)$ is an additive functor. By identity~\eqref{dos comp}, for this we only must check that $\bar{p}_1\xcirc \bar{\iota}_1 + \bar{p}_2\xcirc \bar{\iota}_2=\Ide_{\Harr^2(H,I)}$. Let $\beta$, $\beta_1$ and $\beta_2$ be as in the statement. Note that $[\beta_i]= \bar{\iota}_i([\beta])$, and so we must prove that $[\beta]=\bar{p}_1([\beta_1])+\bar{p}_2([\beta_2])$. Define $t\in \hat{C}_N^{01}(H,I)$ by $t(h_1,h_2)\coloneqq \beta((h_1,0),(0,h_2))$. In order to finish the proof it suffices to check that
$$
p_1^*(\beta_1) + p_2^*(\beta_2)-\beta=\partial_{v}^{02}(t),
$$
which means that for all $h=(h_1,h_2),l=(l_1,l_2)\in H$, we have
\begin{equation}\label{linealizacion es coborde}
\beta((h_1,0),(l_1,0))+\beta((0,h_2),(0,l_2))-\beta(h,l)=t(l)-t(h+l)+t(h).
\end{equation}
Using the equality
\begin{equation}\label{beta es cociclo}
\beta(l,m) - \beta(h+l,m) + \beta(h,l+m) - \beta(h,l) = \partial_{\mathrm{v}}^{03}(\beta)(h,l,m) = 0 \qquad\text{for $h,l,m\in H$},
\end{equation}
with $h=(h_1,0)$, $l=(0,h_2)$ and $m=(l_1,l_2)$, we obtain
\begin{equation}\label{ciclo 1}
\beta((h_1,h_2),(l_1,l_2))=\beta((h_1,0),(0,h_2)+(l_1,l_2))-\beta((h_1,0),(0,h_2))+\beta((0,h_2),(l_1,l_2)).
\end{equation}
From~\eqref{beta es cociclo} with $h=(h_1,0)$, $l=(l_1,0)$ and $m=(0,h_2+l_2)$, we obtain
\begin{equation}\label{ciclo 2}
\beta((h_1,0),(0,h_2)+(l_1,l_2))=\beta((h_1+l_1,0),(0,h_2+l_2))-\beta((l_1,0),(0,h_2+l_2))+\beta((h_1,0),(l_1,0)),
\end{equation}
whereas from~\eqref{beta es cociclo} with $h=(0,h_2)$, $l=(0,l_2)$ and $m=(l_1,0)$, we obtain
\begin{equation}\label{ciclo 3}
\beta((0,h_2),(l_1,l_2))=\beta((0,h_2+l_2),(l_1,0))-\beta((0,l_2),(l_1,0))+\beta((0,h_2),(0,l_2)).
\end{equation}
Inserting into~\eqref{ciclo 1}, the values of $\beta((h_1,0),(0,h_2)+(l_1,l_2))$ and $\beta((0,h_2),(l_1,l_2))$, obtained in~\eqref{ciclo 2} and~\eqref{ciclo 3}, and using the facts that
$$
t((h_1,h_2)+(l_1,l_2))=\beta((h_1+l_1,0),(0,h_2+l_2)),\quad t((h_1,h_2))=\beta((h_1,0),(0,h_2)),\quad t((l_1,l_2))=\beta((l_1,0),(0,l_2)),
$$
and that $\beta$ is symmetric, we obtain
$$
\beta((h_1,h_2),(l_1,l_2))=t((h_1,h_2)+(l_1,l_2))+\beta((h_1,0),(l_1,0))-t(h_1,h_2)+\beta((0,h_2),(0,l_2))-t(l_1,l_2),
$$
which is~\eqref{linealizacion es coborde}, finishing the proof.
\end{proof}

In Lemma~\ref{cociclos verticales} and in Remarks~\ref{generadores canonicos de los cociclos} and~\ref{los bordes verticales}, we assume that the additive underlying structure of $H$ is $\mathds{Z}_d$.

\begin{lemma}\label{cociclos verticales} For $\beta\in \wh{C}_N^{02}(H,I)$, we have $\beta \in \ker(\partial_{\mathrm{v}}^{03})$ if and only if \begin{equation}\label{eq1}
\beta(i,j) = \sum_{k=j}^{i+j-1} \beta(1,k)- \sum_{k=1}^{i-1} \beta(1,k)\qquad\text{for $1\le i,j< d$.}
\end{equation}
\end{lemma}

\begin{proof} Assume that $\beta \in \ker(\partial_{\mathrm{v}}^{03})$. Then, for all $i,j$, we have
\begin{equation*}
0 = \partial_{\mathrm{v}}^{03} \beta(1,i,j) = \beta(i,j) - \beta(i+1,j) + \beta(1,i+j) - \beta(1,i).
\end{equation*}
Thus, $\beta(i+1,j) = \beta(i,j) + \beta(1,i+j) - \beta(1,i)$. An inductive argument on $i$ using this fact proves that~\eqref{eq1} is true. Conversely, assume that~\eqref{eq1} holds. We must show that
\begin{equation*}
\beta(b,c) - \beta(a+b,c) + \beta(a,b+c) - \beta(a,b) = 0\qquad\text{for all $0\le a,b,c< d$}.
\end{equation*}
But this is a consequence of the fact that $\beta(i,j) = \sum_{k=j}^{i+j-1} \beta(1,k) - \sum_{k=1}^{i-1} \beta(1,k)$ for all $i,j\in \mathds{N}$ (which follows from~\eqref{eq1}).
\end{proof}

\begin{remark}\label{generadores canonicos de los cociclos} Lemma~\ref{cociclos verticales} implies that each $\beta \in \ker(\partial_{\mathrm{v}}^{03})$ is uniquely determined by $\beta(1,1), \dots,\beta(1,d-1)$. For~ex\-ample, for each $1\le k < d$ and each $\gamma\in I$, the element $\beta_{kd}(\gamma)$, given, for each $0\le i,j<d$, by
\begin{equation}\label{bala1}
\beta_{kd}(\gamma)(i,j) \coloneqq \begin{cases} \phantom{-}\gamma &\text{if $0< i,j\le k<i+j$,}\\ -\gamma &\text{if $i,j>k$ and $i+j-d\le k$,}\\ \phantom{-} 0 &\text{otherwise,}\end{cases}
\end{equation}
is the unique $\beta \in \ker(\partial_{\mathrm{v}}^{03})$ with $\beta(1,k) = \gamma$ and $\beta(1,j)=0$, for $j\ne k$. Note that $\beta_{kd}$'s are linear in $\gamma$ and that
\begin{equation*}
\beta= \sum_{k=1}^{d-1} \beta_{kd}(\beta(1,k))\qquad\text{for each $\beta \in \ker(\partial_{\mathrm{v}}^{03})$.}
\end{equation*}
In particular $\{\beta_{kd}(\gamma): 1\le k<d\text{ and } \gamma\in I\}$ generate $\ker(\partial_{\mathrm{v}}^{03})$.
\end{remark}

For each $1\le k< d$ and $\gamma\in I$, let $\chi_k(\gamma)\in \wh{C}_N^{01}(H,I)$ be the map given by
$$
\chi_k(\gamma)(i)\coloneqq\begin{cases} \gamma &\text{if $i=k$,}\\ 0 & \text{otherwise.}\end{cases}
$$

\begin{remark}\label{los bordes verticales} Note that $\chi_k(\gamma)$'s generate $\wh{C}_N^{01}(H,I)$. Moreover, a direct computation using Remark~\ref{generadores canonicos de los cociclos}, shows that
\begin{equation*}
\partial_{\mathrm{v}}^{02}(\chi_k(\gamma)) = \begin{cases}  2 \beta_{1d}(\gamma) + \beta_{2d}(\gamma) +\cdots + \beta_{d-1,d}(\gamma) & \text{if $k = 1$,}\\  \beta_{kd}(\gamma) - \beta_{k-1,d}(\gamma) & \text{if $k\ne 1$.}
\end{cases}
\end{equation*}
This implies that the $\beta_{d-1,d}(\gamma)$'s generate $\frac{\ker(\partial_{\mathrm{v}}^{03})} {\ima(\partial_{\mathrm{v}}^{02})}$ and $\frac{\ker(\partial_{\mathrm{v}}^{03})} {\ima(\partial_{\mathrm{v}}^{02})}\simeq \frac{I}{dI}$.
\end{remark}

\begin{remark}\label{nota1} From now on we will assume that the additive underlying structure of $H$ is a direct sum $\mathds{Z}_{d_1}\oplus \dots \oplus \mathds{Z}_{d_n}$, of a finite number of finite cyclic groups. By Lemma~\ref{cociclos verticales} and Remarks~\ref{generadores canonicos de los cociclos} and~\ref{los bordes verticales}, the cohomology group $\Harr^2(\mathds{Z}_{d_i},I)$ is generated by the symmetric cocycles $\alpha_i(\gamma)\colon \mathds{Z}_{d_i}\times \mathds{Z}_{d_i}\to I$ (where $\gamma\in I$) given by
$$
\alpha_i(\gamma)(h,l)\coloneqq \beta_{d_i-1,d_i}(\gamma)(h,l) = \begin{cases} \gamma & \text{if $0\le h,l<d_i$ and $h+l\ge d_i$,}\\ 0 & \text{if $0\le h,l<d_i$ and $h+l<d_i$.}\end{cases}
$$
Moreover, there is a $\mathds{Z}$-linear isomorphism $\Harr^2(\mathds{Z}_{d_i},I)\simeq \frac{I}{d_i I}$, which identifies the class of $\alpha_i(\gamma)$ with the class of~$\gamma$. Combining this with Proposition~\ref{Harr^2 es aditivo}, we obtain an isomorphism $\Harr^2(H,I)\simeq \frac{I}{d_1I}\oplus \dots \oplus \frac{I}{d_nI}$.
\end{remark}

\begin{remark}\label{nota2} For $\gamma_1,\dots,\gamma_n\in I$, we define $\alpha_{\gamma_1,\dots,\gamma_n}\colon H\times H\to I$, by
$$
\alpha_{\gamma_1,\dots,\gamma_n}((h_1,\dots,h_n),(l_1,\dots,l_n))\coloneqq \alpha_1(\gamma_1)(h_1,l_1)+\cdots + \alpha_n(\gamma_n)(h_n,l_n).
$$
By Proposition~\ref{Harr^2 es aditivo} and Remark~\ref{nota1}, each $2$-cocycle in the first column of~\eqref{diagrama1} is cohomologous to some $\alpha_{\gamma_1,\dots,\gamma_n}$.
\end{remark}

\begin{corollary}\label{basta tomar combinacion de estandar} Each $2$-cocycle $(\beta,g) \in \wh{C}^{02}_N(H,I)\oplus \wh{C}^{11}_N(H,I)$, of the complex $(\wh{C}_N^*(H,I),\partial+D)$, is cohomologous to a $2$-cocycle of the form $(\alpha,f)$, with $\alpha\coloneqq \alpha_{\gamma_1,\dots,\gamma_n}$, for some $\gamma_1,\dots,\gamma_n\in I$.
\end{corollary}

\subsection[The cohomology \texorpdfstring{$\Ho^2_{\blackdiamond,\Yleft}(H,I)$}{H2} in the finite abelian case]{The cohomology \texorpdfstring{$\pmb{\Ho^2_{\blackdiamond,\Yleft}(H,I)}$}{H2} in the finite abelian case}

Recall, from Remark~\ref{nota1}, that, as an additive group, $H=\mathds{Z}_{d_1}\oplus \dots \oplus \mathds{Z}_{d_n}$. We let $e_i\in H$ denote the element, whose only non-zero coordinate is the $i$-th one, which equals~$1$. From now on, we will assume that $(H,\times)$ is abelian. So,
$$
(H,\times)=\langle a_1\rangle \oplus \cdots \oplus \langle a_s\rangle.
$$
We let $r_i$ denote the order of $a_i$. We will find all the $2$-cocycles $(\alpha,-f)$ of $(\wh{C}_N^*(H,I),\partial+D)$, with $\alpha\coloneqq \alpha_{\gamma_1,\dots,\gamma_n}$. So, from now on, we assume that $\alpha \coloneqq \alpha_{\gamma_1,\dots,\gamma_n}$ and $f\in \wh{C}_N^{11}(H,I)$.

\subsubsection[Properties of cocycles]{Properties of cocycles}\label{Properties of cocycles}
The aim of this subsection is to show that each $2$-cocycle $(\alpha,-f)$ of $(\wh{C}_N^*(H,I),\partial+D)$ is completed determined by $(\gamma_k)_{1\le k\le n}$ and $(\mathfrak{f}_{ji})_{1\le j\le s,1\le i\le n}$, where $\mathfrak{f}_{ji}\coloneqq f(a_j,e_i)$. We prove this in Theorem~\ref{ppal1}. There, we also prove that the correspondence
$$
\bigl((\gamma_k)_{1\le k\le n},(\mathfrak{f}_{ji})_{1\le j\le s,1\le i\le n}\bigr) \mapsto (\alpha,-f)
$$
is linear. Of course, not every tuple $\bigl((\gamma_k)_{1\le k\le n},(\mathfrak{f}_{ji})_{1\le j\le s,1\le i\le n}\bigr)$ yields a cocycle.

\smallskip

Let $f\in \wh{C}^{11}_N(H,I)$. The pair $(\alpha,-f)$ is a $2$-cocycle of $(\wh{C}_N^*(H,I),\partial+D)$ if and only if $\partial_{\mathrm{v}}^{03}(\alpha) = 0$, $\partial_{\mathrm{h}}^{12}(\alpha)=\partial_{\mathrm{v}}^{12}(f)$ and $(\partial_{\mathrm{h}}^{21}+D_{11}^{21})(f)=D_{02}^{21}(\alpha)$. We already know that the first condition is satisfied. The second one means that
\begin{equation}\label{beta quasi linear}
f(h,l+m)=f(h,l)+f(h,m)+\alpha(h\cdot l,h\cdot m)-h\blackdiamond\alpha(l,m)\qquad\text{for all $h,l,m\in H$;}
\end{equation}
and the last one means that
\begin{equation}\label{delta horizontal es cero}
f(h+l,m)=f(h\cdot l,h\cdot m)+(h\cdot l)\blackdiamond f(h,m)+(h\cdot l)\blackdiamond f(h,l) \Yleft (h+l)\cdot m - (h+l)\blackdiamond \alpha(h,l)\Yleft (h+l)\cdot m.
\end{equation}
If condition~\eqref{beta quasi linear} is satisfied, then we will say that $f$ is {\em $\alpha$-quasi linear in the second variable.}

\begin{proposition}\label{dependencia de f} If $\partial_{\mathrm{h}}^{12}(\alpha) = \partial_{\mathrm{v}}^{12}(f)$, then, for $h\in H$, $1\le i \le n$ and $0\le \lambda_i< d_i$, we have
\begin{align}
& f(h,\lambda_i e_i)= \lambda_i f(h,e_i)+\sum_{\ell=1}^{\lambda_i-1}\alpha(h\cdot e_i,\ell (h\cdot e_i)),\label{quasilineal 1}\\
& f\left(h,\sum_{i=1}^n \lambda_i e_i\right)= \sum_{i=1}^n f(h,\lambda_i e_i) +\sum_{k=2}^n \alpha\left(\sum_{i=1}^{k-1}\lambda_i (h\cdot e_i),\lambda_k (h\cdot e_k)\right), \label{quasilineal 2} \\
\shortintertext{and}
& d_i f(h,e_i)+\sum_{\ell=1}^{d_i-1}\alpha(h\cdot e_i,\ell (h\cdot e_i))-h\blackdiamond \gamma_i = 0.\label{quasilineal 3}
\end{align}
\end{proposition}

\begin{proof} The equality~\eqref{quasilineal 1} follows by induction on $\lambda_i$, using formula~\eqref{beta quasi linear}, with $l=e_i$ and $m=(\lambda_i-1)e_i$, and the fact that $\alpha(l,m) = 0$. To prove~\eqref{quasilineal 2}, it suffices to show that, for $k \le n$,
$$
f\left(h,\sum_{i=1}^k \lambda_i e_i\right)= \sum_{i=1}^k f(h,\lambda_i e_i)+\sum_{\ell=2}^k \alpha\left(\sum_{i=1}^{\ell-1}\lambda_i (h\cdot e_i), \lambda_{\ell} (h\cdot e_{\ell})\right).
$$
This follows by induction on $k$, using~\eqref{beta quasi linear}, with $l=\sum_{i=1}^{k-1}\lambda_i e_i$ and $m=\lambda_{k} e_{k}$, and the fact that $\alpha(l,m) = 0$. Finally, condition~\eqref{quasilineal 3} follows by applying~\eqref{beta quasi linear}, with $l=e_i$ and $m=(d_i-1) e_i$, and then using condition~\eqref{quasilineal 1} and the fact that $\alpha(l,m)=\gamma_i$.
\end{proof}


\begin{corollary}\label{coro de dependencia de f} For each $h,l\in H$ the value $f(h,l)$ is determined by the $\gamma_i$'s and the $f(h,e_i)$'s. More\-over, this~cor\-re\-spondence is linear in $(\gamma_1,\dots,\gamma_n,f(h,e_1),\dots,f(h,e_n))$.
\end{corollary}

\begin{lemma}\label{para b ele} Let $h,m\in H$ and set $\bar{h}_{\ell}\coloneqq h^{\times(\ell+1)}-h$, for $\ell\in \mathds{N}_0$. The following equalities hold:
\begin{equation}\label{primera para b ele}
h\cdot \bar{h}_{\ell}=h^{\times \ell}\quad\text{and}\quad h^{\times(\ell+1)}\cdot m=(h+\bar{h}_{\ell})\cdot m=h^{\times \ell}\cdot (h\cdot m).
\end{equation}
\end{lemma}

\begin{proof} By~\eqref{equivalencia linear cycle set braza}, we have
$$
h\cdot\bar{h}_{\ell}=h^{\times-1}\times (h+\bar{h}_{\ell})=h^{\times-1}\times h^{\times(\ell+1)}=h^{\times \ell}.
$$
Hence, by~\eqref{compatibilidades cdot suma y suma cdot},
$$
h^{\times(\ell+1)}\cdot m= (h+\bar{h}_{\ell})\cdot m=(h\cdot \bar{h}_{\ell})\cdot (h\cdot m)=h^{\times \ell}\cdot (h\cdot m),
$$
which finishes the proof.
\end{proof}

\begin{proposition}\label{delta nos da times} Assume that $(\alpha,-f)\in \wh{C}^{02}_N(H,I)\oplus \wh{C}^{11}_N(H,I)$ is a $2$-cocycle of $(\wh{C}_N^*(H,I),\partial+D)$. Take $h,m\in H$ and set $\bar{h}_{\ell}$ as in the previous lem\-ma. For all $k\in \mathds{N}_0$, we have
\begin{equation}\label{f de times}
f(h^{\times k},m)=\sum_{\ell=0}^{k-1} h^{\times (k-\ell-1)} \blackdiamond f(h,h^{\times \ell}\cdot m)+\sum_{\ell=1}^{k-1} h^{\times \ell} \blackdiamond f(h,\bar{h}_{\ell})\Yleft h^{\times k}\cdot m -\sum_{\ell=1}^{k-1} h^{\times (\ell+1)} \blackdiamond \alpha(h,\bar{h}_{\ell})\Yleft h^{\times k}\cdot m.
\end{equation}
\end{proposition}

\begin{proof} Recall that the equality $(\partial_{\mathrm{h}}^{21}+D_{11}^{21})(f)=D_{02}^{21}(\alpha)$ yields condition~\eqref{delta horizontal es cero}. For $k=0$, the equality~\eqref{f de times} holds sin\-ce $f(0,m) = 0$. Assume this is true for some $k\ge 0$. By Lemma~\ref{para b ele}, we know that
$$
h\cdot \bar{h}_k=h^{\times k}\qquad\text{and}\qquad h^{\times(k+1)}\cdot m=h^{\times k}\cdot (h\cdot m).
$$
Since, also $h+\bar{h}_k=h^{\times(k+1)}$, from condition~\eqref{delta horizontal es cero} with $(h,l,m)=(h,\bar{h}_k,m)$, we obtain
\begin{align*}
f(h^{\times (k+1)},m) &= f(h+\bar{h}_k,m)\\
%
%
&= f(h^{\times k},h\cdot m)+ h^{\times k}\blackdiamond f(h,m)+  h^{\times k}\blackdiamond f(h,\bar{h}_k)\Yleft h^{\times(k+1)}\cdot m -h^{\times(k+1)}\blackdiamond \alpha(h,\bar{h}_k)\Yleft h^{\times(k+1)}\cdot m\\
%
%
&= \sum_{\ell=1}^k h^{\times (k-\ell)}\blackdiamond f(h,h^{\times \ell}\!\cdot m)+\sum_{\ell=1}^{k-1} h^{\times \ell}\blackdiamond f(h,\bar{h}_{\ell})\Yleft h^{\times (k+1)}\!\cdot m -\sum_{\ell=2}^k h^{\times \ell}\blackdiamond \alpha(h,\bar{h}_{\ell-1})\Yleft h^{\times (k+1)}\!\cdot m\\
& + h^{\times k}\blackdiamond f(h,m)+ h^{\times k}\blackdiamond f(h,\bar{h}_k)\Yleft h^{\times(k+1)}\cdot m- h^{\times (k+1)} \blackdiamond \alpha(h,\bar{h}_k)\Yleft h^{\times(k+1)}\cdot m\\
%
%
& = \sum_{\ell=0}^k h^{\times (k-\ell)}\blackdiamond f(h,h^{\times \ell}\cdot m)+\sum_{\ell=1}^k h^{\times \ell}\blackdiamond f(h,\bar{h}_{\ell}) \Yleft h^{\times (k+1)}\cdot m - \sum_{\ell=1}^k h^{\times (\ell+1)}\blackdiamond \alpha(h,\bar{h}_{\ell})\Yleft h^{\times (k+1)}\cdot m,
\end{align*}
as desired.
\end{proof}

\begin{corollary}\label{coro de delta nos da times} For each $h,l\in H$ and $k\in \mathds{N}_0$, the value $f(h^{\times k},l)$ is determined by the $\gamma_i$'s and the $f(h,e_i)$'s.~More\-over, this~cor\-re\-spondence is linear in $(\gamma_1,\dots,\gamma_n,f(h,e_1),\dots,f(h,e_n))$.
\end{corollary}

\begin{proof} This follows easily from the formula~\eqref{f de times} and Corollary~\ref{coro de dependencia de f}, using that $\Yleft$ is left linear.
\end{proof}

\begin{theorem}\label{ppal1} Let $a_1,\dots,a_s$ be as at the beginning of this section. If $(\alpha,-f)$ is a $2$-cocycle of $(\wh{C}_N^*(H,I),\partial+D)$, then, for each $h,l\in H$, the value $f(h,l)$ is determined by the $\gamma_i$'s and the $f(a_j,e_i)$'s. More\-over, this correspondence is linear in $(\gamma_1,\dots,\gamma_n,f(a_1,e_1),\dots,f(a_1,e_n),\dots, f(a_s,e_1),\dots, f(a_s,e_n))$. Furthermore, if we set $\bar{a}_{j\ell}\coloneqq a_j^{\times(\ell+1)}-a_j$, then, for each $m\in H$ and $j=1,\dots,s$, we have
\begin{equation}\label{condicion en times}
\sum_{\ell=0}^{r_j-1} a_j^{\times (r_j-1-\ell)}\blackdiamond f(a_j,a_j^{\times \ell}\cdot m)+\sum_{\ell=1}^{r_j-1} a_j^{\times \ell}\blackdiamond f(a_j,\bar{a}_{j\ell}) \Yleft m -\sum_{\ell=1}^{r_j-1} a_j^{\times (\ell+1)}\blackdiamond \alpha(a_j,\bar{a}_{j\ell})\Yleft m = 0.
\end{equation}
\end{theorem}

\begin{proof} Take $h,l\in H$. Assume that $h = a_j^{\times k}$, with $1\le j\le s$ and $1\le k\le r_j$. By Corollary~\ref{coro de delta nos da times} with $h=a_j$, the value $f(h,l)$ is linearly determined by $(\gamma_1,\dots,\gamma_n,f(a_j,e_1),\dots,f(a_j,e_n))$. Assume now that $h$ is arbitrary and write
$$
h=a_1^{\times k_1}\times a_2^{\times k_2}\times\dots \times a_s^{\times k_m}\quad\text{with $1\le k_j\le r_j$.}
$$
For $1<\ell<s$, set $\ov{h}_{\ell}\coloneqq a_1^{\times k_1}\times \dots \times a_{\ell-1}^{\times k_{\ell-1}}$ and $\ov{l}_{\ell}\coloneqq \ov{h}_{\ell}\times a_{\ell}^{\times k_{\ell}}-\ov{h}_{\ell}$. Note that
$$
\ov{h}_{\ell}+\ov{l}_{\ell}=a_1^{\times k_1}\times \dots \times a_{\ell}^{\times k_{\ell}} = \ov{h}_{\ell+1} \quad\text{and}\quad \ov{h}_{\ell}\cdot \ov{l}_{\ell}=\ov{h}_{\ell}^{\times-1}\times (\ov{h}_{\ell}+\ov{l}_{\ell})=\ov{h}_{\ell}^{\times-1}\times \ov{h}_{\ell}\times a_{\ell}^{\times k_{\ell}}= a_{\ell}^{\times k_{\ell}}.
$$
From~\eqref{delta horizontal es cero} with $\beta=\alpha$, $h=\ov{h}_{\ell}$ and $l=\ov{l}_{\ell}$, we obtain the equality
\begin{align*}
f(\ov{h}_{\ell+1},m)&=f(\ov{h}_{\ell}+\ov{l}_{\ell},m)\\
&=f(a_{\ell}^{\times k_{\ell}},\ov{h}_{\ell}\cdot m)+(\ov{h}_{\ell}\cdot \ov{l}_{\ell})\blackdiamond f(\ov{h}_{\ell},m) + (\ov{h}_{\ell}\cdot \ov{l}_{\ell})\blackdiamond f(\ov{h}_{\ell},\ov{l}_{\ell}) \Yleft \ov{h}_{\ell+1}\cdot m - \ov{h}_{\ell+1}\blackdiamond \alpha(\ov{h}_{\ell},\ov{l}_{\ell}) \Yleft \ov{h}_{\ell+1}\cdot m,
\end{align*}
which shows by induction on $\ell$, that $f(h,l)$ depends linearly on
$$
(\gamma_1,\dots,\gamma_n,f(a_1,e_1),\dots,f(a_1,e_n),\dots, f(a_s,e_1),\dots, f(a_s,e_n)),
$$
as desired. Finally, since $a_j^{\times r_j}\cdot m = 0\cdot m = m$, the formula~\eqref{f de times}, with $\beta=\alpha$, $h=a_j$ and $k=r_j$, yields~\eqref{condicion en times}.
\end{proof}

\begin{proposition}\label{equivalencia cociclo con times} Let $f\colon H\times H\to I$ be a map. The equality
$$
f(h+d,m)=f(h\cdot d,h\cdot m)+(h\cdot d)\blackdiamond f(h,m)+(h\cdot d)\blackdiamond f(h,d) \Yleft (h+d)\cdot m -(h+d)\blackdiamond \alpha(h,d)\Yleft (h+d)\cdot m,
$$
holds, for all $h,d,m\in H$, if and only if
$$
f(h\times l,m)=f(l,h\cdot m)+ l\blackdiamond f(h,m)+ l\blackdiamond f(h,{}^hl)\Yleft (h\times l)\cdot m - (h\times l)\blackdiamond \alpha(h,{}^hl)\Yleft (h\times l)\cdot m,
$$
for all $h,l,m\in H$.
\end{proposition}

\begin{proof} Assume the first equality is true, for all $h,d,m\in H$; fix $l\in H$ and take $d\coloneqq h\times l-h = {}^hl$. Then $h+d=h\times l$ and~$h\cdot d = l$. Replacing these values into the first equality we obtain the second one. Conversely, assume the second equality holds, for all $h,l,m\in H$; fix $d\in H$ and take $l\coloneqq h\cdot d$. Then $h\times l=h+d$ and ${}^hl = d$, and re\-plac\-ing these values into the second equality we obtain the first one.
\end{proof}


\subsubsection[Construction of cocycles]{Construction of cocycles}
Given a set $X$, we let $\M_X$ denote the free magma generated by $X$. Hence, $\M_X = \coprod_{r\in \mathds{N}} X_r$, where $\coprod$ means disjoint union, and the $X_r$'s are recursively defined by
$$
X_1\coloneqq X\qquad\text{and}\qquad X_{r+1}\coloneqq \coprod_{p+q=r+1}  X_p\times X_q.
$$
If $a\in X_r$, then we say that $a$ has degree $r$ and we write $\deg(a) = r$. There is an obvious binary operation in $\M_X$.

\smallskip

Next, we consider the magma $\M_Z$, where $Z\coloneqq \{e_1,\dots,e_n\}$ and we let $+$ denote its binary operation. There~is a canonical magma homomorphism $\pi_Z\colon \M_Z\to H$, where $H$ is considered as a magma via $+$. For the sake of brevity, for each $r\in \mathds{N}$, we will set $\ov{Z}_r\coloneqq Z_1\coprod\cdots \coprod Z_r$, and for all~$a\in \M_Z$, we will write $a$ instead of $\pi_Z(a)$.

\smallskip

At the beginning of Subsection~\ref{Properties of cocycles} we prove that if $(\alpha,-f)$ is a $2$-cocycle of $(\wh{C}_N^*(H,I),\partial+D)$, then $f$ is $\alpha$-quasi linear in the second variable. The following definition is motivated by this fact.

\begin{definition}\label{def de f_{1r}} Let $Y\coloneqq \{a_1,\dots,a_s\}$. Given $f_{11}\colon Y\times Z\to I$, for $r\in \mathds{N}$, we define $f_{1r}\colon Y\times \ov{Z}_r\to I$, recursively~by
$$
f_{1,r+1}(a_j,m) = \begin{cases} f_{1r}(a_j,m) & \text{if $m\in \ov{Z}_r$,}\\ G_r(a_j,h,l) & \text{if $m = h+l$, with $h\in Z_u$, $l\in Z_v$ and $u+v=r+1$,} \end{cases}
$$
where
$$
G_r(a_j,h,l)\coloneqq f_{1r}(a_j,h)+f_{1r}(a_j,l)+\alpha(a_j\cdot h,a_j\cdot l)-a_j\blackdiamond\alpha(h,l).
$$
Then, we take $f_1\colon Y\times \M_Z\to I$ as the union of the $f_{1r}$ and $G\colon Y\times \M_Z\times \M_Z\to I$ as the union of the $G_r$.
\end{definition}

\begin{remark}\label{formula de quasilinealidad} Note that
$$
f_1(a_j,h+l) = G(a_j,h,l) = f_1(a_j,h)+f_1(a_j,l)+\alpha(a_j\cdot h,a_j\cdot l)-a_j\blackdiamond\alpha(h,l).
$$
\end{remark}

\begin{proposition}\label{construccion beta quasi lineal} For $1\le j\le m$ and $h,l,m\in \M_Z$, the following identities hold:
$$
f_1(a_j,(h+l)+m) = f_1(a_j,h+(l+m))\quad\text{and}\quad f_1(a_j,h+l) = f_1(a_j,l+h).
$$
\end{proposition}

\begin{proof} Since $\alpha$ is symmetric the second equality follows immediately from Definition~\ref{def de f_{1r}}. We next prove the first one. By Remark~\ref{formula de quasilinealidad}, we have
\begin{align*}
f_1(a_j,(h+l)+m)& = f_1(a_j,h+l)+f_1(a_j,m)+\alpha(a_j\cdot (h+l),a_j\cdot m)-a_j\blackdiamond\alpha(h+l,m)\\
& =  f_1(a_j,h)+f_1(a_j,l)+\alpha(a_j\cdot h,a_j\cdot l)-a_j\blackdiamond\alpha(h,l)\\
& + f_1(a_j,m)+\alpha(a_j\cdot (h+l),a_j\cdot m)-a_j\blackdiamond\alpha(h+l,m),
\shortintertext{and}
f_1(a_j,h+(l+m))& = f_1(a_j,h)+ f_1(a_j,l+m)+\alpha(a_j\cdot h,a_j\cdot (l+m))-a_j\blackdiamond\alpha(h,l+m)\\
& = f_1(a_j,h)+ f_1(a_j,l)+ f_1(a_j,m)+ \alpha(a_j\cdot l,a_j\cdot m)-a_j\blackdiamond\alpha(l,m)\\
& +\alpha(a_j\cdot h,a_j\cdot (l+m))-a_j\blackdiamond\alpha(h,l+m).
\end{align*}
Thus, in order to finish the proof, it suffices to check that
$$
\alpha(a_j\cdot h,a_j\cdot l)+\alpha(a_j\cdot (h+l),a_j\cdot m)= \alpha(a_j\cdot l,a_j\cdot m)+\alpha(a_j\cdot h,a_j\cdot (l+m))
$$
and
$$
a_j\blackdiamond\alpha(h+l,m)+a_j\blackdiamond\alpha(h,l)=a_j\blackdiamond\alpha(l,m)+a_j\blackdiamond\alpha(h,l+m).
$$
But these two equalities follow from~\eqref{compatibilidades cdot suma y suma cdot}, \eqref{eqq 2} and~\eqref{beta es cociclo}.
\end{proof}

\begin{corollary}\label{coro asociativo suma} The map $f_1\colon Y\times \M_Z\to I$, introduced in Definition~\ref{def de f_{1r}}, factorizes through $Y\times \AS_Z$, where~$\AS_Z$ is the free abelian semigroup generated by $Z$.
\end{corollary}

We will also write $f_1\colon Y\times \AS_Z\to I$ to denote the map induced by $f_1\colon Y\times \M_Z\to I$.

\begin{proposition}\label{casi dependencia de f} The map $f_1\colon Y\times \AS_Z\to I$ satisfies the following identities:
\begin{align}
& f_1(a_j,\lambda_i e_i) = \lambda_i f_1(a_j,e_i)+\sum_{\ell=1}^{\lambda_i-1}\alpha(a_j\cdot e_i,\ell(a_j\cdot e_i)),\label{quasilinealij 1}\\
& f_1\left(a_j,\sum_{i=1}^n \lambda_i e_i\right)= \sum_{i=1}^n f_1(a_j,\lambda_i e_i) +\sum_{k=2}^n \alpha\left(\sum_{i=1}^{k-1}\lambda_i (a_j\cdot e_i), \lambda_k (a_j \cdot e_k)\right), \label{quasilinealij 2}\\
\shortintertext{and}
& f_1(a_j,d_i e_i) = d_i f_1(a_j,e_i) +\sum_{k=1}^{d_i-1}\alpha(a_j\cdot e_i,k(a_j\cdot e_i))-a_j\blackdiamond \gamma_i,\label{quasilinealij 3}
\end{align}
where $1\le i\le n$, $1\le j\le s$ and $0\le \lambda_i<d_i$,
\end{proposition}

\begin{proof} Equalities~\eqref{quasilinealij 1} and~\eqref{quasilinealij 3} follow by induction on $\lambda_i$, using Remark~\ref{formula de quasilinealidad}, that $\alpha(e_i,\lambda_i e_i) = 0$, if $\lambda_i<d_i-1$, and that $\alpha(e_i,(d_i-1)e_i) = \gamma_i$. Finally~\eqref{quasilinealij 2} follows proving that
$$
f_1\left(a_j,\sum_{i=1}^r \lambda_i e_i\right)= \sum_{i=1}^r f_1(a_j,\lambda_i e_i) +\sum_{k=2}^r \alpha\left(\sum_{i=1}^{k-1}\lambda_i (a_j\cdot e_i), \lambda_k (a_j\cdot e_k)\right)
$$
for $1\le r\le n$ (proceed by induction on $r$, using Remark~\ref{formula de quasilinealidad} and the definition of $\alpha$).
\end{proof}

\begin{corollary}\label{coro casi dependencia de f} For $1\le i\le n$ and $1\le j\le s$, the equalities
\begin{equation}\label{basica}
d_i f_1(a_j,e_i) +\sum_{k=1}^{d_i-1}\alpha(a_j\cdot e_i,k (a_j\cdot e_i))-a_j\blackdiamond \gamma_i = 0
\end{equation}
hold if and only if the map $f_1\colon Y\times M_Z\to I$, introduced in Definition~\ref{def de f_{1r}}, factorizes through $Y\times H$.
\end{corollary}

\begin{proof} By Corollary~\ref{coro asociativo suma} and equalities~\eqref{quasilinealij 3}.
\end{proof}

\begin{notation}\label{factorizacion de f_1} We will also write $f_1\colon Y\times H\to I$ to denote the map induced by $f_1\colon Y\times \M_Z\to I$.
\end{notation}

Recall that $Y\coloneqq \{a_1,\dots,a_s\}$. Next, we consider the magma $\M_Y$, and we let $\times$ denote its binary operation. There is a canonical magma homomorphism $\pi_Y\colon \M_Y\to H$, where $H$ is considered as a magma via $\times$. For each~\hbox{$r\in \mathds{N}$}, we will set $\ov{Y}_r\coloneqq Y_1\coprod\cdots \coprod Y_r$, and for all~$h\in \M_Y$, we will write $h$ instead of $\pi_Y(h)$.

\smallskip

The following definition is motivated by condition~\eqref{beta quasi linear}.

\begin{definition} Let $A\subseteq \M_Y$. A map $f\colon A\times H\to I$ is {\em $\alpha$-quasi linear in the second variable}, if
$$
f(h,l+m)=f(h,l)+f(h,m)+\alpha(h\cdot l,h\cdot m)-h\blackdiamond\alpha(l,m),
$$
for all $h\in A$ and $l,m\in H$.
\end{definition}

\begin{remark}\label{f_1 es quasi lineal} By Remark~\ref{formula de quasilinealidad}, the map $f_1\colon Y\times H\to I$ is $\alpha$-quasi linear in the second variable.
\end{remark}

The following definition is motivated by identity~\eqref{delta horizontal es cero} and Proposition~\ref{equivalencia cociclo con times}.

\begin{definition}\label{def de f} Let $f_1\colon Y\times H\to I$ be as above. Assume that conditions~\eqref{basica} are fulfilled. For $r\in \mathds{N}$, we define the maps $f_r\colon \ov{Y}_r\times H\to I$, recursively by
$$
f_{r+1}(h,d) = \begin{cases} f_r(h,d) & \text{if $h\in \ov{Y}_r$,}\\ F_r(l,m,d) & \text{if $h = l\times m$, with $l\in Y_u$, $m\in Y_v$ and $u+v = r+1$,} \end{cases}
$$
where
$$
F_r(l,m,d)\coloneqq f_r(m,l\cdot d)+ m\blackdiamond f_r(l,d)+ m\blackdiamond f_r(l,{}^lm)\Yleft (l\times m)\cdot d - (l\times m)\blackdiamond \alpha(l,{}^lm)\Yleft (l\times m)\cdot d.
$$
Then, we take $f\colon \M_Y\times H\to I$ as the union of the $f_r$ and $F\colon \M_Y\times \M_Y\times H\to I$ as the union of the $F_r$.
\end{definition}

\begin{remark}\label{def de F(a,b,c)} Note that
$$
f(h\times l,m)= f(l,h\cdot m)+ l\blackdiamond f(h,m)+ l\blackdiamond f(h,{}^hl) \Yleft (h\times l)\cdot m -(h\times l)\blackdiamond \alpha(h,{}^hl)\Yleft (h\times l)\cdot m.
$$
\end{remark}

\begin{proposition}\label{quasi linear invariate por times} The map $f\colon \M_Y\times H\to I$ is $\alpha$-quasi linear in the second variable.
\end{proposition}

\begin{proof} By Remark~\ref{f_1 es quasi lineal}, we know that $f_1$ is $\alpha$-quasi linear in the second variable. Assume by induction that $f_r$ is. In order to check that $f_{r+1}$ is also, we must prove that
\begin{equation}\label{cuasi lineal de times}
f_{r+1}(h\times l,m+d)-f_{r+1}(h\times l,m)-f_{r+1}(h\times l,d)= \alpha((h\times l)\cdot m,(h\times l)\cdot d)-(h\times l)\blackdiamond \alpha(m,d),
\end{equation}
for all $h\in Y_s$ and $l\in Y_t$ with $s+t = r+1$. Since $\Yleft$ is bilinear and $\cdot$ is linear in the second variable, we have
\begin{align*}
& l\blackdiamond f_r(h,{}^hl) \Yleft (h\times l)\cdot (m+d)- l\blackdiamond f_r(h,{}^hl) \Yleft (h\times l)\cdot m - l\blackdiamond f_r(h,{}^hl) \Yleft (h\times l)\cdot d = 0
\shortintertext{and}
&(h\times l)\blackdiamond \alpha(h,{}^hl)\Yleft (h\times l)\cdot (m+d)-(h\times l)\blackdiamond \alpha(h,{}^hl)\Yleft (h\times l)\cdot m-(h\times l)\blackdiamond \alpha(h,{}^hl)\Yleft (h\times l)\cdot d = 0.
\end{align*}
Moreover, since $f_r$ is $\alpha$-quasi linear in the second variable,
\begin{align*}
&f_r(l,h\cdot (m+d))-f_r(l,h\cdot m)-f_r(l,h\cdot d)=\alpha(l\cdot(h\cdot m),l\cdot (h\cdot d))-l\blackdiamond \alpha(h\cdot m,h\cdot d)
\shortintertext{and}
& l\blackdiamond f_r(h,m+d)-l\blackdiamond f_r(h,m)-l\blackdiamond f_r(h,d)=l\blackdiamond \bigl(\alpha(h\cdot m,h\cdot d)-h \blackdiamond \alpha(m,d)\bigr).
\end{align*}
Hence
\begin{align*}
f_{r+1}(h\times l,m+d)-f_{r+1}(h\times l,m)-f_{r+1}(h\times l,d) &= F_r(h,l,m+d) - F_r(h,l,m)-F_r(h,l,d)\\
&= \alpha(l\cdot(h\cdot m),l\cdot (h\cdot d))-l\blackdiamond (h\blackdiamond \alpha(m,d)),
\end{align*}
which concludes the proof, since
$$
l\cdot(h\cdot m)=(h\times l)\cdot m,\quad l\cdot(h\cdot d)=(h\times l)\cdot d\quad\text{and}\quad l\blackdiamond (h\blackdiamond \alpha(m,d))=(h\times l) \blackdiamond \alpha(m,d),
$$
by~\eqref{accion tines con punto} and the first condition in~\eqref{eqq 2}.
\end{proof}

\begin{proposition}\label{asociativo} For all $h,l,m\in \M_Y$ and $d\in H$, we have
\begin{equation}\label{pepe1}
f((h\times l)\times m,d) = f(h\times (l\times m),d).
\end{equation}
\end{proposition}

\begin{proof} We proceed by induction on $\deg(h)+\deg(l)+\deg(m)$. Assume that the identity~\eqref{pepe1} holds when the sum of these degrees is less than or equal to $r$ and take $h,l,m\in \M_Y$ with $\deg(h)+\deg(l)+\deg(m) = r+1$. By Remark~\ref{def de F(a,b,c)}, we have
\begin{align*}
& f((h\times l)\times m,d) = f(m,(h\times l)\cdot d)+ m\blackdiamond f(h\times l,d)+ m\blackdiamond f(h\times l,{}^{h\times l}m)\Yleft t\cdot d - t\blackdiamond \alpha(h\times l,{}^{h\times l}m)\Yleft t\cdot d\\
\shortintertext{and}
& f(h\times (l\times m),d) = f(l\times m,h\cdot d)+ (l\times m)\blackdiamond f(h,d)+ (l\times m)\blackdiamond f(h,{}^h(l\times m)) \Yleft t\cdot d-t\blackdiamond \alpha(h,{}^h(l\times m))\Yleft t\cdot d,
\end{align*}
where $t\coloneqq h\times l\times m$. Using now Definitions~\ref{def de F(a,b,c)} and~\ref{def de f}, and taking into account that
$$
(h\times l)\cdot {}^{h\times l}m = m,\quad h\cdot {}^{h\times l}m={}^lm\quad\text{and}\quad (l\times m)\cdot (h\cdot d) = t\cdot d,
$$
we obtain
\begin{align*}
& f((h\times l)\times m,d) = f(m,u\cdot d)+m\blackdiamond\bigl(f(l,h\cdot d)+ l\blackdiamond f(h,d)+ l\blackdiamond f(h,{}^hl) \Yleft u\cdot d - u\blackdiamond \alpha(h,{}^hl)\Yleft u\cdot d\bigr)\\
&\phantom{f((h\times l)\times m,d)} + m\blackdiamond\bigl(f(l,{}^lm)+ l\blackdiamond f(h,{}^um)+ l\blackdiamond f(h,{}^hl) \Yleft m\bigr) \Yleft t\cdot d - m\blackdiamond \bigl(u\blackdiamond \alpha(h,{}^hl)\Yleft m\bigr) \Yleft t\cdot d\\
& \phantom{f((h\times l)\times m,d)} - t\blackdiamond \alpha(u,{}^um)\Yleft t\cdot d\\
\shortintertext{and}
& f(h\times (l\times m),d)  = f(m,l\cdot(h\cdot d))+m\blackdiamond f(l,h\cdot d)+ m\blackdiamond f(l,{}^lm)\Yleft t\cdot d- v\blackdiamond \alpha(l,{}^lm)\Yleft t\cdot d\\
& \phantom{f(h\times (l\times m),d) }+ v\blackdiamond f(h,d)+ v\blackdiamond f(h,{}^hv)\Yleft t\cdot d - t\blackdiamond \alpha(h,{}^hv)\Yleft t\cdot d,
\end{align*}
where $u\coloneqq h\times l = h+{}^u l$ and $v\coloneqq l\times m$. Since $l\cdot(h\cdot d)=u\cdot d$, $v\blackdiamond f(h,d)= m\blackdiamond (l \blackdiamond f(h,d))$, $\blackdiamond $ is linear in the second argument and $\Yleft$ is linear in the first argument, in order to check equality~\eqref{pepe1}, we are reduce to prove that
\begin{equation}\label{asociativo equivalente}
\begin{aligned}
m\blackdiamond\bigl(l\blackdiamond f(h,{}^hl) & \Yleft u\cdot d-u\blackdiamond \alpha(h,{}^h l)\Yleft u\cdot d \bigr)+ m\blackdiamond \left(l\blackdiamond f(h,{}^um)\right)\Yleft t\cdot d \\
&+ m\blackdiamond \left(l\blackdiamond f(h,{}^hl) \Yleft m\right)\Yleft t\cdot d - m\blackdiamond \left(u\blackdiamond \alpha(h,{}^h l)\Yleft m\right) \Yleft t\cdot d -t\blackdiamond \alpha(u,{}^um)\Yleft t\cdot d\\
& = -v\blackdiamond \alpha(l,{}^lm)\Yleft t\cdot d + v\blackdiamond f(h,{}^hv) \Yleft t\cdot d -t\blackdiamond \alpha(h,{}^hv)\Yleft t\cdot d.
\end{aligned}
\end{equation}
Note now that, by the $\alpha$-quasi linearity of $f$ in the second variable, we have
\begin{equation*}\label{f de a y R}
f(h,{}^um)+\alpha(l,{}^lm)-f(h,{}^um+{}^hl)=-f(h,{}^hl)+h\blackdiamond \alpha({}^hl,{}^um).
\end{equation*}
Since $v\blackdiamond f(h,{}^um) = m\blackdiamond \bigl(b\blackdiamond f(h,{}^um)\bigr)$ and ${}^um+{}^hl = {}^h v$, this implies that
\begin{equation*}
m\blackdiamond \bigl(l\blackdiamond f(h,{}^um)\bigr)+v\blackdiamond \alpha(l,{}^lm)-v\blackdiamond f(h,{}^hv) = - v\blackdiamond f(h,{}^hl) + t\blackdiamond \alpha({}^hl,{}^um).
\end{equation*}
Using this and that $\blackdiamond$ is linear in the second variable, we conclude that~\eqref{asociativo equivalente} is equivalent to
\begin{align*}
t\blackdiamond \alpha(h,{}^h v)&\Yleft t\cdot d = -m\blackdiamond\bigl(l\blackdiamond f(h,{}^h l) \Yleft u\cdot d\bigr)+m\blackdiamond\bigl(u\blackdiamond \alpha(h,{}^hl)\Yleft u\cdot d\bigr)- m\blackdiamond \bigl(l\blackdiamond f(h,{}^hl) \Yleft m\bigr) \Yleft t\cdot d\\
&+m\blackdiamond\bigl(u\blackdiamond \alpha(h,{}^h l)\Yleft m\bigr) \Yleft t\cdot d +t\blackdiamond \alpha(u,{}^um)\Yleft t\cdot d+v\blackdiamond f(h,{}^h l) \Yleft t\cdot d -t\blackdiamond \alpha({}^hl,{}^um)\Yleft t\cdot d.
\end{align*}
We claim that the sum of the three terms which have $f$ vanishes. In fact, since
$$
v\blackdiamond f(h,{}^hl) \Yleft t\cdot d =m\blackdiamond\bigl(l\blackdiamond f(h,{}^h l)\bigr) \Yleft t\cdot d,
$$
in order to check this, we must prove that
$$
m\blackdiamond\bigl(y\Yleft u\cdot d\bigr)+ m\blackdiamond (y \Yleft m) \Yleft t\cdot d - (m\blackdiamond y) \Yleft t\cdot d = 0,
$$
where $y\coloneqq l\blackdiamond f(h,{}^h l)$. But this is true by~\eqref{triangulo compatible con times} with $h$ replaced by $m$ and $l$ replaced by $u\cdot d$ (note that $m\cdot (u\cdot d)=t\cdot d$). Hence~\eqref{asociativo equivalente} becomes
\begin{equation}\label{asociativo equivalente 3}
\begin{aligned}
t\blackdiamond \alpha(h,{}^h v)\Yleft t\cdot d & = m\blackdiamond\bigl(u\blackdiamond \alpha(h,{}^hl)\Yleft u\cdot d\bigr) + m\blackdiamond\bigl(u\blackdiamond \alpha(h,{}^h l)\Yleft m\bigr) \Yleft t\cdot d\\
&+t\blackdiamond \alpha(u,{}^um)\Yleft t\cdot d-t\blackdiamond \alpha({}^hl,{}^um)\Yleft t\cdot d.
\end{aligned}
\end{equation}
Using again~\eqref{triangulo compatible con times} with $h$ replaced by $m$, $l$ replaced by $u\cdot d$ and $y$ replaced by $u\blackdiamond \alpha(h,{}^hl)$, we can replace the first two terms on the right hand side of~\eqref{asociativo equivalente 3} by $m \blackdiamond (u\blackdiamond \alpha(h,{}^hl)) \Yleft t\cdot d=t\blackdiamond \alpha(h,{}^h l) \Yleft t\cdot d$. So, we are reduce to prove that
\begin{equation*}
t\blackdiamond \alpha(h,{}^hl) \Yleft t\cdot d + t\blackdiamond \alpha(u,{}^um)\Yleft t\cdot d - t \blackdiamond \alpha({}^hl,{}^um)\Yleft t\cdot d -t\blackdiamond \alpha(h,{}^hv)\Yleft t\cdot d = 0.
\end{equation*}
Since $\Yleft$ is linear in the first argument and $\blackdiamond$ is linear in the second argument, for this it suffices to prove
$$
0 = \alpha({}^h l,{}^um) -\alpha(u,{}^um)+\alpha(h,{}^hv)-\alpha(h,{}^hl).
$$
But, since $u=h+{}^u l$ and ${}^h v={}^h l+{}^u m$, this equality is true by~\eqref{beta es cociclo} with $l$ replaced by ${}^hl$ and $m$ replaced by ${}^um$.
\end{proof}

\begin{corollary}\label{coro asociativo} The map $f\colon \M_Y\times H\to I$, introduced in Definition~\ref{def de f}, factorizes through $\Se_Y\times H$, where $\Se_Y$ is the free semigroup generated by $Y$.
\end{corollary}

\begin{notation} We will also write $f\colon \Se_Y\times H\to I$ to denote the map induced by $f\colon \M_Y\times H\to I$.
\end{notation}

Let $\alpha = \alpha_{\gamma_1,\dots,\gamma_n}$ be as at the beginning of this section and let $f_{11}\colon Y\times Z\to I$ be a map such that
$$
d_i f_{11}(a_j,e_i) +\sum_{k=1}^{d_i-1}\alpha(a_j\cdot e_i,k (a_j\cdot e_i))-a_j\blackdiamond \gamma_i = 0\quad\text{for $1\le i\le n$ and $1\le j\le s$,}
$$
and let $f\colon \Se_Y\times H\to I$ be the map obtained in Corollary~\ref{coro asociativo}.

\begin{proposition}\label{factorizacion} The map $f$ factorizes through $H\times H$ if and only if
\begin{align}
&F_1(a_j, a_k, e_i) = F_1(a_k, a_j, e_i) &&\! \text{for all $i,j,k$}\label{conmutatividad}\\
\shortintertext{and}
&\sum_{\ell=0}^{r_j-1}\! a_j^{\times (r_j-1-\ell)}\blackdiamond f_1(a_j,a_j^{\times \ell}\cdot e_i)+\!\sum_{\ell=1}^{r_j-1}\! a_j^{\times \ell}\blackdiamond f_1(a_j,\bar{a}_{j\ell})\Yleft e_i - \! \sum_{\ell=1}^{r_j-1}\! a_j^{\times (\ell+1)}\blackdiamond \alpha(a_j,\bar{a}_{j\ell})\Yleft e_i = 0, \!&&\text{for all $i,j$,}\label{condicion en times con e_i}
\end{align}
where $\bar{a}_{j\ell}$ is as in Theorem~\ref{ppal1}, $F_1$ is as in Definition~\ref{def de f} and $f_1$ is as in Notation~\ref{factorizacion de f_1}.
\end{proposition}

\begin{proof} Clearly,~\eqref{conmutatividad} holds if and only if the $f$ factorizes through $\AS_Y\times H$. We claim that, for all~$k\in \mathds{N}$, $h\in \AS_Y$ and~$m\in H$,
\begin{equation}\label{f de times 2}
f(h^{\times k},m)=\sum_{\ell=0}^{k-1} h^{\times (k-\ell-1)}\blackdiamond f(h,h^{\times \ell}\cdot m)+\sum_{\ell=1}^{k-1} h^{\times \ell}\blackdiamond f(h,)\Yleft h^{\times k}\cdot m -\sum_{\ell=1}^{k-1} h^{\times (\ell+1)}\blackdiamond \alpha(h,\bar{h}_{\ell})\Yleft h^{\times k}\cdot m,
\end{equation}
where $\bar{h}_{\ell}\coloneqq h^{\times (\ell+1)} - h={}^h(h^{\times \ell})\in H$. In fact, by Remark~\ref{def de F(a,b,c)}, the inductive hypothesis, and \eqref{accion tines con punto}, we have
\begin{align*}
f(h^{\times (k+1)},m) &= f(h^{\times k},h\cdot m)+ h^{\times k}\blackdiamond f(h,m)+  h^{\times k}\blackdiamond f(h,\bar{h}_k)\Yleft h^{\times(k+1)}\cdot m -h^{\times(k+1)}\blackdiamond \alpha(h,\bar{h}_k)\Yleft h^{\times(k+1)}\cdot m\\
%
%
&= \sum_{\ell=1}^k h^{\times (k-\ell)}\blackdiamond f(h,h^{\times \ell}\!\cdot m)+\sum_{\ell=1}^{k-1} h^{\times \ell}\blackdiamond f(h,\bar{h}_{\ell})\Yleft h^{\times (k+1)}\!\cdot m -\sum_{\ell=1}^{k-1} h^{\times (\ell+1)}\blackdiamond \alpha(h,\bar{h}_{\ell}) \Yleft h^{\times (k+1)}\!\cdot m\\
& + h^{\times k}\blackdiamond f(h,m)+ h^{\times k}\blackdiamond f(h,\bar{h}_k)\Yleft h^{\times(k+1)}\cdot m-h^{\times (k+1)}\blackdiamond \alpha(h,\bar{h}_k)\Yleft h^{\times(k+1)}\cdot m\\
&= \sum_{\ell=0}^k h^{\times (k-\ell)}\blackdiamond f(h,h^{\times \ell}\cdot m)+\sum_{\ell=1}^k h^{\times \ell}\blackdiamond f(h,\bar{h}_{\ell}) \Yleft h^{\times (k+1)}\cdot m -\sum_{\ell=1}^k h^{\times(\ell+1)} \alpha(h,\bar{h}_{\ell})\Yleft h^{\times (k+1)}\cdot m,
\end{align*}
which finishes the proof of the claim. Hence, condition~\eqref{condicion en times con e_i} holds if and only if $f(a_j^{\times r_j},e_i)=0$, for all $i$ and $j$ (use that $a_j^{\times r_j}\cdot e_i = e_i$). This proves that condition~\eqref{condicion en times con e_i} is necessary. In order to prove that it is sufficient, we must show that~\eqref{condicion en times con e_i} implies that $f(a_j^{\times r_j},m)=0$, for all $m\in H$. But this is true, since the map $f$ is $\alpha$-quasi linear in the second variable (by Proposition~\ref{quasi linear invariate por times}), and $\pi_Y(a_j^{\times r_j})=0$ (where $\pi_Y\colon AS_Y\to H$ is the canonical projection).
\end{proof}

\begin{notation} We will also write $f\colon H\times H\to I$ to denote the map induced by $f\colon \M_Y\times H\to I$.
\end{notation}

\begin{theorem}\label{teorema central} If we are under the conditions of Proposition~\ref{factorizacion}, then $(\alpha,-f)$ is a $2$-cocycle of $(\wh{C}_N^*(H,I),\partial+D)$.
\end{theorem}

\begin{proof} Since $\partial_{\mathrm{v}}^{03}(\alpha) = 0$, we only must prove that $\partial_{\mathrm{h}}^{12}(\alpha)=\partial_{\mathrm{v}}^{12}(f)$ and $(\partial_{\mathrm{h}}^{21}+D_{11}^{21})(f)=D_{02}^{21}(\alpha)$. In other words, that the equalities~\eqref{beta quasi linear} and~\eqref{delta horizontal es cero} are satisfied. The first identity holds by Proposition~\ref{quasi linear invariate por times}; and the second one, by Remark~\ref{def de F(a,b,c)} and Proposition~\ref{equivalencia cociclo con times}.
\end{proof}

\begin{theorem}\label{cobordes Bne0} A $2$-cocycle $(\alpha,-f)$ of $(\wh{C}_N^*(H,I),\partial+D)$ is a $2$-coboundary if and only if there exist $t_1,\dots,t_n\in I$, such that for all $i,j$, we have
\begin{equation}\label{coborde actua B}
d_i t_i = \gamma_i\qquad \text{and}\qquad  f(a_j,e_i) = - \sum_{k=1}^n \Gamma_{ij}^k t_k + a_j\blackdiamond t_i+\sum_{k=1}^n \lambda_{kj} e_i\blackdiamond t_k \Yleft a_j\cdot e_i,
\end{equation}
where $a_j=\sum_{i=1}^n \lambda_{ij} e_i$ and $a_j\cdot e_i=\sum_{k=1}^n\Gamma_{ij}^k e_k$.
\end{theorem}

\begin{proof} By definition, $(\alpha,-f)$ is a $2$-coboundary if and only if there exist $t\in \wh{C}^{01}_N(H,I)$ such that
\begin{equation}\label{cond0}
\partial^{02}_{\mathrm{v}}(t)=\alpha\qquad\text{and}\qquad (\partial_{\mathrm{h}}^{11}+D_{01}^{11})(t) = -f.
\end{equation}
We claim that $\partial^{02}_{\mathrm{v}}(t)=\alpha$ if and only if
\begin{equation}\label{t debe cumplir}
\gamma_i=d_i t(e_i) \qquad\text{and}\qquad  t\left(\sum_{i=1}^n\lambda_i e_i\right)=\sum_{i=1}^n \lambda_i t(e_i)\quad\text{for $0\le \lambda_i < d_i$.}
\end{equation}
Assume first that $\partial^{02}_{\mathrm{v}}(t)=\alpha$. Then
\begin{align*}
& t(e_i)-t(k e_i)+t((k-1)e_i)=\partial^{02}_{\mathrm{v}}(t)(e_i,(k-1) e_i)=\alpha(e_i,(k-1) e_i)=0 &&\text{for $0<k<d_i$}\\
\shortintertext{and}
& t(e_i)-t(d_i e_i)+t((d_i-1)e_i)=\partial^{02}_{\mathrm{v}}(t)(e_i,(d_i-1) e_i)=\alpha(e_i,(d_i-1) e_i)=\gamma_i.
\end{align*}
Hence, by an inductive argument, $t(ke_i)=kt(e_i)$, for $k<d_i$, and so
$$
d_it(e_i) = t(d_ie_i) +\gamma_i = \gamma_i,
$$
since $t(d_ie_i)= t(0) = 0$. Again an inductive argument using that
$$
t(\lambda_l e_l)-t\left(\sum_{i=1}^{l}\lambda_ie_i\right)+t\left(\sum_{i=1}^{l-1}\lambda_ie_i\right)= \partial^{02}_{\mathrm{v}}(t)\left(\lambda_l e_l,\sum_{i=1}^{l-1}\lambda_ie_i\right) =\alpha\left(\lambda_l e_l,\sum_{i=1}^{l-1}\lambda_ie_i\right)=0,
$$
gives us
$$
t\left(\sum_{i=1}^n\lambda_ie_i\right)=\sum_{i=1}^n t(\lambda_ie_i)=\sum_{i=1}^n \lambda_it(e_i),
$$
which finishes the proof of~\eqref{t debe cumplir}. Conversely, if the map $t$ satisfies~\eqref{t debe cumplir}, then
\begin{equation}\label{cond1}
\begin{aligned}
\partial^{02}_{\mathrm{v}}(t)\left(\sum_{i=1}^n\lambda_ie_i,\sum_{i=1}^n\lambda'_ie_i\right) & = t\left(\sum_{i=1}^n\lambda'_ie_i\right) - t\left(\sum_{i=1}^n(\lambda_i+\lambda'_i)e_i\right) + t\left(\sum_{i=1}^n\lambda_ie_i\right)\\
& = \sum_{i=1}^n\lambda'_it(e_i) - \sum_{i=1}^n r_{d_i}(\lambda_i+\lambda'_i) t(e_i) + \sum_{i=1}^n\lambda_it(e_i)\\
& = \alpha\left(\sum_{i=1}^n\lambda_ie_i,\sum_{i=1}^n\lambda'_ie_i\right).
\end{aligned}
\end{equation}
Thus, the claim is true. Now, assuming that $t$ satisfies~\eqref{t debe cumplir}, we compute
\begin{align*}
(\partial_{\mathrm{h}}^{11}+D_{01}^{11})(t)(a_j,e_i) & = t(a_j\cdot e_i)-a_j\blackdiamond t(e_i)-e_i\blackdiamond t(a_j)\Yleft a_j\cdot e_i\\
& = t\left(\sum_{k=1}^n\Gamma_{ij}^k e_k\right) - a_j\blackdiamond t(e_i) - e_i \blackdiamond t\left(\sum_{k=1}^n\lambda_{kj} e_k\right)\Yleft a_j\cdot e_i\\
& = \sum_{k=1}^n\Gamma_{ij}^k t(e_k) - a_j\blackdiamond t(e_i) - \sum_{k=1}^n \lambda_{kj} e_i\blackdiamond t(e_k)\Yleft a_j\cdot e_i.
\end{align*}
By Theorem~\ref{ppal1}, this and~\eqref{cond1} imply that~\eqref{cond0} hold if and only if~\eqref{coborde actua B} is satisfied with $t_k\coloneqq t(e_k)$.
\end{proof}

\subsubsection[Computing the cohomology]{Computing the cohomology}

Given $((\gamma_i),(\mathfrak{f}_{ji}))\in I^n\oplus I^{sn}$, set $f_{11}(a_j,e_i)\coloneqq \mathfrak{f}_{ji}$ and define $f_1\colon Y\times \AS_Z\to I$, proceeding as in Definition~\ref{def de f_{1r}}, with $\alpha\coloneqq \alpha_{\gamma_1,\dots,\gamma_n}$ (alternatively, you can use formulas~\eqref{quasilinealij 1} and~\eqref{quasilinealij 2}). Then, define the linear maps
$$
T_1,T_2\colon I^n\oplus I^{sn}\to I^{ns},\quad T_3\colon I^n\oplus I^{sn}\to I^{n\binom{s}{2}}\quad\text{and}\quad S\colon I^n\to I^n\oplus I^{sn},
$$
by
\begin{align*}
& T_1((\gamma_i),(\mathfrak{f}_{ji}))\coloneqq \left(d_i \mathfrak{f}_{ji}+\sum_{k=1}^{d_i-1}\alpha(a_j\cdot e_i,k (a_j\cdot e_i))-a_j\blackdiamond \gamma_i\right)_{ij},\\
& T_2((\gamma_i),(\mathfrak{f}_{ji}))\coloneqq \left(\sum_{\ell=0}^{r_j-1} a_j^{\times (r_j-1-\ell)}\blackdiamond f_1(a_j,a_j^{\times \ell}\cdot e_i)+\sum_{\ell=1}^{r_j-1} a_j^{\times \ell}\blackdiamond f_1(a_j,\bar{a}_{j\ell})\Yleft e_i -\sum_{\ell=1}^{r_j-1} a_j^{\times (\ell+1)}\blackdiamond \alpha(a_j,\bar{a}_{j\ell})\Yleft e_i\right)_{ij},\\
& T_3((\gamma_i),(\mathfrak{f}_{ij}))\coloneqq (F_1(a_j,a_i,e_k)-F_1(a_i,a_j,e_k))_{i<j,k},
\shortintertext{and}
& S(t_1,\dots,t_n)\coloneqq \left((d_i t_i)_i,\left(- \sum_{k=1}^n \Gamma_{ij}^k t_k +a_j\blackdiamond  t_i+\sum_{k=1}^n \lambda_{kj} (e_i\blackdiamond t_k)\Yleft a_j\cdot e_i \right)_{ij}\right),
\end{align*}
where $a_j=\sum_{i=1}^n \lambda_{ij} e_i$, $a_j\cdot e_i=\sum_{k=1}^n\Gamma_{ij}^k e_k$ and $F_1$ is as in Definition~\ref{def de f}. Let
\begin{equation*}
T\colon \ker(T_1)\cap \ker(T_2)\cap \ker(T_3) \to \wh{C}^{02}_N(H,I)\oplus \wh{C}^{11}_N(H,I)
\end{equation*}
be the linear map given by $T((\gamma_i),(\mathfrak{f}_{ji})) \coloneqq \bigl(\alpha_{\gamma_1,\dots,\gamma_n},-f_{(\mathfrak{f}_{ji}),(\gamma_i)}\bigr)$, where $f_{(\mathfrak{f}_{ji}),(\gamma_i)}\colon H\times H\to I$ is the map constructed from $f_1$, using the method introduced in Definition~\ref{def de f}.

\begin{corollary}\label{calculo de la homo B'} The map $T$ induces an isomorphism
\begin{equation*}
\ov{T}\colon \frac{\ker(T_1)\cap \ker(T_2)\cap \ker(T_3)}{\ima(S)}\longrightarrow \Ho^2_{\blackdiamond,\Yleft}(H,I).
\end{equation*}
\end{corollary}

\begin{proof} This follows immediately from Corollary~\ref{coro casi dependencia de f}, Proposition~\ref{factorizacion} and Theorems~\ref{teorema central} and~\ref{cobordes Bne0}.
\end{proof}

\begin{remark}\label{T2, T3 y S cuando Ylef=0} If $\Yleft=0$, then the maps $T_2$,~$T_3$ and $S$ simplify to
\begin{align*}
& T_2((\gamma_i),(\mathfrak{f}_{ji}))\coloneqq \left(\sum_{\ell=0}^{r_j-1} a_j^{\times (r_j-1-\ell)}\blackdiamond f_1(a_j,a_j^{\times \ell}\cdot e_i) \right)_{ij},\\
& T_3((\gamma_i),(\mathfrak{f}_{ji}))\coloneqq \bigl(f_1(a_i,a_j\cdot e_k) + a_i\blackdiamond f_1(a_j,e_k)- f_1(a_j,a_i\cdot e_k) - a_j\blackdiamond f_1(a_i,e_k)\bigr)_{i<j,k}\\
\shortintertext{and}
& S(t_1,\dots,t_n)\coloneqq \left((d_i t_i)_i,\left(-\sum_{k=1}^n \Gamma_{ij}^k t_k +a_j\blackdiamond t_i\right)_{ij}\right).
\end{align*}
Moreover the recursive construction of the map $f$ becomes
$$
f_{r+1}(h,d) = \begin{cases} f_r(h,d) & \text{if $h\in \ov{Y}_r $,}\\ f_r(m,l\cdot d)+ m\blackdiamond f_r(l,d) & \text{if $h = l\times m$, with $l\in Y_u$, $m\in Y_v$ and $u+v = r+1$.} \end{cases}
$$
\end{remark}

\subsection[Computing the map \texorpdfstring{$f$}{f}]{Computing the map \texorpdfstring{$\mathbf{f}$}{f}}\label{subsection 2.3}

Although in Corollary~\ref{calculo de la homo B'}, we obtained a group isomorphic to $\Ho^2_{\blackdiamond,\Yleft}(H,I)$, for the explicit calculation of extensions, given $((\gamma_i),(\mathfrak{f}_{ji}))\in \ker(T_1)\cap \ker(T_2)\cap \ker(T_3)$, we need to compute $f\coloneqq f_{(\mathfrak{f}_{ji}),(\gamma_i)}$. In this subsection, we provide a method to carry out this task.

\smallskip

Let $(\alpha,-f) \in \wh{C}^{02}_N(H,I)\oplus \wh{C}^{11}_N(H,I)$, be a $2$-cocycle of$(\wh{C}_N^*(H,I),\partial+D)$, with $\alpha\coloneqq \alpha_{\gamma_1,\dots,\gamma_n}$. In Proposition~\ref{delta nos da times}, we obtained a formula that allows us to compute $f(a_j^{\times k},m)$ in terms of the restriction of $f$ to $\{a_j\}\times H$. Moreover, by identities~\eqref{quasilinealij 1} and~\eqref{quasilinealij 2}, this restriction is computed by the formula
\begin{equation}\label{zazaza}
f\left(a_j,\sum_{i=1}^n \lambda_i e_i\right)= \sum_{i=1}^n \lambda_i f(a_j,e_i)+\sum_{i=1}^n \sum_{\ell=1}^{\lambda_i-1}\alpha(a_j\cdot e_i,\ell(a_j\cdot e_i)) +\sum_{k=2}^n \alpha\left(\sum_{i=1}^{k-1}\lambda_i (a_j\cdot e_i), \lambda_k (a_j \cdot e_k)\right),
\end{equation}
valid when $0\le \lambda_i<d_i$, for all $i$. Combining this with the following result, we can compute $f(a_1^{\times k_1}\times\cdots\times a_1^{\times k_1},m)$ in terms of $\alpha$ and the restriction of $f$ to~$Y\times Z$, where $Y\coloneqq \{a_1,\dots,a_s\}$ and $Z\coloneqq \{e_1,\dots,e_n\}$.

\begin{proposition}\label{calculo de f 1} We have:
\begin{align*}
f(h_1\times\cdots\times h_{\ell},m) & = \sum_{j=1}^{\ell} (h_{j+1}\times \cdots \times h_{\ell})\blackdiamond f(h_j,(h_1\times\cdots \times h_{j-1})\cdot m)\\
& + \sum_{j=1}^{\ell-1} (h_{j+1}\times \cdots \times h_{\ell}) \blackdiamond f\bigl(h_j,{}^{h_j}(h_{j+1}\times \cdots \times h_{\ell})\bigr)\Yleft (h_1\times \cdots \times h_{\ell}) \cdot m\\
& + \sum_{j=1}^{\ell-1} (h_j\times \cdots \times h_{\ell}) \blackdiamond \alpha\bigl(h_j,{}^{h_j}(h_{j+1}\times \cdots \times h_{\ell})\bigr)\Yleft (h_1\times \cdots \times h_{\ell}) \cdot m.
\end{align*}
\end{proposition}

\begin{proof} We proceed by induction on $\ell$. The case $\ell = 1$ is trivial. Assume $\ell\ge 1$ and the formula holds for $\ell$. Then, by Remark~\ref{def de F(a,b,c)},
\begin{align*}
f(h_1\times\cdots\times h_{\ell+1},m) & = f(h_2\times\cdots\times h_{\ell+1},h_1\cdot m)+ (h_2\times \cdots\times h_{\ell+1})\blackdiamond f(h_1,m)\\
& + (h_2\times\cdots\times h_{\ell+1})\blackdiamond f(h_1,{}^{h_1}(h_2\times\cdots\times h_{\ell+1})) \Yleft (h_1\times\cdots\times h_{\ell+1})\cdot m\\
& - (h_1\times\cdots\times h_{\ell+1})\blackdiamond \alpha(h_1,{}^{h_1}(h_2\times\cdots\times h_{\ell+1})) \Yleft (h_1\times\cdots\times h_{\ell+1})\cdot m.
\end{align*}
The formula for $\ell+1$ follows from this, applying the inductive hypothesis to $f(h_2\times\cdots\times h_{\ell+1},h_1\cdot m)$ and taking into account the identity~\eqref{accion tines con punto}.
\end{proof}

\section[The trivial case]{The trivial case}\label{section 3}
In this section we use the notations of the previous section, but we assume that $H = \mathds{Z}_{d_1}\oplus\cdots\oplus \mathds{Z}_{d_n}$ is trivial. That is, that $h\cdot l=l$ for all $h,l\in H$. From this, it follows that $\times=+$, $s=n$ and $a_i=e_i$, for all $i$. Let $\ov{a}_{j\ell}$ be as in Theorem~\ref{ppal1}. We will use that $\ov{e}_{j\ell}=\ell e_j$, since $\times=+$.

\subsection[The actions \texorpdfstring{$\blackdiamond$}{*} and \texorpdfstring{$\Yleft$}{<}]{The actions \texorpdfstring{$\blackdiamond$}{*} and \texorpdfstring{$\pmb\Yleft$}{<}}\label{subsection 3.1}
Let $I$ be an abelian group. In this subsection we characterize the actions $\blackdiamond$ and $\Yleft$, from $H$ on  $I$, satisfying the~condi\-tions~\eqref{eqq 1}--\eqref{eqq 3}.

\begin{proposition}\label{caracterizacion} Let $A_1,\dots,A_n,B_1,\dots,B_n$ be endomorphisms of $I$ such that, for all $1\le i\le n$ and $1\le i<j\le n$,
\begin{align}
&	A_i^{d_i}=\Ide,\quad [A_i,A_j]=0,\quad d_iB_i=0,\label{propiedades de A_i y B_i}\\
&   (A_i-A_i\xcirc B_i)^{d_i}=\Ide,\quad [A_i-A_i\xcirc B_i, A_j-A_j\xcirc B_j]=0,\label{propiedades de A_i-A_iB_i}\\
&	[B_i,A_i]=B_i\xcirc A_i\xcirc B_i\quad\text{and}\quad [B_i,A_j]=B_i\xcirc A_j\xcirc B_j.\label{condiciones corchete A_i, B_j}
\end{align}
Then, the maps $\blackdiamond\colon H\times I\to I$ and $\Yleft\colon I\times H\to I$, given by
\begin{equation}\label{acciones caso trivial}
(\lambda_1e_1+\cdots +\lambda_ne_n)\blackdiamond y\coloneqq A_1^{\lambda_1}\xcirc \cdots \xcirc A_n^{\lambda_n}y \quad\text{and}\quad y\Yleft (\lambda_1e_1+\cdots +\lambda_ne_n)\coloneqq  \lambda_1B_1 y + \cdots +\lambda_nB_n y,
\end{equation}
satisfy the conditions~\eqref{eqq 1}--\eqref{eqq 3}. Conversely, if $\blackdiamond \colon H\times I\to I$ and $\Yleft \colon I\times H\to I$ are maps satisfying these~condi\-tions, then the endomorphisms $A_1,\dots,A_n, B_1,\dots,B_n$, of $I$, given by
\begin{equation}\label{eq 5}
A_i y\coloneqq e_i\blackdiamond y\quad\text{and}\quad B_i y \coloneqq y\Yleft e_i,
\end{equation}
satisfy~\eqref{propiedades de A_i y B_i}--\eqref{condiciones corchete A_i, B_j}.
\end{proposition}

\begin{proof} $\Rightarrow$)\enspace Since $A_1^{d_1}=\cdots = A_n^{d_n}=\Ide$ and $[A_i,A_j]=0$, for $1\le i<j\le n$, the map is well defined and $\blackdiamond$ satisfies conditions~\eqref{eqq 2}. Similarly, since conditions~\eqref{propiedades de A_i-A_iB_i} are satisfied, the map $\triangleleft\colon I\times H\to I$, given by
$$
y\triangleleft (\lambda_1e_1+\cdots +\lambda_ne_n)\coloneqq (A_1-A_1\xcirc B_1)^{\lambda_1}\cdots (A_n-A_n\xcirc B_n)^{\lambda_n} y,
$$
is well defined and satisfies conditions~\eqref{eqq 1}. Moreover, the map $\Yleft$ is well defined, since $d_iB_i=0$, and a direct~compu\-tation proves that the operation
$$
y^{(\lambda_1e_1+\cdots +\lambda_ne_n)}\coloneqq y - y\Yleft (\lambda_1e_1+\cdots +\lambda_ne_n)=(\Ide - \lambda_1 B_1-\cdots -\lambda_n B_n)y,
$$
satisfies~\eqref{eqq 3}. To finish the proof, it remains to check that
$$
(\lambda_1e_1+\cdots +\lambda_ne_n)\blackdiamond (y - y\Yleft (\lambda_1e_1+\cdots +\lambda_ne_n)) = y\triangleleft (\lambda_1e_1+\cdots +\lambda_ne_n),
$$
for $y\in I$ and $\lambda_1e_1+\cdots +\lambda_ne_n\in H$. Concretely, we must show that
\begin{equation}
A_1^{\lambda_1}\xcirc\cdots\xcirc  A_n^{\lambda_n}(\Ide -\lambda_1B_1-\cdots - \lambda_nB_n)= (A_1-A_1\xcirc B_1)^{\lambda_1}\cdots (A_n-A_n\xcirc B_n)^{\lambda_n}, \label{concretamente}
\end{equation}
for all $\lambda_1e_1+\cdots+\lambda_ne_n\in H$. For this we prove by induction on $k\le n$, that
\begin{equation}
A_1^{\lambda_1}\xcirc\cdots\xcirc A_k^{\lambda_k}(\Ide -\lambda_1B_1-\cdots - \lambda_kB_k)= (A_1-A_1\xcirc B_1)^{\lambda_1}\xcirc\cdots\xcirc (A_k-A_k\xcirc B_k)^{\lambda_k}, \label{concretamente con k}
\end{equation}
for $\lambda_1e_1+\cdots+\lambda_ke_k\in H$. Assume that $k=1$. To check~\eqref{concretamente con k}, we must verify~that
\begin{equation}\label{case k=1}
(A_1-A_1\xcirc B_1)^{\lambda}= A_1^{\lambda}\xcirc (\Ide - \lambda B_1)= A_1^{\lambda} - \lambda A_1^{\lambda}\xcirc B_1,\quad\text{for all $\lambda\in \mathds{Z}$.}
\end{equation}
This is trivially true, when $\lambda=0$ and $\lambda=1$. Assume that it is true for some $\lambda\ge 1$. Then, since
$$
B_1\xcirc (A_1 - A_1\xcirc B_1)= A_1\xcirc B_1,
$$
we have
\begin{align*}
(A_1-A_1\xcirc B_1)^{\lambda+1}&=\left(A_1^{\lambda}-{\lambda}A_1^{\lambda}\xcirc B_1\right)\xcirc (A_1-A_1\xcirc B_1)\\
&=A_1^{\lambda+1}-A_1^{\lambda+1}\xcirc B_1-\lambda A_1^\lambda\xcirc B_1\xcirc (A_1-A_1\xcirc B_1)\\
&= A_1^{\lambda+1}-A_1^{\lambda+1}\xcirc B_1-\lambda A_1^{\lambda}\xcirc A_1\xcirc B_1\\
&= A_1^{\lambda+1}-(\lambda+1) A_1^{\lambda+1}\xcirc B_1,
\end{align*}
and so, \eqref{case k=1} holds, for all $\lambda\in \mathds{N}_0$. In order to conclude that it holds, for all $\lambda\in \mathds{Z}$, it suffices to note that
$$
(A_1-A_1\xcirc B_1)^{d_1} = \Ide = A_1^{d_1} = A_1^{d_1}-d_1 A_1^{d_1}\xcirc B_1.
$$
Assume now that~\eqref{concretamente con k} is true for some~$k<n$. Then, by the inductive hypothesis,
\begin{align*}
A_1^{\lambda_1}\xcirc\cdots\xcirc A_{k+1}^{\lambda_{k+1}}&\xcirc (\Ide -\lambda_1B_1-\cdots - \lambda_{k+1}B_{k+1}) = A_1^{\lambda_1}\xcirc\cdots\xcirc A_{k+1}^{\lambda_{k+1}}\xcirc (-\lambda_1B_1 +\Ide -\lambda_2B_2-\cdots - \lambda_{k+1}B_{k+1})\\
&= -\lambda_1A_1^{\lambda_1}\xcirc\cdots\xcirc A_{k+1}^{\lambda_{k+1}}\xcirc B_1+ A_1^{\lambda_1}\xcirc (A_2-A_2\xcirc B_2)^{\lambda_2}\xcirc\cdots\xcirc  (A_{k+1}-A_{k+1}\xcirc B_{k+1})^{\lambda_{k+1}}.
\end{align*}
But an easy induction, using that $A_i\xcirc B_1=B_1\xcirc (A_i - A_i\xcirc B_i)$, proves that
$$
\lambda_1A_1^{\lambda_1}\xcirc\cdots\xcirc A_{k+1}^{\lambda_{k+1}}\xcirc B_1= \lambda_1A_1^{\lambda_1}\xcirc B_1\xcirc (A_2-A_2\xcirc B_2)^{\lambda_2}\xcirc \cdots\xcirc (A_{k+1}-A_{k+1}\xcirc B_{k+1})^{\lambda_{k+1}},
$$
and so
\begin{align*}
A_1^{\lambda_1}\xcirc\cdots\xcirc A_{k+1}^{\lambda_{k+1}}\xcirc (\Ide -\lambda_1B_1-\cdots - \lambda_{k+1}B_{k+1})& = (A_1^{\lambda_1}- \lambda_1A_1^{\lambda_1}\xcirc B_1)\xcirc (A_2-A_2\xcirc B_2)^{\lambda_2}\xcirc\cdots\xcirc (A_{k+1}-A_{k+1}\xcirc B_{k+1})^{\lambda_{k+1}}\\
& = (A_1-A_1\xcirc B_1)^{\lambda_1}\xcirc\cdots\xcirc (A_{k+1}-A_{k+1}\xcirc B_{k+1})^{\lambda_{k+1}},
\end{align*}
as desired.
		
\smallskip
		
\noindent $\Leftarrow$)\enspace By~\eqref{eqq 2}, we have
$$
A_i^{d_i}=\Ide, \text{ for $1\le i \le n$}\quad\text{and}\quad [A_i,A_j]=0, \text{ for $1\le i<j\le n$.}
$$
Furthermore, by the bilinearity of $\Yleft$ and the definition of $B_i$, we have
$$
d_iB_i y = d_i(y\Yleft e_i) = y\Yleft d_ie_i =0.
$$
Thus, the identities~\eqref{propiedades de A_i y B_i} hold. On the other hand, since $y\triangleleft e_i=(A_i-A_i\xcirc B_i)y$, the identities in~\eqref{propiedades de A_i-A_iB_i} follows directly from~\eqref{eqq 1}, and the fact that $d_ie_i=0$ and $e_i+e_j=e_j+e_i$. It remains to check the identities in~\eqref{condiciones corchete A_i, B_j}, but, since these follow easily from the identity~\eqref{concretamente} with $\lambda_1e_1+\cdots +\lambda_ne_n=2e_i$ and with $\lambda_1e_1+\cdots +\lambda_ne_n=e_i+e_j$, it suffices to prove~\eqref{concretamente}. But this is true, because
\begin{align*}
A_1^{\lambda_1}\xcirc\cdots\xcirc A_n^{\lambda_n}\xcirc (\Ide -\lambda_1B_1-\cdots - \lambda_nB_n)y&= (\lambda_1e_1+\cdots +\lambda_ne_n)\blackdiamond (y - \lambda_1  y\Yleft e_1-\cdots -\lambda_n  y\Yleft e_n)\\
&= (\lambda_1e_1+\cdots +\lambda_ne_n)\blackdiamond (y - y\Yleft (\lambda_1e_1+\cdots +\lambda_ne_n))\\
&= y\triangleleft (\lambda_1e_1+\cdots +\lambda_ne_n)\\
&=(A_1-A_1\xcirc B_1)^{\lambda_1}\xcirc\cdots\xcirc (A_n-A_n\xcirc B_n)^{\lambda_n}y,
\end{align*}
where the first equality holds by the definitions of the $A_i$ and $B_j$; the second one, since $\Yleft$ is bilinear; the third one, by the definition of~$\triangleleft$; and the last one, by the first identity in~\eqref{eqq 1}.
\end{proof}

\begin{remark}\label{alpha, etc} Le $I$ be an abelian group, let $\blackdiamond\colon H\times I\to I$ and $\Yleft \colon I\times H\to I$ be maps satisfying coinditions~\eqref{eqq 1}--\eqref{eqq 3}, and let $\alpha = \alpha_{\gamma_1,\dots, \gamma_n}\in \wh{C}^{02}_N(H,I)$ be as above of Corollary~\ref{basta tomar combinacion de estandar}. By Proposition~\ref{dependencia de f}, the fact that $H$ is trivial and the definition of $\alpha$, if $(\alpha,-f)\in \wh{C}^{02}_N(H,I)\oplus \wh{C}^{11}_N(H,I)$ is a $2$-cocycle of the complex $(\wh{C}_N^*(H,I),\partial+D)$, then
$$
f\left(h,\sum_{i=1}^n \lambda_i e_i\right) = \sum_{i=1}^n \lambda_i f(h,e_i)\qquad\text{and}\qquad d_i f(h,e_i)+\gamma_i-h\blackdiamond \gamma_i = 0,
$$
for all $h\in H$, $1\le i \le n$ and $0\le \lambda_i< d_i$.
\end{remark}

\begin{remark}\label{T_1, T2, T3 y S cuando H es trivial} By the fact that $H$ is trivial, the definition of $\alpha_{\gamma_1,\cdots,\gamma_n}$ and the previous remark, the maps $T_1$, $T_2$, $T_3$ and $S$, introduced above Corollary~\ref{calculo de la homo B'}, reduce to
\begin{align*}
& T_1((\gamma_i),(\mathfrak{f}_{ji})) = \left(d_i \mathfrak{f}_{ji}+\gamma_i-A_j\gamma_i\right)_{ij},\\
& T_2((\gamma_i),(\mathfrak{f}_{ji}))\coloneqq \left(\sum_{\ell=0}^{d_j-1} A_j^{r_j-1-\ell} \mathfrak{f}_{ji} +B_i\xcirc \sum_{\ell=1}^{d_j-1} \ell A_j^{\ell} \mathfrak{f}_{jj} - B_i\gamma_j\right)_{ij},\\
& T_3((\gamma_i),(\mathfrak{f}_{ij}))\coloneqq (\mathfrak{f}_{ik} - \mathfrak{f}_{jk} + A_i \mathfrak{f}_{jk} - A_j \mathfrak{f}_{ik}+ B_k\xcirc A_i \mathfrak{f}_{ji}-B_k\xcirc A_j \mathfrak{f}_{ij}
)_{i<j,k},
\shortintertext{and}
& S(t_1,\dots,t_n)\coloneqq \bigl((d_i t_i)_i,\left(A_j t_i- t_i + B_i\xcirc A_it_j \right)_{ij}\bigr),
\end{align*}
where $A_i$ and $B_i$ are as in~\eqref{eq 5}.
\end{remark}

\subsection[Computing the map \texorpdfstring{$f$}{f} in the trivial case]{Computing the map \texorpdfstring{$\mathbf{f}$}{f} in the trivial case}\label{subsection 3.2}
Let $I$ be an abelian group, let $\blackdiamond\colon H\times I\to I$ and $\Yleft \colon I\times H\to I$ be maps satisfying coinditions~\eqref{eqq 1}--\eqref{eqq 3} and let~$(\alpha,-f) \in \wh{C}^{02}_N(H,I)\oplus \wh{C}^{11}_N(H,I)$ be a $2$-cocycle of $(\wh{C}_N^*(H,I),\partial+D)$, with $\alpha\coloneqq \alpha_{\gamma_1,\dots,\gamma_n}$. Recall that $H$ is trivial.

\smallskip

As we saw above Corollary~\ref{calculo de la homo B'}, the map $f$ is determined by the elements $\gamma_i$ and the elements $\mathfrak{f}_{ji}\coloneqq f(e_j,e_i)$.~This map is denoted by $f_{(\mathfrak{f}_{ji}),(\gamma_i)}$. The following result shows that, when $H$ is trivial, $f$ becomes independent of the~$\gamma_i$'s. Therefore, in this case, we will write $f_{(\mathfrak{f}_{ji})}$ instead of $f_{(\mathfrak{f}_{ji}),(\gamma_i)}$.

\begin{theorem}\label{formula para f} For $0\le k_1,\lambda_1<d_1,\dots,0\le k_n,\lambda_n<d_n$, the following formula holds:
\begin{align*}
f\left(\sum k_je_j,\sum \lambda_ue_u\right) & = \sum_{j=1}^n \sum_{\ell=0}^{k_j-1} \sum_{u=1}^n \lambda_u A_j^{k_j-\ell-1}\xcirc A_{j+1}^{k_{j+1}}\xcirc\cdots\xcirc A_n^{k_n}f(e_j,e_u)\\
& + \sum_{j=1}^{n-1}\sum_{\ell=0}^{k_j-1} \sum_{i=j+1}^n \sum_{u=1}^n k_i\lambda_u B_u\xcirc A_j^{k_j-\ell-1}\xcirc A_{j+1}^{k_{j+1}}\xcirc\cdots\xcirc  A_n^{k_n} f(e_j, e_i)\\
& + \sum_{j=1}^n\sum_{\ell=1}^{k_j-1}\sum_{u=1}^n \ell\lambda_u B_u\xcirc A_j^{\ell}\xcirc A_{j+1}^{k_{j+1}}\xcirc\cdots\xcirc A_n^{k_n} f(e_j,e_j),
\end{align*}
where $A_i$ and $B_i$ are as in~\eqref{eq 5}.
\end{theorem}

\begin{proof} For the sake of brevity we write $m\coloneqq \lambda_1e_1+\cdots+\lambda_ne_n$. By the definition of $\alpha$, we have
$$
\alpha\left(k_j e_j,{}^{k_j e_j}(k_{j+1} e_{j+1}+\cdots +k_ne_n) \right) = \alpha\left(k_j e_j, k_{j+1} e_{j+1}+\cdots +k_ne_n \right) = 0.
$$
Hence, by Proposition~\ref{calculo de f 1},
\begin{align*}
f(k_1e_1+\cdots+k_ne_n,m) & = \sum_{j=1}^n (k_{j+1}e_{j+1}+\cdots+k_ne_n) \blackdiamond f(k_je_j,m)\\
& + \sum_{j=1}^{n-1} (k_{j+1}e_{j+1}+ \cdots + k_n e_n) \blackdiamond f(k_je_j, k_{j+1}e_{j+1}+ \cdots + k_ne_n)\Yleft  m.
\end{align*}
On the other hand, since $\alpha(e_j,\bar{e}_{j\ell}) = \alpha(e_j,\ell e_j) = 0$, for $0\le \ell < d_j-1$, by Proposition~\ref{delta nos da times}, we have
$$
f(k e_j,h) = \sum_{\ell=0}^{k-1} (k-\ell-1)e_j\blackdiamond f(e_j,h) +\sum_{\ell=1}^{k-1} \ell e_j \blackdiamond f(e_j,\ell e_j)\Yleft h.
$$
Consequently
\begin{align*}
f(k_1e_1+\cdots+k_ne_n,m) & = \sum_{j=1}^n \sum_{\ell=0}^{k_j-1} \left((k_j-\ell-1)e_j+k_{j+1}e_{j+1}+\cdots+k_ne_n\right) \blackdiamond f(e_j,m)\\
& + \sum_{j=1}^n \sum_{\ell=1}^{k_j-1} (k_{j+1}e_{j+1}+\cdots+k_ne_n) \blackdiamond \bigl(\ell e_j\blackdiamond f(e_j,\ell e_j)\Yleft m\bigr)\\
& + \sum_{j=1}^{n-1} \sum_{\ell=0}^{k_j-1} \left((k_j-\ell-1)e_j + k_{j+1}e_{j+1}+ \cdots + k_n e_n\right) \blackdiamond f(e_j, k_{j+1}e_{j+1}+ \cdots + k_ne_n)\Yleft  m\\
& + \sum_{j=1}^{n-1}\sum_{\ell=1}^{k_j-1} (k_{j+1}e_{j+1}+ \cdots + k_n e_n) \blackdiamond \Bigl(\bigl(\ell e_j\blackdiamond f(e_j, \ell e_j)\bigr) \Yleft (k_{j+1}e_{j+1}+ \cdots + k_ne_n)\Bigr)\Yleft m\\
& = \sum_{j=1}^n \sum_{\ell=0}^{k_j-1} \left((k_j-\ell-1)e_j+k_{j+1}e_{j+1}+\cdots+k_ne_n\right) \blackdiamond f(e_j,m)\\
& + \sum_{j=1}^n \sum_{\ell=1}^{k_j-1} (k_{j+1}e_{j+1}+\cdots+k_ne_n) \blackdiamond \bigl(\ell e_j\blackdiamond f(e_j,\ell e_j)\Yleft m\bigr)\\
& + \sum_{j=1}^{n-1} \sum_{\ell=0}^{k_j-1} \left((k_j-\ell-1)e_j + k_{j+1}e_{j+1}+ \cdots + k_n e_n\right) \blackdiamond f(e_j, k_{j+1}e_{j+1}+ \cdots + k_ne_n)\Yleft  m\\
& + \sum_{j=1}^{n-1} \sum_{\ell=1}^{k_j-1} \left(\ell e_j + k_{j+1}e_{j+1}+ \cdots + k_n e_n\right) \blackdiamond f(e_j,\ell e_j) \Yleft m\\
& - \sum_{j=1}^{n-1}\sum_{\ell=1}^{k_j-1} (k_{j+1}e_{j+1}+ \cdots + k_n e_n) \blackdiamond \bigl(\ell e_j\blackdiamond f(e_j, \ell e_j) \Yleft m\bigr)\\
& = \sum_{j=1}^n \sum_{\ell=0}^{k_j-1} \left((k_j-\ell-1)e_j+k_{j+1}e_{j+1}+\cdots+k_ne_n\right) \blackdiamond f(e_j,m)\\
& + \sum_{j=1}^{n-1} \sum_{\ell=0}^{k_j-1} \left((k_j-\ell-1)e_j + k_{j+1}e_{j+1}+ \cdots + k_n e_n\right) \blackdiamond f(e_j, k_{j+1}e_{j+1}+ \cdots + k_ne_n)\Yleft  m\\
& + \sum_{j=1}^n \sum_{\ell=1}^{k_j-1} \left(\ell e_j + k_{j+1}e_{j+1}+ \cdots + k_n e_n\right) \blackdiamond f(e_j,\ell e_j) \Yleft m,
\end{align*}
where the second equality holds by~\eqref{triangulo compatible con times}. Note now that, since
\begin{align*}
&\alpha(e_j\cdot e_i,\ell(e_j\cdot e_i)) = \alpha(e_i,\ell e_i) = 0\quad\text{for $0\le \ell <d_i-1$}
\shortintertext{and}
&\alpha\left(\sum_{i=1}^{k-1}\lambda_i (e_j\cdot e_i), \lambda_k (e_j \cdot e_k)\right) = \alpha\left(\sum_{i=1}^{k-1}\lambda_i e_i, \lambda_k e_k\right) = 0,
\end{align*}
from the identity~\eqref{zazaza}, it follows that
\begin{align*}
f(k_1e_1+\cdots+k_ne_n,m) & = \sum_{j=1}^n \sum_{\ell=0}^{k_j-1} \left((k_j-\ell-1)e_j+k_{j+1}e_{j+1}+\cdots+k_ne_n\right) \blackdiamond f(e_j,m)\\
& + \sum_{j=1}^{n-1}\sum_{\ell=0}^{k_j-1} \sum_{i=j+1}^n \left((k_j-\ell-1)e_j + k_{j+1}e_{j+1}+\cdots+ k_n e_n\right)\blackdiamond k_i f(e_j, e_i)\Yleft m\\
& + \sum_{j=1}^n \sum_{\ell=1}^{k_j-1} \left(\ell e_j + k_{j+1}e_{j+1}+ \cdots + k_n e_n\right)\blackdiamond \ell f(e_j,e_j) \Yleft m.
\end{align*}
Applying again the identity~\eqref{zazaza}, and using  Proposition~\ref{caracterizacion}, and that $\blackdiamond$ and $\Yleft$ are right linear, we obtain the~formu\-la in the statement.
\end{proof}

\subsection[Some examples]{Some examples}
Assume that $\Yleft=0$ (which, by~\cite{GGV1}*{Remark~5.13}, is the condition we must require when looking for extensions of $H$ by $I$ such that $I$ is included in the socle of the extension). Then, by~\eqref{acciones caso trivial}, we have $B_1=\cdots =B_n = 0$. Consequently, conditions~\eqref{propiedades de A_i y B_i}--\eqref{condiciones corchete A_i, B_j} reduce to
$$
A_i^{d_i} = \Ide\quad\text{and}\quad [A_i,A_j] = 0\qquad\text{for all $i$ and $j$.}
$$
Moreover, a straightforward computation, using Remark~\ref{T_1, T2, T3 y S cuando H es trivial}, shows that
\begin{align*}
& T_1((\gamma_i),(\mathfrak{f}_{ji}))\coloneqq \left(d_i \mathfrak{f}_{ji}+\gamma_i - A_j\gamma_i\right)_{ij},\\
& T_2((\gamma_i),(\mathfrak{f}_{ji}))\coloneqq \left(\sum_{\ell=0}^{d_j-1}A_j^{\ell}\mathfrak{f}_{ji} \right)_{ij},\\
& T_3((\gamma_i),(\mathfrak{f}_{ij}))\coloneqq \bigl(\mathfrak{f}_{ik} + A_i \mathfrak{f}_{jk}- \mathfrak{f}_{jk} - A_j \mathfrak{f}_{ik}\bigr)_{i<j,k}\\
\shortintertext{and}
& S(t_1,\dots,t_n)\coloneqq \left((d_i t_i)_i, (A_jt_i - t_i)_{ij}\right).
\end{align*}
Moreover, by Theorem~\ref{formula para f}, we have
\begin{equation}\label{zazaza1}
f\left(\sum k_je_j,\sum\lambda_ie_i\right) = \sum_{j=1}^n\sum_{\ell=0}^{k_j-1}\sum_{i=1}^n\lambda_i A_j^{k_j-\ell-1}\xcirc A_{j+1}^{k_{j+1}}\xcirc\cdots \xcirc  A_n^{k_n} \mathfrak{f}_{ji},
\end{equation}
where $0\le k_1,\lambda_1<d_1,\dots,0\le k_n,\lambda_n<d_n$ and $f\coloneqq f_{(\mathfrak{f}_{ji})}$.

\begin{example}\label{caso trivial I= Zp} Let $p$ be a prime number. Assume that $d_i\coloneqq p^{\eta_i}$ with $\eta_i\in \mathds{N}$ and that $I\coloneqq \mathds{Z}_p$. Then
$A_i=A_i^{p^{\eta_i}}=\Ide$.
 In this case $T_1=T_2=0$, $T_3=0$ and $S=0$. Thus, $\Ho^2_{\blackdiamond}(H,I) = I^n\oplus I^{sn}$. Moreover equality~\eqref{zazaza1} becomes
\begin{equation*}
f(k_1e_1+\cdots+k_ne_n,\lambda_1e_1+\cdots+\lambda_ne_n) = \sum_{j=1}^n \sum_{i=1}^n  k_j\lambda_i f(e_j,e_i).
\end{equation*}
Note that in this case the condition $B_i=0$ is automatically fulfilled, since $\Ide-B_i=(\Ide-B_i)^{p^{\eta_i}}=\Ide$. 
\end{example}

\begin{example}\label{caso trivial I= Zp^2} Let $p$ be a prime number and let $2\le \eta_1,\eta_2$. Assume that $n=2$, $d_1\coloneqq p^{\eta_1}$, $d_2\coloneqq p^{\eta_2}$ and $I\coloneqq \mathds{Z}_{p^2}$. Then, there exist $0\le b_i<p$ such that $A_i$ is the multiplication by $1+pb_i$. Assume first that some $b_i\ne 0$. Renaming the indices if necessary we can suppose that $b_1\ne 0$. Let $0<d<p$ be such that $db_1\equiv 1\pmod{p}$, and let $\Phi\colon I^6\to I^6$ be the map given by
$$
\Phi(y_1,y_2,y_{11},y_{12},y_{21},y_{22}) = (y_1,y_2,y_{11},y_{12},y_{21}+db_2y_{11},y_{22}+db_2y_{12}).
$$
Let $\wt{T}_1\coloneqq T_1\xcirc \Phi$, $\wt{T}_2\coloneqq T_2\xcirc \Phi$, $\wt{T}_3\coloneqq T_3\xcirc \Phi$ and $\wt{S}\coloneqq \Phi^{-1}\xcirc S$. We claim that
\begin{align*}
& \wt{T}_1(y_1,y_2,y_{11},y_{12},y_{21},y_{22}) = (-pb_1y_1,-pb_2y_1,-pb_1y_2,-pb_2y_2),\\
& \wt{T}_2(y_1,y_2,y_{11},y_{12},y_{21},y_{22}) = (0,0,0,0),\\
&\wt{T}_3(y_1,y_2,y_{11},y_{12},y_{21},y_{22}) = (pb_1y_{21},pb_1y_{22})\\
\shortintertext{and}
&\wt{S}(x_1,x_2) = (0,0,pb_1x_1,pb_1x_2,0,0).
\end{align*}
The formulas for $\wt{T}_1$, $\wt{T}_3$ and $\wt{S}$ follow by a straightforward computation; while the formula for $\wt{T}_2$, holds since
\begin{equation*}
\sum_{j=0}^{p^{\eta_i}-1} (1+pb_i)^j = \sum_{j=0}^{p^{\eta_i}-1}\sum_{\ell=0}^j \binom{j}{\ell}p^{\ell}b_i^{\ell} = \sum_{\ell=0}^{p^{\eta_i}-1}\sum_{j=\ell}^{p^{\eta}-1} \binom{j}{\ell}p^{\ell}b_i^{\ell} = \sum_{\ell=0}^{p^{\eta_i}-1}\binom{p^{\eta_i}}{\ell+1} p^{\ell}b_i^{\ell} = 0,
\end{equation*}
where the last equality holds because  $p^2\mid \binom{p^{\eta_i}}{\ell+1}p^{\ell}$, for all $\ell$. Hence
\begin{equation*}
\ho_{\blackdiamond}^2(H,I) \simeq \frac{\ker \wt{T}_1\cap \ker \wt{T}_2\cap \ker \wt{T}_3}{\ima \wt{S}} = p\mathds{Z}_{p^2}\oplus p\mathds{Z}_{p^2}\oplus \frac{\mathds{Z}_{p^2}}{p\mathds{Z}_{p^2}}\oplus \frac{\mathds{Z}_{p^2}}{p\mathds{Z}_{p^2}}\oplus p\mathds{Z}_{p^2}\oplus p\mathds{Z}_{p^2}.
\end{equation*}
Assume now that $b_1=b_2= 0$. Then, $T_1=T_2=0$, $T_3=0$ and $S=0$. Thus, $\Ho^2_{\blackdiamond}(H,I) = I^2\oplus I^4$.
\end{example}

\begin{example}\label{ejemplo} Assume that $n=s=1$ and write $\gamma\coloneqq \gamma_1$, $\mathfrak{f}_0\coloneqq \mathfrak{f}_{11}$, $A\coloneqq A_1$ and $d\coloneqq d_1$. We have $e_1=a_1=1$. In this case, $T_3$ is the empty map and $T_1$, $T_2$ and $S$ become
$$
T_1(\gamma,\mathfrak{f}_0)=d\mathfrak{f}_0+(\Ide-A)\gamma,\quad T_2(\gamma,\mathfrak{f}_0)=N(A)\mathfrak{f}_0\quad\text{and}\quad   S(t)=(dt,At-t),
$$
where $N(A)\coloneqq \sum_{\ell=0}^{d-1}A^{\ell}$. This can be simplified in some cases:
\begin{enumerate}

\smallskip

\item[a)] If there exists $c\in \mathds{Z}$ such that $c(A-\Ide)= d\Ide$, then we define $\Phi\colon I^2\to I^2$ by $\Phi(x,y)\coloneqq (x+cy,y)$, and we set $\wt{S}\coloneqq \Phi^{-1}\xcirc S$, $\wt{T}_1\coloneqq T_1\xcirc \Phi$ and $\wt{T}_2\coloneqq T_2\xcirc \Phi$. Clearly $\Phi$ induces an iso\-morphism
\begin{equation}\label{Phi barra1}
\ov{\Phi}\colon \frac{\ker \wt{T}_1\cap \ker \wt{T}_2}{\ima \wt{S}}\longrightarrow \frac{\ker T_1\cap \ker T_2}{\ima S}.
\end{equation}
A straightforward computation shows that,
$$
\wt{T}_1(x,y)=(\Ide-A)x, \quad \wt{T}_2(x,y)=N(A)y\quad\text{and}\quad \wt{S}(t)=(0,At-t).
$$
From Corollary~\ref{calculo de la homo B'}, if follows that
\begin{equation*}
\ho_{\blackdiamond}^2(H,I) \simeq \frac{\ker \wt{T}_1\cap \ker \wt{T}_2}{\ima \wt{S}} = \ker(A-\Ide)\oplus \frac{\ker(N(A))}{(A-\Ide)I} = \ho^0(\mathds{Z}_{d},I)\oplus \ho^1(\mathds{Z}_{d},I),
\end{equation*}
where the action of $\mathds{Z}_d=\langle \sigma \rangle$ on $I$ is given by $\sigma(y)\coloneqq A(y)$. Let
$$
\ov{T}_{\Phi}\colon \ker(A-\Ide)\oplus \frac{\ker(N(A))}{(A-\Ide)I}\longrightarrow \ho_{\blackdiamond}^2(H,I)
$$
be the isomorphism $\ov{T}_{\Phi}\coloneqq \ov{T}\xcirc\Phi$, where $\ov{T}$ is as in Corollary~\ref{calculo de la homo B'}. Clearly,
$\ov{T}_{\Phi}(\gamma,\mathfrak{f}_0) = (\alpha_{\gamma+c\mathfrak{f}_0},f_{\mathfrak{f}_0,\gamma+c\mathfrak{f}_0})$.

\smallskip

\item[b)] If there exists $c\in \mathds{Z}$ such that $A-\Ide = cd\Ide$, then we define $\Phi\colon I^2\to I^2$ by $\Phi(x,y)\coloneqq (x,y+cx)$, and we set $\wt{S}\coloneqq \Phi^{-1}\xcirc S$, $\wt{T}_1\coloneqq T_1\xcirc \Phi$ and $\wt{T}_2\coloneqq T_2\xcirc \Phi$. Clearly $\Phi$ induces an iso\-morphism
\begin{equation}\label{Phi barra}
\ov{\Phi}\colon \frac{\ker \wt{T}_1\cap \ker \wt{T}_2}{\ima \wt{S}}\longrightarrow \frac{\ker T_1\cap \ker T_2}{\ima S}.
\end{equation}
A straightforward computation shows that,
$$
\wt{T}_1(x,y)=dy, \quad \wt{T}_2(x,y)= N(A)(y+cx)\quad\text{and}\quad \wt{S}(t)=(dt,0).
$$
Note that if $(x,y)\in \ker \wt{T}_1$, then $(A-\Ide)y=cdy=0$. Hence $Ay=y$, and so $N(A)y=dy=0$. Therefore
\begin{equation}\label{acqua1}
\ho_{\blackdiamond}^2(H,I) \simeq \frac{\ker \wt{T}_1\cap \ker \wt{T}_2}{\ima \wt{S}} = \frac{\ker(cN(A))}{\ima(d\Ide)}\oplus \ker(d\Ide).
\end{equation}
Since $A=\Ide+cd\Ide$, we have
$$
\quad\qquad N(A) = \sum_{k=0}^{d-1}\sum_{\ell=0}^k\binom{k}{\ell}(cd)^{\ell}\Ide=\sum_{\ell=0}^{d-1}\sum_{k=\ell}^{d-1}\binom{k}{\ell}(cd)^{\ell}\Ide = \sum_{\ell=0}^{d-1}\binom{d}{\ell+1}(cd)^{\ell}\Ide,
$$
which can be useful for explicit calculations. Let
$$
\ov{T}_{\Phi}\colon \frac{\ker(cN(A))}{\ima(d\Ide)}\oplus \ker(d\Ide)\longrightarrow \ho_{\blackdiamond}^2(H,I)
$$
be the isomorphism $\ov{T}_{\Phi}\coloneqq \ov{T}\xcirc\Phi$, where $\ov{T}$ is as in Corollary~\ref{calculo de la homo B'}. Clearly,
$\ov{T}_{\Phi}(\gamma,\mathfrak{f}_0) = (\alpha_{\gamma},f_{\mathfrak{f}_0+c\gamma,\gamma})$.
\end{enumerate}
Note that if $c=0$ in case~a), or equivalently $dI=0$, then $\Phi = \Ide$, $\wt{T}_1 = T_1$, $\wt{T}_2 = T_2$, $\wt{S}=S$ and $\ov{T}_{\Phi} = T$. On the other hand, if $c=0$ in case~b), or equivalently $A=\Ide$, then again $\Phi = \Ide$, $\wt{T}_1 = T_1$, $\wt{T}_2 = T_2$, $\wt{S}=S$ and $\ov{T}_{\Phi} = T$.
\end{example}

\section[An application]{An application}\label{section 4}

Let $H\coloneqq \mathds{Z}_{p^{\eta}}$, endowed with the trivial linear cycle set structure. In this section, we explicitly compute all ex\-tension classes of $(\iota,I\times_{\beta,f}^{\blackdiamond,\Yleft} H,\pi)$, of $H$ by a finite cyclic $p$-group $I$, under the assumption that $\Yleft=0$ (according to~\cite{GGV1}*{Re\-mark~5.13}, this condition holds if and only if $\iota(I) \subseteq \Soc (I\times_{\beta,f}^{\blackdiamond,\Yleft} H)$). By Proposition~\ref{caracterizacion}, this requires computing all the endomorphisms $A\colon I\to I$, such that $A^{p^{\eta}} = \Ide$, and then, for each such endomorphism, determining the~coho\-mology group $\ho_{\blackdiamond}^2(H,I)$. For this latter task, we rely on Example~\ref{ejemplo}. In each case we examine, we must~com\-pute the maps $T_1$, $T_2$ and $S$, considered in that example, in order to obtain a family of pairs
$$
(\gamma,\mathfrak{f}_0)\in \ker(T_1)\cap \ker(T_2),
$$
which parametrizes a complete set of representatives of $2$-cocycles modulo coboundaries in ${\mathcal{C}}_{\blackdiamond}(H,I)$, via the correspondence $T(\gamma,\mathfrak{f}_0) = (\alpha_{\gamma},-f_{\mathfrak{f}_0})$, where, as we saw above Theorem~\ref{formula para f}, we write $f_{\mathfrak{f}_0}$ instead of $f_{\mathfrak{f}_0,\gamma}$. In each example, the parame\-ter family $(\gamma,\mathfrak{f}_0)$ is in bijective correspondence with a set of pairs $(z_1,z_2)$, where $z_1$ and $z_2$ belong to subquotients of~$\mathds{Z}_{p^r}$.

\smallskip

\begin{note} For the sake of simplicity in all cases we express the representative pair $(\gamma,\mathfrak{f}_0)$ as a function of the parameters $z_1$ and $z_2$. For example, the statement  ``In this case $(\gamma,\mathfrak{f}_0)\in\{(p^{r-\eta}z_1,z_2):0\le z_1,z_2<p^{\eta}\}$'' should be understood as ``In this case, a complete family $(\alpha_{\gamma},f_{\mathfrak{f}_0})$, of representatives of $2$-cocycles modulo coboundaries in ${\mathcal{C}}_{\blackdiamond}(H,I)$, is parameterized by the set of pairs $(\gamma,\mathfrak{f}_0)\in \{(p^{r-\eta}z_1,z_2):0\le z_1,z_2<p^{\eta}\}$''.
\end{note}

Next, we establish two lemmas needed to compute the endomorphisms $A\colon I\to I$, satisfying $A^{p^{\eta}} = \ide$.

\begin{lemma}\label{para ejemplos} Let $p\in \mathds{N}$ be an odd prime and let $r\ge 1$. Let $0\le a<p^r$ and $\eta\in \mathds{N}$. Then
\begin{equation*}
a^{p^{\eta}}\equiv 1\pmod{p^r} \Longleftrightarrow a^{p^{\eta_0}}\equiv 1\pmod{p^r} \Longleftrightarrow  a\equiv 1\pmod{p^{r-\eta_0}},
\end{equation*}
where $\eta_0\coloneqq \min(r-1,\eta)$.
\end{lemma}

\begin{proof} If $a\notin U(\mathds{Z}_{p^r})$, none of the conditions are met. Else, by Euler Theorem, $a^{(p-1)p^{r-1}}\equiv 1\pmod{p^r}$. Hence, $a^{p^{\eta}}\equiv 1\pmod{p^r}$ if and only if $a^{p^{\eta_0}}\equiv 1\pmod{p^r}$. When $r=1$, this finishes the proof. So we assume that $r>1$. Let~$s$ be a generator of $U(\mathds{Z}_{p^r})$ and write $\ov{a} = s^l$, where $\ov{a}$ is the class of $a$ in $\mathds{Z}_{p^r}$. Note that $a^{p^{\eta_0}} \equiv 1 \pmod{p^r}$ if and only if $(p-1)p^{r-1-\eta_0}\mid l$. Let $\hat{a}$ and $\hat{s}$ be the classes in $\mathds{Z}_{p^{r-1}}$, of $a$ and $s$, respectively. Since $\hat{s}$ is a generator of $U(\mathds{Z}_{p^{r-1}})$ and $\hat{a} = \hat{s}^l$, the same argument as above proves that $(p-1)p^{(r-2)-(\eta_0-1)}= (p-1)p^{r-1-\eta_0}\mid l$ if and only if $a^{p^{\eta_0-1}}\equiv 1\pmod{p^{r-1}}$. An inductive argument concludes the proof.
\end{proof}

\begin{remark}\label{complemento para ejemplos} The previous lemma and its proof hold for $p=2$ when $1\le r\le 2$ (because then $U(\mathds{Z}_{2^r})$ is cyclic).
\end{remark}

\begin{lemma}\label{para ejemploscon2} Let $r>2$, $0\le a<2^r$ and $\eta\in \mathds{N}$. Then
\begin{equation*}
a^{2^{\eta}}\equiv 1\pmod{2^r} \Longleftrightarrow a^{2^{\eta_0}}\equiv 1\pmod{2^r} \Longleftrightarrow  a^2\equiv 1\pmod{2^{r+1-\eta_0}},
\end{equation*}
where $\eta_0\coloneqq \min(r-2,\eta)$.
\end{lemma}

\begin{proof}  If $a\notin U(\mathds{Z}_{2^r})$, none of the conditions are met. Else $a^{2^{r-2}}\equiv 1\pmod{2^r}$, and so $a^{2^{\eta}}\equiv 1\pmod{2^r}$ if and only if $a^{2^{\eta_0}}\equiv 1\pmod{2^r}$. When $r=3$, this finishes the proof. So we assume that $r>3$. Let $s\in U(\mathds{Z}_{2^r})$ be such that $U(\mathds{Z}_{2^r}) = \{s^i,-s^i:0\le i<2^{r-2}\}$. Write $\ov{a} = \pm s^l$, where $\ov{a}$ is the class of $a$ in $\mathds{Z}_{2^r}$. Since $\eta_0\ge 1$, we have $a^{2^{\eta_0}} \equiv 1\pmod{2^r}$ if and only if $2^{r-2-\eta_0}\mid l$. Let $\hat{a}$ and $\hat{s}$ be the classes in $\mathds{Z}_{2^{r-1}}$, of $a$ and $s$, respectively. Since $U(\mathds{Z}_{2^{r-1}}) = \{\hat{s}^i,-\hat{s}^i:0\le i<2^{r-3}\}$ and $\hat{a} = \pm\hat{s}^l$, the same argument as above proves that $2^{(r-3)-(\eta_0-1)}= 2^{r-2-\eta_0}\mid l$ if and only if $a^{2^{\eta_0-1}}\equiv 1\pmod{2^{r-1}}$. An inductive argument concludes the proof.
\end{proof}

\begin{remark}\label{A para ejemploscon2} It is well known that $a^2\equiv 1\pmod{2^{r+1-\eta_0}}$ if and only if $a\equiv \pm 1\pmod{2^{r-\eta_0}}$.
\end{remark}

\subsection[Case \texorpdfstring{$H=\mathds{Z}_{p^{\eta}}$}{H=Zpn} and \texorpdfstring{$I=\mathds{Z}_{p^r}$}{I=Zpr} with \texorpdfstring{$r\le \eta$}{r<=n}, if \texorpdfstring{$p$}{p} is odd, and \texorpdfstring{$r\le \min(2,\eta)$}{r<= min(2,n)}, if \texorpdfstring{$p=2$}{p=2}]{Case \texorpdfstring{$\pmb{H=\mathds{Z}_{p^{\eta}}}$}{H=Zpn} and \texorpdfstring{$\pmb{I=\mathds{Z}_{p^r}}$}{I=Zpr} with \texorpdfstring{$\pmb{r\le \eta}$}{r<=n}, if \texorpdfstring{$\pmb{p}$}{p} is odd, and \texorpdfstring{$\pmb{r\le \min(2,\eta)}$}{r<= min(2,n)}, if \texorpdfstring{$\pmb{p=2}$}{p=2}}

\begin{proposition} Let $H\coloneqq\left(\mathds{Z}_{p^{\eta}},+\right)$, endowed with the trivial cycle set structure, and let $I\coloneqq \mathds{Z}_{p^r}$. Assume~$r\le \eta$, if~$p$ is odd, and $r\le \min(2,\eta)$, if $p=2$. Then, there exists $0\le k<p^{r-1}$, such that  $h\blackdiamond y =  (1+pk)^h y$. If $k\ne 0$, then we write $k=p^us$ with $0\le u<r-1$ and $p\nmid s$. Under these conditions two distinct cases arise:

\begin{description}[font=\normalfont\scshape, leftmargin=0cm]

\item[\underline{$k=0$}] In this case $(\gamma,\mathfrak{f}_0)\in\{(z_1,z_2):0\le z_1,z_2<p^r\}$. Moreover, $h\blackdiamond y = y$ and $f_{\mathfrak{f}_0}(h,h')=\mathfrak{f}_0 hh'$.

\item[\underline{$k\ne 0$}] In this case $(\gamma,\mathfrak{f}_0)\in\{(p^{r-u-1}z_1,z_2):0\le z_1,z_2<p^{u+1}\}$. Moreover,
$$
h\blackdiamond y = \sum_{i=0}^h \binom{h}{i} p^ik^iy\quad\text{and}\quad f_{\mathfrak{f}_0}(h,h') = \sum_{\ell=0}^{h-1} \binom{h}{\ell+1} p^{(u+1)\ell}s^{\ell} \mathfrak{f}_0 h'.
$$

\end{description}
\end{proposition}

\begin{proof} Let $A$ be as in Example~\ref{ejemplo}. Since~$A$ is the multiplication by an invertible element $a\in \mathds{Z}_{p^r}$ and $A^{p^{\eta}} = \Ide$, we know that $a = 1+pk$, with $0\le k< p^{r-1}$ (see Lemma~\ref{para ejemplos} and Remark~\ref{complemento para ejemplos}). From the first formula in~\eqref{acciones caso trivial}, we obtain the formula for $h\blackdiamond y$; while, from equality~\eqref{zazaza1}, we derive
\begin{equation}\label{cociclo1 ej 3}
f_{\mathfrak{f}_0}(h,h') =\sum_{j=0}^{h-1} (1+pk)^j \mathfrak{f}_0h'= \sum_{j=0}^{h-1} \sum_{\ell=0}^j\binom{j}{\ell} p^{\ell}k^{\ell}\mathfrak{f}_0h'= \sum_{\ell=0}^{h-1} \sum_{j=\ell}^{h-1} \binom{j}{\ell} p^{\ell}k^{\ell}\mathfrak{f}_0h' = \sum_{\ell=0}^{h-1} \binom{h}{\ell+1} p^{\ell}k^{\ell} \mathfrak{f}_0h'.
\end{equation}
We are in the case~a) of Example~\ref{para ejemplos}, with $c=0$, since $p^{\eta}I = 0$. We claim that $N(A) = 0$. Indeed, a direct~computa\-tion yields
\begin{equation*}
N(A) = \sum_{j=0}^{p^{\eta}-1} (1+pk)^j = \sum_{j=0}^{p^{\eta}-1}\sum_{\ell=0}^j \binom{j}{\ell}p^{\ell}k^{\ell} = \sum_{\ell=0}^{p^{\eta}-1}\sum_{j=\ell}^{p^{\eta}-1} \binom{j}{\ell}p^{\ell}k^{\ell} = \sum_{\ell=0}^{p^{\eta}-1}\binom{p^{\eta}}{\ell+1} p^{\ell}k^{\ell}.
\end{equation*}
Write $\ell+1 = p^tv$ with $p\nmid v$ and $t\ge 0$. It is well known that $p^{\eta-t}\mid \binom{p^{\eta}}{\ell+1}$. Hence, the exponent of $p$ in $\binom{p^{\eta}}{\ell+1} p^{\ell}$, is greater than or equal to $\eta-t+\ell$. Consequently $N(A) = 0$, because $\eta-t+\ell\ge \eta\ge r$, since $\ell\ge t$ and $r\le \eta$. Hence,
$$
\frac{\ker T_1\cap \ker T_2}{\ima S} = \ker(A-\Ide)\oplus \frac{\ker(N(A))}{(A-\Ide)I} = \begin{cases} \mathds{Z}_{p^r}\oplus \mathds{Z}_{p^r} & \text{if $k=0$,}\\ p^{r-u-1}\mathds{Z}_{p^r}\oplus\frac{\mathds{Z}_{p^r}}{p^{u+1}\mathds{Z}_{p^r}} & \text{if $k=p^us$ with $0\le u< r-1$ and $p\nmid s$.}\end{cases}
$$
From this and~\eqref{cociclo1 ej 3}, the conditions in the statement follow immediately.
\end{proof}

\subsection[Case \texorpdfstring{$H=\mathds{Z}_{p^{\eta}}$}{H=Zpn} and \texorpdfstring{$I=\mathds{Z}_{p^r}$}{I=Zpr} with \texorpdfstring{$r>\eta$}{r>n}, if \texorpdfstring{$p$}{p} is odd, and \texorpdfstring{$(\eta,r)=(1,2)$}{(n,r)=(1,2)}, if \texorpdfstring{$p=2$}{p=2}]{Case \texorpdfstring{$\pmb{H=\mathds{Z}_{p^{\eta}}}$}{H=Zpn} and \texorpdfstring{$\pmb{I=\mathds{Z}_{p^r}}$}{I=Zpr} with \texorpdfstring{$\pmb{r>\eta}$}{r>n}, if \texorpdfstring{$\pmb{p}$}{p} is odd, and \texorpdfstring{$\pmb{(\eta,r)=(1,2)}$}{(n,r)=(1,2)}, if \texorpdfstring{$\pmb{p=2}$}{p=2}}

\begin{proposition}\label{prop 4.10} Let $H\coloneqq\left(\mathds{Z}_{p^{\eta}},+\right)$, endowed with the trivial cycle set structure, and let $I\coloneqq \mathds{Z}_{p^r}$. Assume $\eta<r$ if $p$ is odd, and $(\eta,r) = (1,2)$ if $p=2$. Then, there exists an integer $0\le k<p^{\eta}$ such that $h\blackdiamond y = (1+p^{r-\eta}k)^h y$. If~$k\ne 0$, then we write $k=p^us$, with $0\le u<\eta$ and $p\nmid s$. Under these conditions four distinct cases arise:
\begin{description}[font=\normalfont\scshape, leftmargin=0cm]

\item[\underline{$k=0$}] In this case $(\gamma,\mathfrak{f}_0)\in \{(z_1,p^{r-\eta}z_2):0\le z_1,z_2<p^{\eta}\}$. Moreover, $h\blackdiamond y = y$ and  $f_{\mathfrak{f}_0}(h,h') = \mathfrak{f}_0hh'$.

\item[\underline{$p\ne 2$, $k\ne 0$ and $r+u\ge 2\eta$}] In this case
\begin{equation}\label{caso p>2, cond ppal}
(\gamma,\mathfrak{f}_0)\in\{(p^{\eta-u}z_1,p^{r-\eta}(z_2+sz_1)):0\le z_1<p^u\text{ and } 0\le z_2<p^{\eta}\}.
\end{equation}
Moreover,
\begin{equation}\label{caso p>2,...}
h \blackdiamond y = y+p^{r-\eta+u}shy \quad\text{and}\quad f_{\mathfrak{f}_0}(h,h') = \mathfrak{f}_0hh'.
\end{equation}

\item[\underline{$p=2$, $k\ne 0$ and $r+u\ge 2\eta$}] In this case $k=1$ and $(\gamma,\mathfrak{f}_0)\in \{(z_1,2z_2+z_1):0\le z_1, z_2<2\}$. Moreover,
\begin{equation}\label{caso p=2,...}
h \blackdiamond y = (1+2h)y \quad\text{and}\quad f_{\mathfrak{f}_0}(h,h') = \mathfrak{f}_0hh'.
\end{equation}

\item[\underline{$k\ne 0$ and $r+u< 2\eta$}] In this case $p\ne 2$ and
\begin{equation}\label{caso r+u<2eta,...form ppal}
(\gamma,\mathfrak{f}_0)\in\{(p^{\eta-u}(z_1+z_2),p^{r-\eta}sz_2): 0\le z_1<p^{r-\eta+u} \text{ and } 0\le z_2<p^u\}.
\end{equation}
More\-over,
\begin{equation}\label{caso r+u<2eta,...}
h \blackdiamond y =\sum_{i=0}^h \binom{h}{i} p^{(r-\eta+u)i}s^iy \quad\text{and}\quad f_{\mathfrak{f}_0}(h,h') =\sum_{\ell=0}^{h-1} p^{(r-\eta+u)\ell} s^{\ell} \binom{h}{\ell+1} \mathfrak{f}_0h'.
\end{equation}
\end{description}
\end{proposition}

\begin{proof} Let $A$ be as in Example~\ref{ejemplo}. Since~$A$ is the multiplication by an invertible element $a\in \mathds{Z}_{p^r}$ and $A^{p^{\eta}} = \Ide$, by Lemma~\ref{para ejemplos} and Remark~\ref{complemento para ejemplos}, we know that $a = 1+p^{r-\eta}k$, with $0\le k<p^{\eta}$. From~\eqref{acciones caso trivial}, we obtain that
\begin{equation}\label{accion1 ej 4}
h \blackdiamond y = (1+p^{r-\eta}k)^h y = \sum_{i=0}^h \binom{h}{i} p^{(r-\eta)i}k^iy;
\end{equation}
while, from equality~\eqref{zazaza1}, we obtain that
\begin{equation}\label{cociclo1 ej 4}
f_{\mathfrak{f}_0}(h,h') = \sum_{j=0}^{h-1}a^j \mathfrak{f}_0h' = \sum_{j=0}^{h-1} \sum_{\ell=0}^j \binom{j}{\ell} p^{(r-\eta)\ell}k^{\ell}\mathfrak{f}_0h'= \sum_{\ell=0}^{h-1} \binom{h}{\ell+1} p^{(r-\eta)\ell}k^{\ell}\mathfrak{f}_0h'.
\end{equation}
Assume first that $k=0$. Then, we are in the case~b) of Example~\ref{ejemplo} with $c=0$ and we have $N(A) = p^{\eta}$.~Conse\-quent\-ly, by~\eqref{acqua1}
\begin{equation*}
\ho_{\blackdiamond}^2(H,I) \simeq \frac{\ker T_1\cap \ker T_2}{\ima S} = \frac{\ker(cN(A))}{\ima(p^{\eta}\Ide)}\oplus \ker(p^{\eta}\Ide) = \frac{\mathds{Z}_{p^r}}{p^{\eta}\mathds{Z}_{p^r}}\oplus p^{r-\eta}\mathds{Z}_{p^r}.
\end{equation*}
The present case of the statement follows from this, \eqref{accion1 ej 4} and~\eqref{cociclo1 ej 4}. Assume now that $k\ne 0$ and $r+u\ge 2\eta$. Then, we are again in the case~b) of Example~\ref{ejemplo}, with~$c=p^{r+u-2\eta}s$ and
$$
N(A) = \sum_{\ell=0}^{p^{\eta}-1}\binom{p^{\eta}}{\ell+1}p^{(r+u-\eta)\ell}s^{\ell}.
$$
If $p=2$, then $\eta=1$, $r=2$, $k=1$, and a direct computation gives $c=1$ and $N(A) = \binom{2}{1} + \binom{2}{2}2 = 4 = 0$. Hence,
\begin{equation*}
\ho_{\blackdiamond}^2(H,I) \simeq \frac{\ker \wt{T}_1\cap \ker \wt{T}_2}{\ima \wt{S}} = \frac{\ker(cN(A))}{\ima(2\Ide)}\oplus \ker(2\Ide) = \frac{\mathds{Z}_4}{2\mathds{Z}_4}\oplus 2\mathds{Z}_4.
\end{equation*}
The fact that $(\gamma,\mathfrak{f}_0)\in \{(z_1,2z_2+z_1):0\le z_1, z_2<2\}$ follows from this, applying the map $\ov{\Phi}$ of \eqref{Phi barra}; the first equality in~\eqref{caso p=2,...}, follows from~\eqref{accion1 ej 4}; and, the second one, from~\eqref{cociclo1 ej 4}. As\-sume now that $p$ is odd and write $\ell+1 = p^tv$ with~$p\nmid v$ and $t\ge 0$. It is well known that $p^{\eta-t}\mid \binom{p^{\eta}}{\ell+1}$. Consequently, the exponent $w$, of $p$ in $\binom{p^{\eta}}{\ell+1} p^{(r+u-\eta)\ell}$, is greater than or equal to $\eta-t+(r+u-\eta)\ell$. Thus, if $\ell\ge 1$, then
\begin{equation*}
w\ge \eta-t+(r+u-\eta)\ell = r-t+u+(r+u-\eta)(\ell-1)\ge r+u-t + \ell-1\ge r,
\end{equation*}
where we have used that $\ell-1 = vp^t-2\ge t$, for $\ell\ge 1$. Hence, in this case, $p^r\mid \binom{p^{\eta}}{\ell+1}p^{(r+u-\eta)\ell}$, and so $N(A) = p^{\eta}$. Thus, $cN(A) = p^{r+u-\eta}s$, and so
\begin{equation*}
\ho_{\blackdiamond}^2(H,I) \simeq \frac{\ker \wt{T}_1\cap \ker \wt{T}_2}{\ima \wt{S}} = \frac{\ker(cN(A))}{\ima(p^{\eta}\Ide)}\oplus \ker(p^{\eta}\Ide) = \frac{p^{\eta-u}\mathds{Z}_{p^r}}{p^{\eta}\mathds{Z}_{p^r}}\oplus p^{r-\eta}\mathds{Z}_{p^r}.
\end{equation*}
Applying the map $\ov{\Phi}$ of~\eqref{Phi barra}, we obtain~\eqref{caso p>2, cond ppal}. Moreover, since $2(r-\eta)+u\ge r$, formulas~\eqref{accion1 ej 4} and~\eqref{cociclo1 ej 4}~be\-come
\begin{align*}
&h \blackdiamond y = \sum_{i=0}^h \binom{h}{i} p^{(r-\eta+u)i} s^iy=y+hp^{r-\eta+u}sy
\shortintertext{and}
&f_{\mathfrak{f}_0}(h,h') = \sum_{\ell=0}^{h-1} \binom{h}{\ell+1} p^{(r-\eta)(\ell+1)+u\ell}s^{\ell}p^{r-\eta}(z_2+sz_1)h' = p^{r-\eta}(z_2+sz_1)hh' = \mathfrak{f}_0hh',
\end{align*}
which gives~\eqref{caso p>2,...}. Assume finally that $k\ne 0$ and that $r+u<2\eta$ (thus $(\eta,r)\ne (1,2)$, which implies $p\ne 2$). Clearly, we are in the case~a) of Example~\ref{ejemplo}, with $c=p^{2\eta-r-u}s'$, where $s'\in\mathds{N}_0$ is the unique element such that $s'<p^r$ and~$ss'\equiv 1\pmod{p^r}$. Computing, we obtain
\begin{equation*}
N(A) = \sum_{j=0}^{p^{\eta}-1} (1+p^{r+u-\eta}s)^j = \sum_{j=0}^{p^{\eta}-1}\sum_{\ell=0}^j \binom{j}{\ell}p^{(r+u-\eta)\ell}s^{\ell} = \sum_{\ell=0}^{p^{\eta}-1}\sum_{j=\ell}^{p^{\eta}-1} \binom{j}{\ell} p^{(r+u-\eta)\ell}s^{\ell} = \sum_{\ell=0}^{p^{\eta}-1}\binom{p^{\eta}}{\ell+1} p^{(r+u-\eta)\ell}s^{\ell}.
\end{equation*}
Write $\ell+1 = p^tv$, with $p\nmid v$ and $t\ge 0$. It is well known that $p^{\eta-t}\mid \binom{p^{\eta}}{\ell+1}$. So, the exponent $w$, of $p$ in $\binom{p^{\eta}}{\ell+1} p^{(r+u-\eta)\ell}$, is greater than or equal to $\eta-t+(r+u-\eta)\ell$. Thus, if $\ell\ge 1$, then
\begin{equation*}
w\ge \eta-t+(r+u-\eta)\ell = r-t+u+(r+u-\eta)(\ell-1)\ge r+u-t + \ell-1\ge r,
\end{equation*}
where we have used that $r+u-\eta>0$ and that $\ell-1 = vp^t-2\ge t$, for $\ell\ge 1$. Consequently, $N(A) = p^{\eta}$, and so
\begin{equation*}
\ho_{\blackdiamond}^2(H,I) \simeq \frac{\ker \wt{T}_1\cap \ker \wt{T}_2}{\ima \wt{S}} = \ker(A-\Ide)\oplus \frac{\ker(N(A))}{(A-\Ide)I} = p^{\eta-u}\mathds{Z}_{p^r}\oplus\frac{p^{r-\eta}\mathds{Z}_{p^r}}{p^{r+u-\eta}\mathds{Z}_{p^r}}.
\end{equation*}
Applying the map $\ov{\Phi}$ of~\eqref{Phi barra1}, we obtain~\eqref{caso r+u<2eta,...form ppal}. Finally, the formulas in~\eqref{caso r+u<2eta,...} follow from~\eqref{accion1 ej 4} and~\eqref{cociclo1 ej 4}.
\end{proof}

\subsection[Case \texorpdfstring{$H=\mathds{Z}_{2^{\eta}}$}{H=Z2n} and \texorpdfstring{$I=\mathds{Z}_{2^r}$}{Z2r} with \texorpdfstring{$3\le r\le \eta+1$}{3<=r<=n+1}]{Case \texorpdfstring{$\pmb{H=\mathds{Z}_{2^{\eta}}}$}{H=Z2n} and \texorpdfstring{$\pmb{I=\mathds{Z}_{2^r}}$}{Z2r} with \texorpdfstring{$\pmb{3\le r\le \eta+1}$}{3<=r<=n+1}}

\begin{proposition}\label{prop 3.12} Let $H\coloneqq\left(\mathds{Z}_{2^{\eta}},+\right)$, equipped with the trivial cycle set structure, and let $I\coloneqq \mathds{Z}_{2^r}$, where we assume that $3\le r\le \eta+1$. Then, there exists an element $a = \pm (1+4k)$, with $0\le k<2^{r-2}$, such that $h\blackdiamond y = a^h y$. When~$k\ne 0$, we write $k=2^us$, with $0\le u<r-2$ and $2\nmid s$. Under these conditions, four distinct cases arise:
\begin{description}[font=\normalfont\scshape, leftmargin=0cm]

\item[\underline{$a=1$}] In this case, $k=0$ and
\begin{equation}\label{k=0, a=1, caso par}
(\gamma,\mathfrak{f}_0)\in\begin{cases} \{(z_1,z_2) : 0\le z_1,z_2<2^r\} & \text{if $r\le \eta$,}\\ \{(z_1,2z_2) : 0\le z_1,z_2<2^{r-1}\} & \text{if $r = \eta+1$.}\end{cases}
\end{equation}
Moreover,
\begin{equation}\label{k=0, a=1, caso par'''}
h\blackdiamond y = y \quad\text{and}\quad f_{\mathfrak{f}_0}(h,h')=\mathfrak{f}_0hh'.
\end{equation}

\item[\underline{$a=-1$}] In this case, $k=0$ and $(\gamma,\mathfrak{f}_0)\in\{(2^{r-1}z_1-2^{\eta-1}z_2,z_2) :  0\le z_1,z_2<2\}$. Moreover,
\begin{equation}\label{k=0, a=-1, caso par}
h\blackdiamond y = (-1)^h y \quad\text{and}\quad f_{f_0}(h,h') = \begin{cases} f_0h'  & \text{if $h$ is odd,}\\ 0 & \text{if $h$ is even.}\end{cases}
\end{equation}

\item[\underline{$a=1+4k$ with $k\ne 0$}] Then,
\begin{equation}\label{a=1+4k with k ne 0}
(\gamma,\mathfrak{f}_0)\in \begin{cases} \{(2^{r-u-2}z_1,z_2):0\le z_1,z_2<2^{u+2}\} &\text{if $r\le\eta$,}\\
\{(2^{r-u-2}z_1+2^{\eta-u-1}z_2,2z_2) : 0\le z_1<2^{u+2} \text{ and } 0\le z_2<2^{u+1}\} &\text{if $r=\eta+1$.}\end{cases}
\end{equation}
Moreover,
\begin{equation}\label{a=1+4k with k ne 0'}
h\blackdiamond y = \sum_{i=0}^h \binom{h}{i} 4^ik^iy \quad\text{and}\quad f_{\mathfrak{f}_0}(h,h') =\sum_{\ell=0}^{h-1} \binom{h}{\ell+1}4^{\ell}k^{\ell}\mathfrak{f}_0h'.
\end{equation}

\item[\underline{$a=-1-4k$ with $k\ne 0$}] Then, $(\gamma,\mathfrak{f}_0)\in\{(2^{r-1}z_1-2^{\eta-1}z_2,z_2):0\le z_1, z_2<2\}$. Moreover,
\begin{equation}\label{a=-1k4k with k ne 0'}
h\blackdiamond y = (-1)^h \sum_{i=0}^h \binom{h}{i} 4^ik^iy \quad\text{and}\quad f_{f_0}(h,h') =\sum_{\ell=0}^{h-1} A_{h\ell} 4^{\ell}k^{\ell}f_0h',
\end{equation}
where $A_{h\ell}\coloneqq \sum_{j=\ell}^{h-1} (-1)^{j} \binom{j}{\ell}$.
\end{description}
\end{proposition}

\begin{proof} Let $A$ be as in Example~\ref{ejemplo}. By Lemma~\ref{para ejemploscon2} and Remark~\ref{A para ejemploscon2}, since~$A$ is the multiplication by an~in\-vert\-ible element $a\in \mathds{Z}_{2^r}$ and $A^{p^{\eta}} = \Ide$, we know that $a = \pm (1+4k)$, with $0\le k<2^{r-2}$. By~\eqref{acciones caso trivial}, since $1^{\times h}=h$, we have:
\begin{equation}\label{calculo diamante p=2 ej 6}
h \blackdiamond y =\begin{cases} (1+4k)^h y = \sum_{i=0}^h \binom{h}{i} 4^ik^iy & \text{if $a = 1+4k$,}\\ (-1)^h(1+4k)^h y = (-1)^h\sum_{i=0}^h \binom{h}{i} 4^ik^iy & \text{if $a = -1-4k$;}\end{cases}
\end{equation}
while, from equality~\eqref{zazaza1}, we obtain
\begin{equation*}
f_{\mathfrak{f}_0}(h,h') = h'\sum_{j=0}^{h-1}a^j \mathfrak{f}_0,
\end{equation*}
Assume first that $a=1+4k$. Then, $a^j = \sum_{\ell=0}^j \binom{j}{\ell}4^{\ell}k^{\ell}$, and so
\begin{equation}\label{cociclo1caso1 ej 6}
f_{\mathfrak{f}_0}(h,h') = \sum_{j=0}^{h-1} \sum_{\ell=0}^j \binom{j}{\ell} 4^{\ell}k^{\ell}\mathfrak{f}_0h' = \sum_{\ell=0}^{h-1} \sum_{j=\ell}^{h-1} \binom{j}{\ell} 4^{\ell}k^{\ell}\mathfrak{f}_0h'= \sum_{\ell=0}^{h-1} \binom{h}{\ell+1} 4^{\ell}k^{\ell}\mathfrak{f}_0h'.
\end{equation}
Assume now that $a = -1-4k$. Then, $a^j = \sum_{\ell=0}^j (-1)^j\binom{j}{\ell}4^{\ell}k^{\ell}$, and so
\begin{equation}\label{cociclo1caso2 ej 6}
f_{\mathfrak{f}_0}(h,h') = \sum_{j=0}^{h-1} (-1)^j\sum_{\ell=0}^{j} \binom{j}{\ell} 4^{\ell}k^{\ell}\mathfrak{f}_0h' = \sum_{\ell=0}^{h-1}\sum_{j=\ell}^{h-1} (-1)^j\binom{j}{\ell} 4^{\ell}k^{\ell}\mathfrak{f}_0h'= \sum_{\ell=0}^{h-1} A_{h\ell} 4^{\ell}k^{\ell}\mathfrak{f}_0h'.
\end{equation}
Next we compute the cohomology. Assume first that $a=1$. Then, we are in the case~b) of Example~\ref{ejemplo}, with $c=0$, and we have $N(A) = 2^{\eta}$. So,
\begin{equation*}
\ho_{\blackdiamond}^2(H,I)\simeq \frac{\ker T_1\cap \ker T_2}{\ima S} = \frac{\ker(cN(A))}{\ima(2^{\eta}\Ide)}\oplus \ker(2^{\eta}\Ide) = \begin{cases} \mathds{Z}_{2^r} \oplus \mathds{Z}_{2^r} &\text{if $r\le \eta$,}\\ \frac{\mathds{Z}_{2^r}}{2^{r-1}\mathds{Z}_{2^r}}\oplus 2\mathds{Z}_{2^r} &\text{if $r=\eta+1$.}\end{cases}
\end{equation*}
The assertion in~\eqref{k=0, a=1, caso par} follows from this; the first assertion in~\eqref{k=0, a=1, caso par'''}, from~\eqref{calculo diamante p=2 ej 6}; and, the second one, from~\eqref{cociclo1caso1 ej 6}. Assume now that $a=-1$. Then, we are in the case~a) of Example~\ref{ejemplo}, with $c=-2^{\eta-1}$, and we have $N(A) = 0$. Hence,
\begin{equation*}
\ho_{\blackdiamond}^2(H,I) \simeq \frac{\ker \wt{T}_1\cap \ker \wt{T}_2}{\ima \wt{S}} = \ker(A-\Ide)\oplus \frac{\ker(N(A))}{(A-\Ide)I} = 2^{r-1} \mathds{Z}_{2^r} \oplus \frac{\mathds{Z}_{2^r}}{2\mathds{Z}_{2^r}}.
\end{equation*}
The fact that $(\gamma,\mathfrak{f}_0)\in\{(2^{r-1}z_1-2^{\eta-1}z_2,z_2):0\le z_1,z_2<2\}$ follows from this, applying the map $\ov{\Phi}$ of~\eqref{Phi barra1}; the first assertion in~\eqref{k=0, a=-1, caso par}, follows from~\eqref{calculo diamante p=2 ej 6}; and, the second one, from~\eqref{cociclo1caso2 ej 6}. Assume now that $a = 1+4k$, with $k\ne 0$. Then, we are in the case~a) of Ex\-ample~\ref{ejemplo}, with $c=0$ if~$r\le \eta$, and $c=2^{\eta-u-2}$ if $r = \eta+1$. In fact, if $r\le \eta$, then this is clear, since $c(A-\ide) = 0 = 2^{\eta}$; while if~$r=\eta+1$, then we have $c(A-\ide) = 2^{\eta}s = 2^{\eta}$. Computing, we obtain
\begin{equation*}
N(A) = \sum_{j=0}^{2^{\eta}-1} (1+2^{u+2}s)^j = \sum_{j=0}^{2^{\eta}-1}\sum_{\ell=0}^j \binom{j}{\ell} 2^{(u+2)\ell}s^{\ell} = \sum_{\ell=0}^{2^{\eta}-1} \sum_{j=\ell}^{2^{\eta}-1} \binom{j}{\ell} 2^{(u+2)\ell}s^{\ell} = \sum_{\ell=0}^{2^{\eta}-1}\binom{2^{\eta}}{\ell+1} 2^{(u+2)\ell}s^{\ell}.
\end{equation*}
Write $\ell+1 = 2^tv$, with $2\nmid v$ and $t\ge 0$. It is well known that $2^{\eta-t}\mid \binom{2^{\eta}}{\ell+1}$. So, the exponent $w$, of $2$ in $\binom{2^{\eta}}{\ell+1} 2^{(u+2)\ell}$,~is greater than or equal to $\eta-t+(u+2)\ell$. Thus, if $\ell\ge 1$, then $w\ge \eta-t+(u+2)\ell\ge \eta+(u+1)\ell  \ge \eta +1 \ge r$, where we have used that $\ell \ge t$. Con\-sequently, $N(A) = 2^{\eta}$, and so
\begin{equation*}
\ho_{\blackdiamond}^2(H,I)\simeq \frac{\ker \wt{T}_1\cap \ker \wt{T}_2}{\ima \wt{S}} = \ker(A-\Ide)\oplus \frac{\ker(N(A))}{(A-\Ide)I} = \begin{cases} 2^{r-u-2}\mathds{Z}_{2^r}\oplus\frac{\mathds{Z}_{2^r}}{2^{u+2}\mathds{Z}_{2^r}} & \text{if $r\le\eta $,}\\ 2^{r-u-2}\mathds{Z}_{2^r}\oplus \frac{2\mathds{Z}_{2^r}}{2^{u+2} \mathds{Z}_{2^r}} & \text{if $r=\eta+1$.}
\end{cases}
\end{equation*}
Applying the map $\ov{\Phi}$ of~\eqref{Phi barra1}, we obtain~\eqref{a=1+4k with k ne 0}; while~\eqref{a=1+4k with k ne 0'} follows immediately from~\eqref{calculo diamante p=2 ej 6} and~\eqref{cociclo1caso1 ej 6}. Assume~final\-ly that $a = -1-4k$, with $k\ne 0$. Then, we are in the case~a) of Ex\-ample~\ref{ejemplo}, with $c=-2^{\eta-1}$. Indeed,~a~direct~com\-pu\-tation proves that $c(A-\Ide) = 2^{\eta}(1+2k) = 2^{\eta}$. Computing, we obtain
\begin{equation*}
N(A) = \sum_{j=0}^{2^{\eta}-1} (-1-4k)^j = \sum_{j=0}^{2^{\eta}-1}(-1)^j\sum_{\ell=0}^j \binom{j}{\ell} 4^{\ell}k^{\ell} = \sum_{\ell=0}^{2^{\eta}-1} \sum_{j=\ell}^{2^{\eta}-1} (-1)^j\binom{j}{\ell} 4^{\ell}k^{\ell} = \sum_{\ell=0}^{2^{\eta}-1} A_{2^{\eta}\ell} 4^{\ell}k^{\ell}.
\end{equation*}
But, by~\cite{Spivey}*{identity 217}, we have $A_{2^{\eta}\ell} = \sum_{i=1}^{\ell} \binom{2^{\eta}}{i}\bigl(\frac{-1}{2}\bigr)^{\ell+1-i}$. Clearly $A_{2^{\eta}0}=0$ and $A_{2^{\eta}1}=-2^{\eta-1}$.~More\-over, using that $2A_{2^{\eta},\ell+1} =  - A_{2^{\eta}\ell}- \binom{2^{\eta}}{\ell+1}$, for all $1\le \ell < 2^{\eta}-1$ (and an inductive argument), we obtain that
\begin{equation}\label{divisibilidad de B sub l}
2^{2\ell-\eta-2} A_{2^{\eta}\ell}\in \mathds{Z}\qquad\text{for $2\le \ell<2^{\eta}$.}
\end{equation}
Hence $4^{\ell}A_{2^{\eta}\ell} = 2^r 2^{2\ell - r}A_{2^{\eta}\ell} \in 2^r\mathds{Z}$, for $2\le \ell < 2^{\eta}$ (since $r\le \eta+1$), and so $N(A) = 2^{\eta+1}k = 0$. Consequently,
\begin{equation*}
\ho_{\blackdiamond}^2(H,I)\simeq \frac{\ker \wt{T}_1\cap \ker \wt{T}_2}{\ima \wt{S}} = \ker(A-\Ide)\oplus \frac{\ker(N(A))}{(A-\Ide)I} = 2^{r-1}\mathds{Z}_{2^r}\oplus\frac{\mathds{Z}_{2^r}}{2\mathds{Z}_{2^r}}.
\end{equation*}
Applying the map $\ov{\Phi}$ of~\eqref{Phi barra1}, we obtain that $(\gamma,\mathfrak{f}_0)\in\{(2^{r-1}z_1-2^{\eta-1}z_2,z_2) : 0\le z_1, z_2<2\}$; while~\eqref{a=-1k4k with k ne 0'} follows from~\eqref{calculo diamante p=2 ej 6} and~\eqref{cociclo1caso2 ej 6}.
\end{proof}

\subsubsection[Case \texorpdfstring{$H=\mathds{Z}_{2^{\eta}}$}{H=Z2n} and \texorpdfstring{$I=\mathds{Z}_{2^r}$}{Z2r} with \texorpdfstring{$r>\eta+1$}{r>n+1}]{Case \texorpdfstring{$\pmb{H=\mathds{Z}_{2^{\eta}}}$}{H=Z2n} and \texorpdfstring{$\pmb{I=\mathds{Z}_{2^r}}$}{I=Z2r} with \texorpdfstring{$\pmb{r>\eta+1}$}{r>n+1}}

\begin{proposition} Let $H\coloneqq\left(\mathds{Z}_{2^{\eta}},+\right)$, endowed with the trivial cycle set structure, and let $I\coloneqq \mathds{Z}_{2^r}$, where we assume that $r>\eta+1$. Then, there exists $a = \pm (1+2^{r-\eta}k)$ with $0\le k<2^{\eta}$, such that $h\blackdiamond y = a^h y$. Moreover, if $k\ne 0$, then we write $k=2^us$ with $0\le u<\eta$ and $2\nmid s$. Under these conditions, five distinct cases arise:
\begin{description}[font=\normalfont\scshape, leftmargin=0cm]

\item[\underline{$a=1$}] In this case $k=0$ and $(\gamma,\mathfrak{f}_0)\in\{(z_1,2^{r-\eta}z_2) : 0\le z_1,z_2<2^{\eta}\}$. Moreover,
\begin{equation}\label{k=0, a=1, caso par r>eta+1}
h\blackdiamond y = y\quad\text{and}\quad f_{f_0}(h,h') = \mathfrak{f}_0 hh'.
\end{equation}

\item[\underline{$a=-1$}] In this case $k=0$ and $(\gamma,\mathfrak{f}_0)\in\{(2^{r-1}z_1-2^{\eta-1}z_2,z_2):0\le z_1,z_2<2\}$. Moreover,
\begin{equation}\label{a=-1, caso par r>eta+1}
h\blackdiamond y = (-1)^h y \quad\text{and}\quad f_{\mathfrak{f}_0}(h,h') = \begin{cases} \mathfrak{f}_0h'  & \text{if $h$ is odd,}\\ 0 & \text{if $h$ is even.}\end{cases}
\end{equation}

\item[\underline{$a=1+2^{r-\eta}k$ with $k\ne 0$ and $r+u < 2\eta$}] Then
\begin{equation}\label{a=1+2algo, caso par r+u le 2eta}
(\gamma,\mathfrak{f}_0)\in\{(2^{\eta-u}(z_1+z_2),2^{r-\eta}sz_2):0\le z_1<2^{r+u-\eta}\text{ and } 0\le z_2<2^u\}.
\end{equation}
Moreover,
\begin{equation}\label{a=1+2algo, caso par r+u le 2eta'}
h \blackdiamond y =\sum_{i=0}^h \binom{h}{i} 2^{(r-\eta)i}k^iy\quad\text{and}\quad f_{\mathfrak{f}_0}(h,h') =  \sum_{\ell=0}^{h-1} \binom{h}{\ell+1}  2^{(r-\eta)\ell} k^{\ell}\mathfrak{f}_0h'.
\end{equation}

\item[\underline{$a=1+2^{r-\eta}k$ with $k\ne 0$ and $r+u\ge 2\eta$}] Then
\begin{equation}\label{a=1+2algo, caso par r>eta+1}
(\gamma,\mathfrak{f}_0)\in\{(2^{\eta-u}z_1,2^{r-\eta} (sz_1+z_2)): 0\le z_1<2^u \text{ and } 0\le z_2<2^{\eta}\}.
\end{equation}
Moreover,
\begin{equation}\label{a=1+2algo, caso par r>eta+1'}
h \blackdiamond y = (1+2^{r-\eta}kh)y\quad\text{and}\quad f_{\mathfrak{f}_0}(h,h') = \mathfrak{f}_0hh'.
\end{equation}

\item[\underline{$a=-1-2^{r-\eta}k$ with $k\ne 0$}] Then
\begin{equation}\label{a=-1-2algo, caso par r>eta+1}
(\gamma,\mathfrak{f}_0)\in \begin{cases} \{(2^{r-1}z_1-2^{\eta-1}z_2,z_2):0\le z_1,z_2<2\} & \text{if $k$ is even,}\\ \{(2^{r-1}z_1,0):0\le z_1<2\} & \text{if $k$ is odd.}\end{cases}
\end{equation}
Moreover,
\begin{equation}\label{a=-1-2algo, caso par r>eta+1'}
\quad\qquad h\blackdiamond y =(-1)^h\sum_{i=0}^h \binom{h}{i} 2^{(r-\eta)i}k^iy\quad\text{and}\quad f_{\mathfrak{f}_0}(h,h') = \sum_{\ell=0}^{h-1}  A_{h\ell} 2^{(r-\eta)\ell}k^{\ell}\mathfrak{f}_0h'.
\end{equation}
\end{description}
\end{proposition}

\begin{proof} Note that $r>\eta+1$ implies $r\ge 3$. Let $A$ be as in Example~\ref{ejemplo}. By Lemma~\ref{para ejemploscon2} and Remark~\ref{A para ejemploscon2}, since~$A$ is the multiplication by an invertible element $a\in \mathds{Z}_{2^r}$ and $A^{p^{\eta}} = \Ide$, we have $a = \pm (1+2^{r-\eta}k)$, with $0\le k<2^{\eta}$. Since~$1^{\times h}=h$, by this and~\eqref{acciones caso trivial}, we have:
\begin{equation}\label{calculo diamante p=2}
h \blackdiamond y =\begin{cases} (1+2^{r-\eta}k)^h y = \sum_{i=0}^h \binom{h}{i} 2^{(r-\eta)i}k^{i} y & \text{if $a = 1+2^{r-\eta}k$,}\\ (-1)^h(1+2^{r-\eta}k)^h y = (-1)^h\sum_{i=0}^h \binom{h}{i} 2^{(r-\eta)i}k^{i}y & \text{if $a = -1-2^{r-\eta}k$;}\end{cases}
\end{equation}
while, from~\eqref{zazaza1}, we obtain
\begin{equation*}
f_{\mathfrak{f}_0}(h,h') = h'\sum_{j=0}^{h-1}a^j \mathfrak{f}_0,
\end{equation*}
Assume that $a=1+2^{r-\eta}k$. Then $a^j = \sum_{l=0}^j \binom{j}{l}2^{(r-\eta)l}k^l$, and so
\begin{equation}\label{cociclo1caso1}
f_{\mathfrak{f}_0}(h,h') = \sum_{j=0}^{h-1} \sum_{\ell=0}^j \binom{j}{\ell} 2^{(r-\eta)\ell}k^l\mathfrak{f}_0h' = \sum_{\ell=0}^{h-1} \sum_{j=\ell}^{h-1} \binom{j}{\ell} 2^{(r-\eta)\ell}k^{\ell}\mathfrak{f}_0h'= \sum_{\ell=0}^{h-1} \binom{h}{\ell+1} 2^{(r-\eta)\ell}k^{\ell}\mathfrak{f}_0h'.
\end{equation}
Assume now that $a = -1-2^{r-\eta}k$. Then $a^j = \sum_{l=0}^j (-1)^j\binom{j}{l}2^{(r-\eta)l}k^l$, and so
\begin{equation}\label{cociclo1caso2}
f_{\mathfrak{f}_0}(h,h') = \sum_{j=0}^{h-1} (-1)^j\sum_{\ell=0}^j \binom{j}{\ell} 2^{(r-\eta)\ell}k^{\ell}\mathfrak{f}_0h'= \sum_{\ell=0}^{h-1} \sum_{j=\ell}^{h-1} (-1)^j\binom{j}{\ell}  2^{(r-\eta)\ell} k^l\mathfrak{f}_0h'= \sum_{\ell=0}^{h-1} A_{h\ell} 2^{(r-\eta)\ell}k^{\ell} \mathfrak{f}_0h'.
\end{equation}
We next compute the cohomology. Assume first that $a=1$. Then, we are in the case~b) of Example~\ref{ejemplo}, with $c=0$ and we have $N(A) = 2^{\eta}$. So,
\begin{equation*}
\ho_{\blackdiamond}^2(H,I)\simeq \frac{\ker T_1\cap \ker T_2}{\ima S} = \frac{\ker(cN(A))}{\ima(2^{\eta}\Ide)}\oplus \ker(2^{\eta}\Ide) = \frac{\mathds{Z}_{2^r}}{2^{\eta}\mathds{Z}_{2^r}}\oplus 2^{r-\eta}\mathds{Z}_{2^r}.
\end{equation*}
The fact that $(\gamma,\mathfrak{f}_0)\in\{(z_1,2^{r-\eta}z_2) : 0\le z_1,z_2<2^{\eta}\}$ follows from this; the first assertion in~\eqref{k=0, a=1, caso par r>eta+1}, from~\eqref{calculo diamante p=2}; and, the second one, from~\eqref{cociclo1caso1}. Assume now that $a=-1$. Then, we are in the case~a) of Example~\ref{ejemplo}, with $c=-2^{\eta-1}$ and we have $N(A) = 0$. Hence,
\begin{equation*}
\ho_{\blackdiamond}^2(H,I) \simeq \frac{\ker \wt{T}_1\cap \ker \wt{T}_2}{\ima \wt{S}} = \ker(A-\Ide)\oplus \frac{\ker(N(A))}{(A-\Ide)I} = 2^{r-1} \mathds{Z}_{2^r} \oplus \frac{\mathds{Z}_{2^r}}{2\mathds{Z}_{2^r}}.
\end{equation*}
The fact that $(\gamma,\mathfrak{f}_0)\in\{(2^{r-1}z_1-2^{\eta-1}z_2,z_2):0\le z_1,z_2<2\}$ follows from this applying the map $\ov{\Phi}$ of~\eqref{Phi barra1}; the first assertion in~\eqref{a=-1, caso par r>eta+1}, from formula~\eqref{calculo diamante p=2}; and, the second one, from formula~\eqref{cociclo1caso2}. Assume now that $a=1+2^{r-\eta}k$, with $k\ne 0$ and $r+u\le 2\eta$. Then, we are in the case~a) of Ex\-ample~\ref{ejemplo}, with $c=2^{2\eta-r-u}s'$, where $0\le s'<2^{r-\eta}$ and $s's\equiv 1\pmod{2^{r-\eta}}$. In fact, a direct~com\-putation shows that $c(A-\Ide) = 2^{2\eta-r-u}s' 2^{r+u-\eta}s = 2^{\eta}ss' = 1$. Computing, we obtain
\begin{equation*}
N(A) = \sum_{j=0}^{2^{\eta}-1} (1+2^{r-\eta}k)^j = \sum_{j=0}^{2^{\eta}-1}\sum_{\ell=0}^j \binom{j}{\ell} 2^{(r-\eta)\ell}k^{\ell} = \sum_{\ell=0}^{2^{\eta}-1} \sum_{j=\ell}^{2^{\eta}-1} \binom{j}{\ell} 2^{(r-\eta)\ell}k^{\ell} = \sum_{\ell=0}^{2^{\eta}-1}\binom{2^{\eta}}{\ell+1} 2^{(r+u-\eta)\ell}s^{\ell}.
\end{equation*}
Write $\ell+1 = 2^tv$ with $2\nmid v$ and $t\ge 0$. Since $2^{\eta-t}\mid \binom{2^{\eta}}{\ell+1}$, the exponent $w$, of $2$ in $\binom{2^{\eta}}{\ell+1} 2^{(r+u-\eta)\ell}$, is greater than or equal to $\eta-t+(r+u-\eta)\ell = r+u-t+(r+u-\eta)(\ell-1)$. Hence, if $\ell\ge 2$, then
\begin{equation*}
w \ge r+u-t+(r+u-\eta)(\ell-1) \le r+u-t+(u+1)(\ell-1) \le  r+u+u(\ell-1)\le r,
\end{equation*}
where we have used that $r>\eta$, and that if $\ell\ge 2$, then $t<\ell$. Consequently,
$$
N(A) = 2^{\eta} + 2^{\eta-1}(2^{\eta}-1)2^{r+u-\eta}s = 2^{\eta}(1 - 2^{r+u-\eta-1}s),
$$
and so
\begin{equation*}
\ho_{\blackdiamond}^2(H,I)\simeq \frac{\ker \wt{T}_1\cap \ker \wt{T}_2}{\ima \wt{S}} = \ker(A-\Ide)\oplus \frac{\ker(N(A))}{(A-\Ide)I} = 2^{\eta-u}\mathds{Z}_{2^r}\oplus\frac{2^{r-\eta}\mathds{Z}_{2^r}}{2^{r+u-\eta}\mathds{Z}_{2^r}}.
\end{equation*}
Applying the map~$\ov{\Phi}$ of~\eqref{Phi barra1} we obtain~\eqref{a=1+2algo, caso par r+u le 2eta}; while~\eqref{a=1+2algo, caso par r+u le 2eta'} follows from formulas~\eqref{calculo diamante p=2} and~\eqref{cociclo1caso1}. Assume now that $a=1+2^{r-\eta}k$, with $k\ne 0$ and $r+u>2\eta$. Then, we are in the case~b) of Ex\-ample~\ref{ejemplo}, with $c=2^{r+u-2\eta}s$, and we have
$$
N(A) = \sum_{\ell=0}^{2^{\eta}-1}\binom{2^{\eta}}{\ell+1} 2^{(r+u-\eta)\ell}s^{\ell}= 2^{\eta}(1-2^{r+u-\eta-1}s),
$$
where the last equality follows arguing as above. So,
\begin{equation*}
\ho_{\blackdiamond}^2(H,I)\simeq \frac{\ker \wt{T}_1\cap \ker \wt{T}_2}{\ima S} = \frac{\ker(cN(A))}{\ima(2^{\eta}\Ide)}\oplus \ker(2^{\eta}\Ide) = \frac{2^{\eta-u}\mathds{Z}_{2^r}}{2^{\eta}\mathds{Z}_{2^r}}\oplus 2^{r-\eta}\mathds{Z}_{2^r}.
\end{equation*}
Applying the map~$\ov{\Phi}$ of \eqref{Phi barra}, we obtain~\eqref{a=1+2algo, caso par r>eta+1}; while~\eqref{a=1+2algo, caso par r>eta+1'} follows from~\eqref{calculo diamante p=2} and~\eqref{cociclo1caso1}, because $2r-2\eta+u\ge r$. Assume finally that $a=-1-2^{r-\eta}k$. Then, we are in the case~a) of Ex\-ample~{ejemplo}, with $c=-2^{\eta-1}(1+2^{r-\eta-1}k)$. In fact
$$
c(A-\Ide) = -2^{\eta-1}(1+2^{r-\eta-1}k)(-2-2^{r-\eta}k) = 2^{\eta}(1+2^{r-\eta-1}k)^2 = 2^{\eta},
$$
where the last equality holds since  $(1+2^{r-\eta-1}k)^2\equiv 1\pmod{2^{r-\eta}}$. Computing, we obtain
\begin{equation*}
N(A) = \sum_{j=0}^{2^{\eta}-1} (-1-2^{r-\eta}k)^j = \sum_{j=0}^{2^{\eta}-1}(-1)^j\sum_{\ell=0}^j \binom{j}{\ell} 2^{(r-\eta)\ell}k^{\ell} = \sum_{\ell=0}^{2^{\eta}-1} \sum_{j=\ell}^{2^{\eta}-1} (-1)^j\binom{j}{\ell} 2^{(r-\eta)\ell}k^{\ell} = \sum_{\ell=0}^{2^{\eta}-1} A_{2^{\eta}\ell} 2^{(r+u-\eta)\ell} s^{\ell}.
\end{equation*}
But, as we saw at the end of the proof of Proposition~\ref{prop 3.12},
$$
A_{2^{\eta}0}=0,\qquad A_{2^{\eta}1}=-2^{\eta-1}\qquad\text{and}\qquad 2^{2\ell-\eta-2} A_{2^{\eta}\ell}\in \mathds{Z}\quad\text{for $2\le \ell<2^{\eta}$.}
$$
Since $(r+u-\eta)(\ell-1)+u-\eta\ge 2\ell-\eta-2$ (because $r>\eta+1$), we have
$$
2^{(r+u-\eta)\ell} A_{2^{\eta}\ell} = 2^r2^{(r+u-\eta)(\ell-1)+u-\eta} A_{2^{\eta}\ell}\in 2^r \mathds{Z}\qquad\text{for $2\le \ell < 2^{\eta}$,}
$$
and so $N(A) = -2^{r+u-1}s$. Consequently,
\begin{equation*}
\ho_{\blackdiamond}^2(H,I)\simeq \frac{\ker \wt{T}_1\cap \ker \wt{T}_2}{\ima \wt{S}} = \ker(A-\Ide)\oplus \frac{\ker(N(A))}{(A-\Ide)I} = \begin{cases} 2^{r-1}\mathds{Z}_{2^r}\oplus\frac{\mathds{Z}_{2^r}}{2\mathds{Z}_{2^r}} &\text{if $k$ is even,}\\ 2^{r-1}\mathds{Z}_{2^r}\oplus 0 &\text{if $k$ is odd.} \end{cases}
\end{equation*}
Applying the map $\ov{\Phi}$ of~\eqref{Phi barra1}, we obtain~\eqref{a=-1-2algo, caso par r>eta+1}; while~\eqref{a=-1-2algo, caso par r>eta+1'} follows from~\eqref{calculo diamante p=2} and~\eqref{cociclo1caso2}.
\end{proof}

\begin{remark} When $a=-1-2^{r-\eta}k$, with $k\ne 0$ and $r+u\ge 2\eta$, the formulas for $\blackdiamond$ and $f_{\mathfrak{f}_0}$ simplify to
$$
h \blackdiamond y = (-1)^h (1+2^{r-\eta}k)y\quad\text{and}\quad f_{\mathfrak{f}_0} = \begin{cases} - 2^{r-\eta-1} k \mathfrak{f}_0 hh' & \text{if $h$ is even,}\\\phantom{-} \mathfrak{f}_0h' + 2^{r-\eta-1} k \mathfrak{f}_0 (h-1)h' & \text{if $h$ is odd.} \end{cases}
$$
\end{remark}


\begin{bibdiv}
\begin{biblist}

\bib{B}{article}{
	title={Extensions, matched products, and simple braces},
	author={Bachiller, David},
	journal={Journal of Pure and Applied Algebra},
	volume={222},
	number={7},
	pages={1670--1691},
	year={2018},
	publisher={Elsevier}
}		

\bib{B2}{article}{
   author={Bachiller, David},
   title={Classification of braces of order $p^3$},
   journal={J. Pure Appl. Algebra},
   volume={219},
   date={2015},
   number={8},
   pages={3568--3603},
   issn={0022-4049},
   review={\MR{3320237}},
   doi={10.1016/j.jpaa.2014.12.013},
}

\bib{CR}{article}{
  title={Regular subgroups of the affine group and radical circle algebras},
  author={Catino, Francesco},
  author={Rizzo, Roberto},
  journal={Bulletin of the Australian Mathematical Society},
  volume={79},
  number={1},
  pages={103--107},
  year={2009},
  publisher={Cambridge University Press}
}

\bib{CJO}{article}{
	author={Ced{\'o}, Ferran},
	author={Jespers, Eric},
	author={Okni{\'n}ski, Jan},
	title={Retractability of set theoretic solutions of the Yang--Baxter equation},
	journal={Advances in Mathematics},
	volume={224},
	number={6},
	pages={2472--2484},
	year={2010},
	publisher={Elsevier}
}

\bib{CJO1}{article}{
  title={Braces and the Yang--Baxter equation},
  author={Ced{\'o}, Ferran},
  author={Jespers, Eric},
  author={Okni{\'n}ski, Jan},
  journal={Communications in Mathematical Physics},
  volume={327},
  number={1},
  pages={101--116},
  year={2014},
  publisher={Springer}
}

\bib{CJR}{article}{
title={Involutive Yang-Baxter groups},
  author={Ced{\'o}, Ferran},
  author={Jespers, Eric},
  author={Del Rio, Angel},
  journal={Transactions of the American Mathematical Society},
  volume={362},
  number={5},
  pages={2541--2558},
  year={2010}
}

\bib{Ch}{article}{
	title={Fixed-point free endomorphisms and Hopf Galois structures},
	author={Childs, Lindsay},
	journal={Proceedings of the American Mathematical Society},
	volume={141},
	number={4},
	pages={1255--1265},
	year={2013},
	review={\MR{3008873}} 	
}

\bib{De1}{article}{
	title={Set-theoretic solutions of the Yang--Baxter equation, RC-calculus, and Garside germs},
	author={Dehornoy, Patrick},
	journal={Advances in Mathematics},
	volume={282},
	pages={93--127},
	year={2015},
	publisher={Elsevier},
	review={\MR{3374524}}
}

\bib{DDM}{article}{
	title={Garside families and Garside germs},
	author={Dehornoy, Patrick},
	author={Digne, Fran{\c{c}}ois},
	author={Michel, Jean},
	journal={Journal of Algebra},
	volume={380},
	pages={109--145},
	year={2013},
	publisher={Elsevier},
	review={\MR{3023229}}
}		

\bib{Di}{article}{
   author={Dietzel, Carsten},
   title={Braces of order $p^2q$},
   journal={J. Algebra Appl.},
   volume={20},
   date={2021},
   number={8},
   pages={Paper No. 2150140, 24},
   issn={0219-4988},
   review={\MR{4297324}},
   doi={10.1142/S0219498821501401},
}

\bib{Do}{article}{
   author={Pulji\'c, Dora},
   title={Classification of braces of cardinality $p^4$},
   journal={J. Algebra},
   volume={660},
   date={2024},
   pages={1--33},
   issn={0021-8693},
   review={\MR{4778842}},
   doi={10.1016/j.jalgebra.2024.07.007},
}

\bib{Dr}{article}{
   author={Drinfeld, Vladimir G.},
   title={On some unsolved problems in quantum group theory},
   conference={
      title={Quantum groups},
      address={Leningrad},
      date={1990},
   },
   book={
      series={Lecture Notes in Math.},
      volume={1510},
      publisher={Springer, Berlin},
   },
   date={1992},
   pages={1--8},
   doi={10.1007/BFb0101175},
}

\bib{ESS}{article}{
	author={Etingof, Pavel},
	author={Schedler, Travis},
	author={Soloviev, Alexandre},
	title={Set-theoretical solutions to the quantum Yang-Baxter equation},
	journal={Duke Math. J.},
	volume={100},
	date={1999},
	number={2},
	pages={169--209},
	issn={0012-7094},
	review={\MR{1722951}},
	doi={10.1215/S0012-7094-99-10007-X},
}

\bib{GI2}{article}{
	title={Set-theoretic solutions of the Yang--Baxter equation, braces and symmetric groups},
	author={Gateva-Ivanova, Tatiana},
	journal={Advances in Mathematics},
	volume={338},
	pages={649--701},
	year={2018},
	publisher={Elsevier},
	review={\MR{3861714}}
}

\bib{GI4}{article}{
	title={Quadratic algebras, Yang--Baxter equation, and Artin--Schelter regularity},
	author={Gateva-Ivanova, Tatiana},
	journal={Advances in Mathematics},
	volume={230},
	number={4-6},
	pages={2152--2175},
	year={2012},
	publisher={Elsevier},
	review={\MR{2927367}}
}

\bib{GIVB}{article}{
	title={Semigroups ofI-Type},
	author={Gateva-Ivanova, Tatiana},
	author={Van den Bergh, Michel},
	journal={Journal of Algebra},
	volume={206},
	number={1},
	pages={97--112},
	year={1998},
	publisher={Elsevier},
	review={\MR{1637256}}
}

\bib{GGV1}{article}{
   author={Guccione, Jorge A.},
   author={Guccione, Juan J.},
   author={Valqui, Christian},
   title={Extensions of linear cycle sets and cohomology},
   journal={Eur. J. Math.},
   volume={9},
   date={2023},
   number={1},
   pages={Paper No. 15, 29},
   issn={2199-675X},
   review={\MR{4551665}},
   doi={10.1007/s40879-023-00592-6},
}

\bib{GGV2}{article}{
  author={Guccione, Jorge A.},
  author={Guccione, Juan J.},
  author={Valqui, Christian},
  title={Extensions of a family of Linear Cycle Sets},
  journal={arXiv preprint arXiv:2501.08357},
  year={2025}
}

\bib{JO}{article}{
	title={Monoids and groups of I-type},
	author={Jespers, Eric},
	author={Okni{\'n}ski, Jan},
	journal={Algebras and representation theory},
	volume={8},
	number={5},
	pages={709--729},
	year={2005},
	publisher={Springer},
	review={\MR{2189580}}	
}

\bib{LV}{article}{
	author={Lebed, Victoria},
	author={Vendramin, Leandro},
	title={Cohomology and extensions of braces},
	journal={Pacific Journal of Mathematics},
	volume={284},
	number={1},
	pages={191--212},
	year={2016},
	publisher={Mathematical Sciences Publishers}
}

\bib{Mu}{article}{
   author={Mukherjee, Snehashis},
   title={Classification of right nilpotent $\mathbb{F_p}$-braces of cardinality
   $p^5$},
   journal={J. Algebra},
   volume={665},
   date={2025},
   pages={503--537},
   issn={0021-8693},
   review={\MR{4832032}},
   doi={10.1016/j.jalgebra.2024.11.009},
}

\bib{R2}{article}{
	title={Braces, radical rings, and the quantum Yang--Baxter equation},
	author={Rump, Wolfgang},
	journal={Journal of Algebra},
	volume={307},
	number={1},
	pages={153--170},
	year={2007},
	publisher={Elsevier}
}

\bib{R3}{article}{
	title={The brace of a classical group},
	author={Rump, Wolfgang},
	journal={Note di Matematica},
	volume={34},
	number={1},
	pages={115--145},
	year={2014},
}

\bib{S}{article}{
   author={Soloviev, Alexander},
   title={Non-unitary set-theoretical solutions to the quantum Yang-Baxter
   equation},
   journal={Math. Res. Lett.},
   volume={7},
   date={2000},
   number={5-6},
   pages={577--596},
   doi={10.4310/MRL.2000.v7.n5.a4},
}

\bib{Spivey}{book}{
   author={Spivey, Michael Z.},
   title={The art of proving binomial identities},
   series={Discrete Mathematics and its Applications (Boca Raton)},
   publisher={CRC Press, Boca Raton, FL},
   date={2019},
   pages={xiv+368},
   isbn={978-0-8153-7942-3},
   review={\MR{3931743}},
   doi={10.1201/9781351215824},
}

\end{biblist}
\end{bibdiv}
\end{document}